\documentclass[reqno]{amsart}
\usepackage{amsmath, amsfonts, amssymb, amsthm, esint, mathtools}
\usepackage[margin=1in]{geometry}
\usepackage{graphicx,xcolor}

\usepackage{cite}
\usepackage[colorlinks,linkcolor = blue,citecolor = green]{hyperref}

\newtheorem{theorem}{Theorem}[section]
\newtheorem{proposition}[theorem]{Proposition}

\newtheorem{lemma}[theorem]{Lemma}
\newtheorem{defn}[theorem]{Definition}

\newtheorem{remark}[theorem]{Remark}
\numberwithin{equation}{section}

\newcommand{\R}{\mathbb R}

\DeclareMathOperator{\dive}{div}

\DeclareMathOperator{\tr}{\mathrm\mathrm{tr}}

\newcommand{\eps}{\varepsilon}

\usepackage{comment, color}

\usepackage{dsfont}
\DeclareMathOperator{\dv}{\mathrm{div}}

 \def \E{\mathbb{E}}

\author{Nestor Guillen}
\address{Courant Institute, New York University, New York, NY 10012}
\email{nestor.g@nyu.edu}
\author{Vladimir A. Kobzar}
\address{Department of Mathematics, The Ohio State University, Columbus, OH 43210}
\email{kobzar.1@osu.edu}

\title{Online  Koml\'os converges to mean curvature flow}
\date{\today}
\dedicatory{Dedicated to Robert V. Kohn, 1953-2026}

\begin{document}

\begin{abstract}

We determine the asymptotics of a game inspired by classic vector balancing problems in combinatorial discrepancy theory.  In this game, which we call  the online Koml\'os game, two players, Paul and Carol, update the state vector $y$ in  $\mathbb{R}^m$, initially placed at $0$. At each round, Paul chooses freely a set of $n$ vectors in the Euclidean unit ball, and   Carol chooses, for each such vector, whether to leave  it unchanged or reverse its sign. The resulting vectors are all added to  $y$, and the game proceeds to a new round.  After $T$ rounds, the game ends, and the $\ell_\infty$ norm of the state vector $y$ is determined. Paul's objective throughout the game is  to maximize this norm, and  Carol's objective is to minimize it. As $T$ gets large, we establish that the leading order term of the value of this game is $\sqrt{T/2\tau}$, where $\tau$ is the extinction time of the unit cube in $\mathbb{R}^m$ under a curvature-based flow characterized by the values of $m$ and $n$. When $n\geq m-1$, this flow is the mean curvature flow, and we show that $1/\sqrt{2\tau} =\Theta(\sqrt{\log m})$.   Our results build upon the work of Kohn and Serfaty on deterministic games and mean curvature flow, combined with Banaszczyk's $\ell^2$ analogue of the Beck-Fiala theorem. As the large $T$ limit of the online Koml\'os game amounts to a localization of the  classic Koml\'os problem, we hope this work can shed light on this and other vector balancing problems. Our results generalize to  the version of the online  Koml\'os game with the final value given by an arbitrary norm in $\R^m$.

\end{abstract}

\maketitle

\section{Introduction}\label{sec:intro}

Discrepancy theory covers a wide array of problems and methods in combinatorics, geometry and analysis  that involve approximation of a mathematical object by discrete elements  \cite{Matousek99,Travaglini14}. This field has been a long-standing focus of computer science and  scientific computing, ranging from  differential privacy and analog-to-digital conversion  to causal inference and  learning theory \cite{Chazelle2000,   CST14, Gunturk12, NTZ13, HSS24,APZ19}.  One of the main problems in \emph{combinatorial} discrepancy is vector balancing: given vectors $a_1,\ldots,a_k$ in a normed space,  our objective is to choose signs $\varepsilon_1,\ldots,\varepsilon_k \in \{-1,1\}$ in order to approximate the zero vector, i.e., minimize
\begin{align*}
  \| \varepsilon_1a_1+\ldots+\varepsilon_k a_k\|.
\end{align*}
We will focus on a version of this problem  posed as a zero-sum game between two players, Paul and Carol.\footnote{These names are mnemonics for \emph{pusher} of vectors and  \emph{chooser} of  signs, respectively, in the games introduced by Spencer.} The game is played for $T$ rounds where Paul reveals a set of $k=n T$ vectors  in  $\mathbb{R}^m$ for $m \geq 2$ to Carol not all at once, but sequentially in $T$ batches of size $n$.

Since the case when $T=1$  corresponds to the classic open problem  known as the Koml\'os conjecture, we will refer to our game as the \emph{online Koml\'os game}, formally defined as follows. At round $t=0$, we initialize the vector $y_t \in \mathbb{R}^m$ with $y_0 = 0$. In each subsequent round from $t= 1$ until the final round $T$, Paul chooses $n$ vectors $a_{1,t},\ldots,a_{n,t} \in \mathbb{R}^m$, each with at most unit Euclidean length, and reveals them to Carol. Then Carol irrevocably chooses $n$ signs $\bar \varepsilon_t = (\varepsilon_{1,t},\ldots,\varepsilon_{n,t}) \in \{-1,1\}^n$, and we  update
\begin{align*}
  y_{t} = y_{t-1}+(\varepsilon_{1,t}a_{1,t}+\ldots+ \varepsilon_{n,t}a_{n,t}).
\end{align*}
We will write this update  more compactly as $y_{t} = y_{t-1} + A_t\bar \varepsilon_t$ where $A_t$ denotes the $m \times n$ matrix whose columns are given by the vectors $a_{1,t},\ldots,a_{n,t} $. Paul's objective is to make the norm $\|y_T\|_\infty$ as large as possible, while Carol's objective is to make it as small as possible.  

Paul is referred to as an \emph{adaptive} adversary because he can choose each $A_t$ based on Carol's choices of $\bar \eps_\tau$'s at earlier times $\tau <t$. Accordingly, the  value $K_{T}(m,n)$ of the online Koml\'os  game is given by 
\begin{align}\label{e:online Komlos}
  K_{T}(m,n) := \max_{A_1}\min_{\bar \varepsilon_1} \ldots \max_{A_T}\min_{\bar \varepsilon_T}\|A_1\bar \varepsilon_1+\ldots+A_T\bar \varepsilon_T\|_\infty
\end{align}
where the feasible set of each $A_t$ is given by all $m \times n$ matrices with columns of at most unit Euclidean length  and each $\bar \eps_t \in \{\pm 1 \}^n$.  

Kohn and Serfaty \cite{KohnSerfaty2006} established a connection between a version of this game and flow by mean curvature (see also \cite{kohn07}). Concretely, they found that their Paul and Carol game that is closest to ours (which they refer to as the time-dependent ``inverse game"), is characterized in a continuum limit by the following nonlinear PDE
\begin{align*}
  \partial_tu + \max\big \{\Delta u-\big(D^2u \tfrac{\nabla u}{|\nabla u|},\tfrac{\nabla u}{|\nabla u|}\big),0\big \} = 0.
\end{align*}
They also consider the continuum limit of several other games and arrive at various other PDE, such as 
\begin{align}\label{e:motion by mean curvature}
  \partial_tu + \Delta u-\big(D^2u \tfrac{\nabla u}{|\nabla u|},\tfrac{\nabla u}{|\nabla u|}\big) = 0.
\end{align}
They obtain this last PDE using a game where Paul is restricted to choose batches forming orthonormal frames of $m-1$ vectors. The equation \eqref{e:motion by mean curvature} is known as the level set formulation for mean curvature flow (MCF), since the level sets of the solution $u(\cdot,t)$ move with velocity equal to their mean curvature.

In this work we obtain a continuum limit for the online Koml\'os game, whose value is given by \eqref{e:online Komlos}. Namely, we show an analogue of the result of Kohn and Serfaty holds in our setting, when the only condition on the vectors $a_{i,t}$ is that they all lie in the Euclidean unit ball. Accordingly, our  first result shows that the leading order term of $K_{T}(m,n)$ as $T\to \infty$ is given in terms of  the extinction time of the cube $[-1,1]^m$ under mean curvature flow.  (In the interest of exposition we state here only special cases of the main theorems and defer the complete statements along with relevant definitions to Section \ref{s:main results}.) 

\begin{theorem}\label{thm:main}
If $m \geq 2$ and $n\geq m-1$, 
\begin{align*}
\lim_{T\to\infty} \frac{K_{T}(m,n)}{\sqrt{T}}  =  \frac{1}{\sqrt{2\tau}}, 
\end{align*}
where $\tau$ is the extinction time of $[-1,1]^m$ under mean curvature flow. 

\end{theorem}

The complete result is recorded later in Theorem \ref{t:main theorem}; it covers asymptotics for online vector balancing in any norm and  batches of any size $n$, including $n=1$. This result opens the door to studying the asymptotics of the value function $K_{T}(m,n)$ using tools from PDEs and geometry. Our second result characterizes how the long-time asymptotics in Theorem \ref{thm:main} grow with $m$.

\begin{theorem}\label{thm:asymptotics as m goes to infinity} The following holds as $m,n\to \infty$ with $n\geq m-1$
\begin{align*}
 (1-o(1)) \sqrt{\log m}\leq \lim_{T\to\infty} \frac{K_{T}(m,n)}{ \sqrt{T}} \leq (\sqrt{2}+o(1))\sqrt{\log m}.
\end{align*}
Equivalently, for $\tau$ the extinction time of $[-1,1]^m$ under mean curvature flow, we have 
\begin{align*}
 \frac{1-o(1)}{2\log m}\leq 2\tau \leq \frac{1+o(1)}{\log m}.
\end{align*}

\end{theorem}
These  estimates are based on a more general result covering
centrally symmetric convex bodies, see  Theorem \ref{t:extinction time asymptotics} and Remark \ref{rem:gaussian_constants}; they  are obtained by means of comparison principle arguments, relating the solution of \eqref{e:motion by mean curvature} to solutions of a linear heat equation. We believe our   bounds  are of interest independent of the online Koml\'os  game. As we discuss in Section \ref{sec:previous_work}, smooth potentials, like the hyperbolic cosine and  exponential weights,  have been used extensively to prove \emph{upper} bounds in offline and online vector balancing problems. On the other hand, our use of a potential (the heat equation solution) to  establish a \emph{lower} bound appears less common in this setting (one example is \cite{BJSS20}). A similar approach has been used to establish lower bounds on regret in online learning, see for example \cite{kobzar, kobzar_geom, kobzar22, kobzar26}.

\begin{figure}
    \centering
    \hspace{0.5in} \includegraphics[scale=0.2]{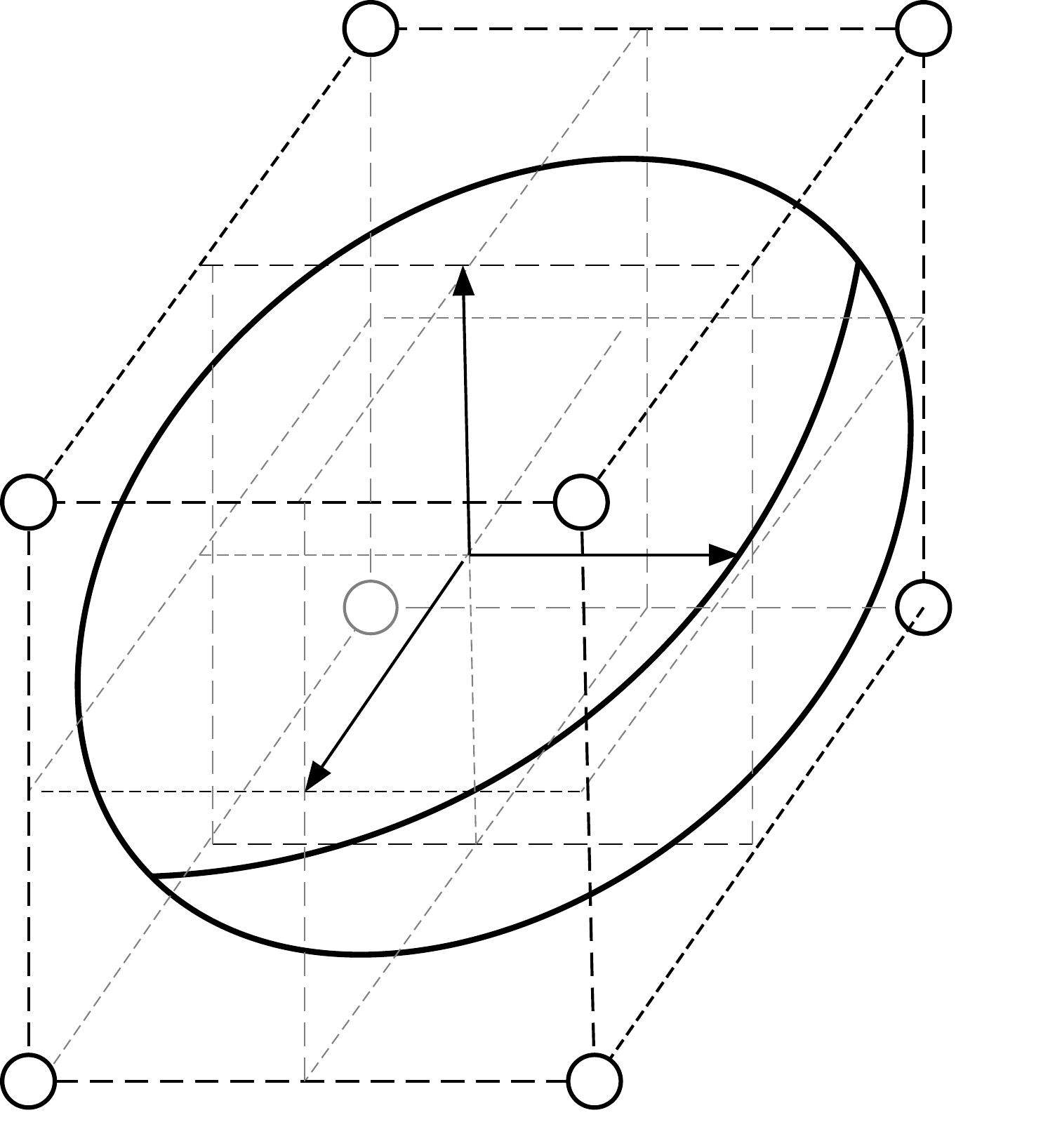}
    \caption{We have a convex body $E$ and a frame with origin placed at the body's center. If $A$ is the matrix whose columns represent the vectors in the frame, then as $\bar \varepsilon$ ranges over $\{-1,1\}^3$ we obtain $8$ vectors $A\bar \varepsilon$ whose locations are marked above by white circles. In the classic Koml\'os problem Paul aims to choose the $n$ (in this figure, $n=3$) column vectors of $A$ so that the $2^n$ vertices $\{A\overline \varepsilon\}_{\overline\varepsilon}$ of the frame it defines all lie outside some $\lambda E$, for a factor $\lambda$ to be made as large as possible within the constraints for $A$.}
    \label{fig:Classic Komlos}
\end{figure}

\subsection{The classic and online Koml\'os problems}  As noted earlier, when $T=1$, our game reduces to a classic  problem, which is subject to a long-standing conjecture attributed to J\'anos Koml\'os \cite {Spe85}: 
\begin{align*}
 \emph{The Koml\'os conjecture: }  K_1(m,n) = O(1) \text{ as }  m,n\to \infty. 
\end{align*}
 Without loss of generality, it is sufficient to resolve this conjecture when $n = m$  \cite{BECK19811},  \cite[Lecture 5]{Spe94}. In the 90's,  Banaszczyk established  estimates for a broad class of vector balancing problems \cite{Banaszczyk1993, Ban98}, obtaining an $O(\sqrt{\log m})$ upper bound for $K_{1}(m,n)$ \cite {Ban98}.  Paul can achieve the trivial lower bound of 1 on $K_{1}(m,n)$ by choosing an arbitrary orthonormal set of $m$ vectors in $\R^m$, and Kunisky  improved this lower bound  to  $1 +\sqrt{2} -o(1)$ as $m$ gets large \cite{Kun23}. In  recent work, Bansal and  Jiang  improved on the state of the art by proving  an $\tilde O(\log ^{1/4} n)$ upper bound in the  Koml\'os setting \cite{BJ2025}.\footnote{The prefactor inside $\tilde O(\cdot)$ is polynomial in $\log\log n$, specifically $(\log\log n)^c$ for some fixed exponent $c>0$ uniformly in $m$.} 

The Koml\'os conjecture is a statement about the asymptotics for $K_{T}(m,n)$ for $T=1$ as the dimension $m$ and number of vectors $n$ become large, while our Theorems \ref{thm:main} and \ref{thm:asymptotics as m goes to infinity}  cover the large $T$ regime, and consider the large $m$ and $n$ limit only after passing to the limit in $T$. It is an interesting possibility that the behavior of $u(x,t)$ solving \eqref{e:motion by mean curvature} with $u(x,T) = \|x\|_\infty$ as $t\to T$ might  encode information about the offline Koml\'os problem. We have tried without success to find such a connection. There might be a connection, or it might be that information about $K_{1}(m,n)$ is lost in the continuum limit.

Nonetheless, as random walk and potential-based iterative methods are used broadly in offline discrepancy problems, we hope  our results will provide a helpful point of view for these methods. We also hope the PDE perspective proves helpful to other types of vector balancing problems, such as online problems with oblivious adversary or random vectors sampled from a fixed distribution. These problems  are reviewed in Section \ref{sec:previous_work}. 

Determining the dimensional asymptotics of the ``large $T$ asymptotics'' is a localization of the original Koml\'os conjecture, as we now explain. It will be convenient to consider the Koml\'os problem for a general norm $\|\cdot\|_E$ 
\begin{align*}
  \mathcal{K}_{n,1}(E) := \max_{A}\min_{\overline \varepsilon}\|A\overline \varepsilon\|_E, 
\end{align*}
where $E$ stands for the unit ball for the norm (later, the case $\|\cdot\|_\infty$ will be approximated by other norms). From Paul's perspective, the classic Koml\'os problem consists (see Figure \ref{fig:Classic Komlos} above) in finding $a_1,\ldots,a_n$ so that no matter what Carol chooses for $\varepsilon_1,\ldots,\varepsilon_n$ Paul can reach, when starting from the origin, and in one step, as large as possible a multiple of $\partial E$. As such,  Paul must consider all of $E$ in choosing $a_1,\ldots,a_n$. Heuristically speaking, this is a global problem in $E$.

Contrast this with what happens in the online Koml\'os game with respect to $\|\cdot\|_E$ after a large number of steps (see  Figure \ref{f:online Komlos}), that is, when we are considering for a large $T$,
\begin{align*}
  \mathcal{K}_{n,T}(E) := \max_{A_1}\min_{\overline \varepsilon_1}\ldots\max_{A_T}\min_{\overline \varepsilon_T}\|A_1\overline \varepsilon_1+\ldots+A_T\overline \varepsilon_T\|_E.
\end{align*}
Consider a vector $y_{T-1}$ with the property that
\begin{align*}
  \mathcal{K}_{n,T}(E) = \max_{A}\min_{\overline \varepsilon}\|y_{T-1}+A\overline \varepsilon\|_E,
\end{align*}
in particular, we expect $\lambda := \|y_{T-1}\|_E$ to be large when $T$ is large (see \eqref{e:intro value function bound} later on). In this case, Paul finds $y_{T-1}$ placed on a large rescaling $\lambda \partial E$ of $\partial E$, and he is trying to reach a set $(\lambda+h) \partial E$ for as large a value of $h>0$ as possible. However, since the columns of $A$ have $\ell_2$ norm at most $1$,  the available positions for Paul which are all of the form $y_{T-1}+A\overline\varepsilon$ see only a relatively small region of $\lambda \partial E$ -- since we expect $\lambda$ to be large when $T$ is large. 

Then, if $\partial E$ is smooth, the region of $\lambda \partial E$ immediately accessible to Paul from $y_{T-1}$ (see Figure \ref{f:online Komlos}) will be approximately an ellipsoid (again, provided  $T$ and thus $\lambda$ is large).  This ellipsoid will be one passing through $y_{T-1}$ and have the same curvature at $y_{T-1}$ as $\partial E$, indicating that in this regime ($T$ large) the curvature of $\partial E$ governs how $\mathcal{K}_{n,T}(E)$ is changing with $T$. This heuristic picture is discussed further in Section \ref{s:heuristics and overview of the proof}.

\begin{figure}
    \centering
    \hspace{0.5in} \includegraphics[scale=0.25]{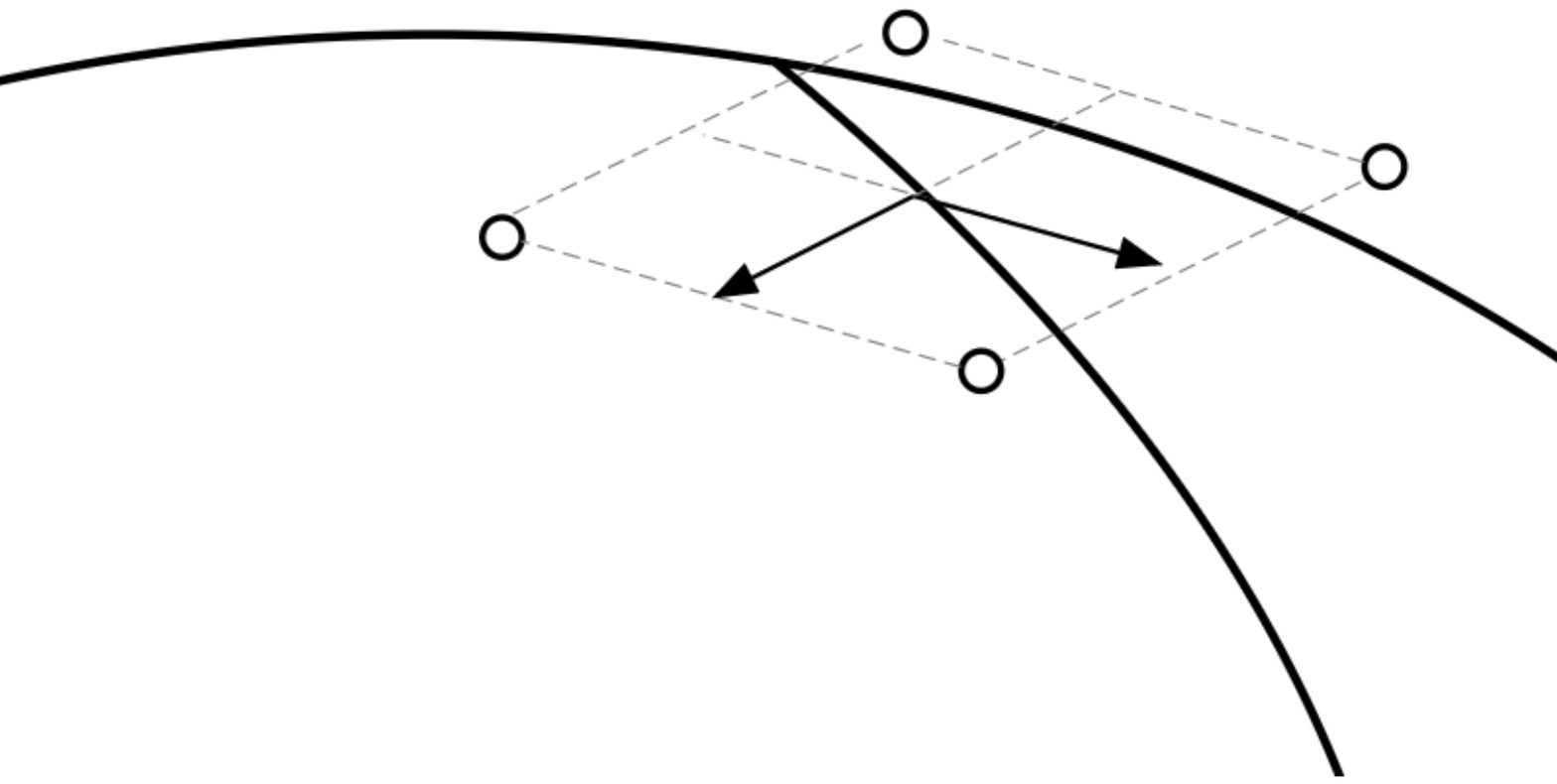}
    \caption{For the online problem ($m=3,n=2$ in this figure) we expect after a large number of steps $t$ that $y_t$ will lie on $\lambda \partial E$ for a large $\lambda$. Moreover, since the column vectors of $A$ have length at most $1$, the game is effectively zooming in to a small portion of $\partial E$. If the body $E$ is smooth, the area where Paul and Carol play is well approximated by paraboloid tangent to $\partial E$ at $\hat y_t := \|y_t\|_E^{-1}y_t$. Therefore, Paul's choice for the two vectors (and corresponding frame of $2^n$ points) is largely determined by the curvature of $\partial E$ at $\hat y_t$.}
    \label{f:online Komlos}
\end{figure}

\subsection{Notation} 
\label{sec:notation} For a vector $x$ in $\R^n$,  $\|x \|_p$ denotes its $\ell_p$ norm, and  $|x |:=\|x\|_2$ specifically denotes the Euclidean ($\ell_2$) norm. For vectors $x$ and $y$ in $\R^m$, $(x, y)$ denotes their Euclidean inner product. The vector $e_i$ represents the $i$-th canonical basis vector of $\R^m$. Given a centrally symmetric body $E \subset \mathbb{R}^m$,  $\|\cdot\|_E$ denotes the unique norm in $\mathbb{R}^m$ whose unit ball is equal to $E$. Its  dual norm $\| \cdot\|^*_E$ is given  by $\|z\|^*_E:= \sup_{\| y\|_E\leq 1} ( z, y )$.  For a matrix $M \in \R^{m  \times m}$,  its operator norm is denoted by $\|M\|$. Finally, the discrepancy of a matrix $A \in \R^{m \times n}$ is defined as 
\[
\text{disc}(A) :=\min_{\bar \eps \in \{ \pm 1\}^n }  \|A \bar \eps\|_ \infty.
\] 

\subsection{Previous work on vector balancing} 
\label{sec:previous_work} 

The taxonomy of vector balancing problems  is based on two primary characteristics. The first one specifies how the vectors are chosen. When Paul chooses vectors adversarially, as in our  game, this setting is called the \emph{worst-case} discrepancy. On the other hand, the \emph{average case}  discrepancy refers to the setting when the vectors are sampled independently from a distribution known to Carol.  The second primary characteristic specifies when the vectors are revealed to Carol: in the \emph{offline} setting Paul reveals all of them before Carol chooses any signs, while in the \emph{online} setting, Paul reveals the vectors  one-by-one and Carol irrevocably chooses the sign of each vector revealed in a given period, before the next vector is revealed \cite{Spencer77}. Our online Koml\'os  game interpolates between  the purely online setting when Paul reveals one vector at a time  and is the strongest adversary ($n=1$), and the purely offline setting when Paul reveals all vectors at once and is the weakest adversary ($T=1$).

Offline and online vector balancing problems have been studied in large part using randomized and/or potential-based iterative methods.  Since it is plausible that some of these problems may be analyzed using  PDE-based methods,  we will now review some of the previous work on vector balancing, highlighting connections to those iterative methods.  Due to space limits, we could not survey all potentially relevant recent discrepancy problems, such as Gaussian discrepancy  \cite{CGRT22}, online geometric and interval discrepancy \cite {JKS19, BJSS20}, bad science matrices \cite{steinerberger2024badsciencematrices, albors2025structurebadsciencematrices, sinha2025structureextremalbadscience}, or matrix discrepancy \cite{BKMZ24, kunisky2025asymptoticboundsonlinealgorithms,bandeira2026matrixdiscrepancyrepresentationsfinite}.
 We refer the reader to the ICM survey by Bansal for a more comprehensive and in-depth treatment of  discrepancy theory and related algorithms \cite{Bansal22}.

\textbf{Offline worst case.} One of the classic  discrepancy minimization problems entails coloring the  elements of a set $U$ of size $n$ with binary colors $\{\pm 1\}^n$  so that  each member of a given collection  of $m$ subsets of $U$ is colored in the most balanced way possible. This collection can be represented by the incidence matrix  $A \in \{0,1\}^{m\times n}$, and the balance of coloring is measured by $\text{disc}(A)$. For a closely related class of matrices: $A \in \R^{m \times n }$ and all $|A_{ij}|\leq 1$ (we will refer to this setting  as the \emph {set coloring}),  Spencer established that \[
\text{disc (A)} =O \big(\sqrt {n \log( 2m /n) }\big) 
\]by a partial coloring method using the binary entropy potential \cite {Spe85}. This result was also obtained by  Gluskin  using a different method also based on partial coloring \cite{Glushkin89}, and Giannopoulos provided a simplified version of Gluskin's proof  \cite{Giannopoulos1997}.  When $n \geq m$,  this  upper bound can be taken to be   $5.32 \sqrt {m}$,  which led to its memorable ``six standard deviations suffice" characterization. (Pesenti and Vladu \cite{PesentiVladu2023} improved the constant prefactor in the set discrepancy setting from 5.32 to 3.68.) A matching $\Omega (\sqrt{m})$ lower bound in this setting is achieved by Hadamard matrices \cite{Chazelle2000, AS2000}.

Spencer also showed that if Carol  can delete a constant fraction of vectors, i.e., set their colors to zero, then the Koml\'os conjecture holds for this modified problem\cite{Spe85}. Bandeira, Maillard and  Zhivotovskiy \cite{bandeira2022remark} provided an alternative proof of this result  based on the partial coloring theorem of Kashin \cite{Kashin85}, and yet  another earlier proof of this  result was given by Lonke \cite{Lonke98}.  On the other hand, Chasapis and  Skarmogiannis showed that if Carol is prevented from using all possible colorings, specifically if the feasible set of $\bar \eps $ is $S  \subset \{\pm 1 \}^m$ has  at most $2^{\delta n}$ colorings for $\delta \in (1/m, 1)$, then  the discrepancy  in this modified  problem is  lower bounded as $\Omega \big(\sqrt{\log(e/\delta)}\big)$ \cite{CS21}, and  Hajela established a similar result earlier \cite{hajela88}. Recently, Smirnov and  Vershynin gave a Fisher information perspective on partial coloring based on the first Dirichlet eigenvalue of the hypercube \cite{SV26}.

Bansal developed the first polynomial-time  algorithm for \emph{partial} coloring that achieves Spencer's upper bound using  a random walk with covariance determined by semidefinite programming (SDP)\cite{Bansal10}. Subsequently, Lovett and Meka developed a simpler algorithm  that achieved the same upper bound using  a random walk inside the unit hypercube in  a polytope  $\R^m$ with certain discrepancy constraints  \cite{LM15}.   Rothvoss, as well as Eldan and Singh, developed partial coloring algorithms that generalize to symmetric convex bodies beyond polytopes \cite{Rothvoss17, ES18}. Charikar, Newman, and Nikolov  established that it is NP-hard to distinguish between set systems with discrepancy zero and $\Omega(\sqrt {m})$; thus we do not expect to be able to efficiently find an optimal coloring for set systems with an $o(\sqrt {m})$ discrepancy \cite {CNN11}. On the other hand, a closely related quantity called \emph{hereditary discrepancy}  of a matrix $A$, or $\text{heredisc}(A)$, that upper bounds  $\text{disc} (A)$, can be efficiently approximated \cite{MNT20}. 

 Bansal, Dadush, and Garg  developed the first efficient algorithm based on a random walk with covariance  determined by an SDP solution that achieved Banaszczyk's upper bound in the Koml\'os    setting \cite{BDG19}. Levy, Ramadas and Rothvoss  developed a   deterministic algorithm based on the  multiplicative weights potential that achieved this upper bound; this algorithm also achieves the  $O(\sqrt {m})$ bound in the set discrepancy setting \cite{LRR17}. Subsequently, Bansal,   Dadush,  Garg and Lovett   developed an algorithm that achieves Banaszczyk's bound for vector balancing in general convex bodies; their algorithm is based on a random walk that evolves in  directions obtained by the Gram-Schmidt process \cite{BDGL19}. Harshaw, S\"avje,  Spielman, and Zhang improved the variance bounds of the signed vectors sums  obtained in this process \cite{HSS24}, and Bednorz and Godlewski  used this Gram-Schmidt walk  to improve (i) the constant prefactor in Banaszczyk's $O(\sqrt {\log m})$ upper bound in the Koml\'os setting and (ii) the discrepancy bounds in the smoothed analysis of the Koml\'os  conjecture with Gaussian noise (this setting is further discussed below)\cite{BG26}.     Lastly, Bansal, Laddha and Vempala gave a deterministic algorithm based on 
a time-dependent barrier function that achieves all the bounds described above -- Spencer's  set coloring bound $O \big(\sqrt {n \log( 2m /n) }\big) $, Banaszczyk's upper bounds in the Koml\'os   setting -- as well as certain bounds in the spectral hypergraph and Gaussian random matrix settings \cite{BLV22}.  

Bansal,  Jiang,  Meka,  Singla and Sinha  applied the smoothed analysis of Spielman and Teng \cite{ST04} to show that the Koml\'os conjecture holds asymptotically almost surely (a.a.s.) when $n= \omega(m \log m)$ and $A$ is perturbed by adding small Gaussian noise \cite {BJMSS22} (see also \cite{BHMSS22} for the smoothed analysis of the prefix discrepancy). Aigner-Horev, Hefetz and Trushkin established a similar result  when $A$ is perturbed by Rademacher noise \cite {AHT25}.

\textbf{Offline average case.} While the  Koml\'os conjecture remained unresolved, a substantial amount of work was dedicated to estimating discrepancy of special classes of random and pseudorandom matrices. For  Gaussian matrices $A \in \R^{m \times n}$ with each $A_{ij} \sim N(0,1)$ independently, Costello established that $\text{disc}(A) = \Theta (2^{-n/m}\sqrt{n})$ a.a.s. when $m$ is fixed \cite{Costello09},  and Turner, Meka and Rigollet established the same result when $\omega(1) = m =o(n)$ \cite{TMR20}.\footnote{Thus, $\text{disc}(A) = O(1)$ a.a.s. as long as $n \geq c\, m \log m$ for a some uniform constant $c$. By standard Gaussian properties, each column $A_{:i}$ has unit $\ell_2$ norm  a.a.s., and by the union bound $A/\sqrt {m}$  satisfies the hypothesis of the    Koml\'os conjecture a.a.s.} In the Gaussian matrix setting, Chandrasekaran and Vempala upper bounded  discrepancy as $O(\sqrt {n \log (m/n)})$, modulo a polylog term in $m$ and $n$, when $2n \leq m \leq 2^n$ \cite {CV14}.   Pesenti and Vladu gave an algorithm based on the Stieltjes regularizer  of the max function that  was used to prove the Koml\'os conjecture  for  certain pseudorandom matrices, as well as for random orthogonal and random Gaussian matrices \cite{PesentiVladu2023}.     

In the discrete random matrix setting, Ezra and Lovett showed that when the columns of $A\in \R^{m \times n}$ are sampled from $s$-sparse vectors in $\{0,1\}^m$ uniformly and  $m \geq n \geq s$,  $\text{heredisc}(A) = O(\sqrt {s \log s})$ and when $n \gg m^s$,  $\text{heredisc}(A) = O(1)$ with high probability \cite{EL19}. Hoberg and Rothvoss showed $\text{disc}(A) \leq 1$ with high probability when $A_{ij} \sim \text{Bernoulli}(p)$ independently provided that $n = \Omega (m^2 \log m)$ and $mp = \Omega (\log n)$ \cite{HR19}. Franks and Saks showed that $\text{disc}(A) \leq 2$ with high probability for a rather broad class of random $s$-sparse matrices when   $n = \Omega (m^3 \log^2 m)$ \cite{FS20}.  Altschuler and Niles-Weed showed $\text{disc}(A) = O(1)$ with high probability when $A_{ij} \sim \text{Bernoulli}(p)$ independently and $n \geq C m \log m$ where $p$ may be varying with $n$; they also provided a  discrepancy upper bound for a broad class of Poisson random matrices \cite{AN22}.

 \textbf{Online worst case with adaptive adversary.}    As noted earlier, in the context of online discrepancy, Paul is called an adaptive adversary whenever he is able to choose the vectors based on Carol's choices of signs at earlier times.  Spencer showed that in the online version of the set discrepancy problem, where Paul chooses $a_t$ with $\|a_t\|_\infty \leq 1$  adaptively, the  discrepancy is  $\Theta (\sqrt {T \log T})$; the upper bound is established using the $\cosh$ potential, while the lower bound is established by  analyzing the distribution of certain hitting times of a one-dimensional random walk \cite[Lecture 4]{Spe94} (see also \cite{Spencer86}).   Haghtalab,  Roughgarden and  Shetty gave an algorithm that guarantees an $\tilde O(\log^2(mT/\sigma))$ upper bound on discrepancy  in a modified version of our online Koml\'os game  with $n=1$ where the adversary selects the vectors from an adaptive $\sigma$-smooth isotropic distribution where density is upper bounded pointwise by $1/\sigma$ times the density of the uniform distribution. Their algorithm uses  the $\cosh$-based potential  introduced by  Bansal, Jiang, Meka,  Singla and  Sinha in the  online average case   discussed below  \cite {HRS24,BJMSS21}. Without smoothing our online Koml\'os  game with  $n=1$ appears trivial: by equation \eqref{e:intro value function bound}, 
 \[
 1/\sqrt{2} \leq K_{t}(m,1)/\sqrt{t} \leq 1,
 \] uniformly in $m$ (see also the discussion of Game 1 in \cite {Spencer77}). Perhaps due to this, much of the recent online vector discrepancy literature focused on  the average case discrepancy, as well as  adversaries  that are  not adaptive, the so-called \emph{oblivious} adversaries. 

\textbf{Online average case.} When the vectors are sampled from the uniform  distribution over $\{\pm 1\}^m$ independently and revealed to Carol online, Bansal and Spencer gave  a strategy based on an inverse polynomial potential that guarantees with high probability an $O (\sqrt {m})$ discrepancy, as well as another  strategy based on the $\cosh$ potential that guarantees an $O (\sqrt {m } \log m)$ prefix discrepancy  \cite {BS20} (the \emph{prefix discrepancy}  controls the $\ell_\infty$ norm of the partial sums $y_t$ uniformly across time, not  just at the final time). When the vectors  are sampled independently from an arbitrary distribution over $[-1,1]^m$ with potentially non-independent entries, Bansal, Jiang, Singla  and Sinha gave an $O(m^2( \log T + \log m))$ online prefix discrepancy bound using the $\cosh$ potential \cite {BJSS20}. Improving on this result, Bansal, Jiang, Meka, Singla  and Sinha gave an $O( \log^4 (mT))$  bound on online prefix discrepancy when vectors are sampled independently from an arbitrary distribution on the unit $\ell_2$  ball in $\R^m$    \cite{BJMSS21}. 

A special case of the Koml\'os conjecture is the \emph{Beck-Fiala conjecture}: for any $A \in \{0,1\}^{m\times n}$ with column sparsity of at most $s$,  $\text{disc}(A) = O(\sqrt {s})$\cite{BECK19811}. In this setting, Altschuler and Tikhomirov established that there exists an asymptotic gap between the discrepancy in the online and offline vector balancing of independent uniformly random $s$-sparse vectors \cite{altschuler2025thresholdonlinebalancingsparse}. Finally, when the vectors are sampled independently from the standard $m$-dimensional Gaussian    and revealed to Carol online, Fiedler, Jackson, Lacker and  Niles-Weed  established  a $\Theta (\sqrt {m})$ discrepancy estimate in expectation; their lower bound applies more generally to all random vectors with i.i.d. mean zero and variance 1 entries with finite fourth moments\cite {FJLN26}. Although their methods are grounded in mean-field stochastic control, their use of a discrete value function defined by dynamic programming is similar to ours.

 \textbf{Online worst case with oblivious adversary.} In the oblivious setting, Paul still reveals the vectors to Carol sequentially. However,  he cannot use Carol's earlier choices of signs when he selects a vector in a given round. Since  every deterministic algorithm for choosing signs has the same  worst case discrepancy against an oblivious adversary as against the corresponding adaptive adversary (whenever the oblivious adversary correctly guesses the worst case pattern of vectors in the adaptive case) Carol is permitted to use randomized policies, and the  discrepancy is evaluated with high probability with respect to Carol's randomization. 

 In the online Koml\'os setting with an oblivious adversary, Kulkarni,  Reis and Rothvoss  achieved the state-of-the-art $O \big (\sqrt {\log (\min (m,T))}\big )$ discrepancy upper bound and $\Theta \big (\sqrt {\log T})$ prefix discrepancy upper and lower bounds, modulo a term that depends on the probability of attaining this bound\cite {KJR24}. Their upper bound strategy depends on the entire series vectors $a_1, \dots a_{t-1}$ (and therefore does not appear to be Markovian). Earlier, Alweiss, Liu and Sawhney  achieved an $O(\log mT)$ prefix discrepancy bound with high probability using a "self-balancing random walk" algorithm that is Markovian - it depends on  the current time $t$ and the inner product of the  partial sum 
 $\sum_{\tau=1}^{t-1}   \eps_{\tau} a_{\tau}$ with the vector $a_t$ revealed to Carol in the current period  \cite{ALS21}.

\subsection{Mean curvature flow and the level set PDE}\label{s:level set PDE} The nonlinear PDE \eqref{e:motion by mean curvature} is known as the level set formulation of mean curvature flow, and it is one of several alternative descriptions of this geometric evolution. The mean curvature flow appears in many areas of pure and applied mathematics, receiving thorough and complementary treatments from the perspectives of geometric measure theory \cite{brakke1978motion,tonegawa2019brakke}, the theory of nonlinear parabolic equations \cite{EvansSpruck1991,ChenGigaGoto1989,GigaGotoIshiiSato1991}, and geometric analysis \cite{colding2015mean}, among others. 

\textbf{Brakke's mean curvature flow}. The mean curvature flow was introduced by Brakke in his doctoral dissertation \cite{brakke1978motion}, and accordingly sometimes it is known as Brakke's mean curvature flow. Informally speaking, a time-evolving family of closed hypersurfaces $S_t \subset \mathbb{R}^m$ (enclosing domains $E_t$, so $S_t = \partial E_t$) is said to be evolving by mean curvature flow if the inner normal velocity of $S_t$ at any point $p\in S_t$ is equal to the sum of the principal curvatures of $S_t$ at $p$. 

The informal description above masks a number of deep analytical and geometrical difficulties revolving around the right definition of such a flow, and the nature of the singularities that may form over time even if one starts the flow from a smooth hypersurface. Brakke's original work \cite{brakke1978motion} falls within the framework of geometric measure theory, working with varifolds to formulate a notion of solution for this flow. For a modern treatment of the mean curvature flow from this perspective, see Tonegawa's monograph \cite{tonegawa2019brakke}.
 
\textbf{Huisken's theorem and extinction times.} A celebrated Theorem of Huisken \cite{Huisken1984} says that if the initial surface $S_0$ is the boundary of a convex body $E_0$, then the surface remains convex for the duration of the mean curvature flow. To be exact, Huisken showed there is a positive time $\tau = \tau(E_0)$ and a mean curvature flow $S_t$ defined for $t\in [0,\tau)$ and starting from $S_0$, and such that $S_t$ is the boundary of a smooth convex set for every $t\in (0,\tau)$, and that $S_t\to \{x_*\}$ in the Hausdorff topology to some point $x_* \in \mathbb{R}^m$. 

In Huisken's work the flow is represented via a map $F:\mathbb{S}^{m-1}\times [0,\tau) \to \mathbb{R}^m$ such that $F(\cdot,t)$ gives a parametrization of $S_t$, the evolution of the map $F(\cdot,t)$ with $t$ is governed by a nonlinear parabolic equation for the map $F$. As $t\to \tau$, the map $F(\cdot,\tau)$ converges to a constant (i.e. to the point to which the surfaces are shrinking to as $t\to \tau$). The time $\tau$ is known as the extinction time of $S_0$, or equivalently, of $E_0$ (see also \eqref{e:extinction time functional}). It defines a functional over the set  of convex bodies. In general, there is no explicit expression for the extinction time $\tau(E)$ for a given convex body $E$ (a notable exception is a ball). We note that work of Giga and Yama-uchi \cite{GigaYamauchi1993} provides an estimate for the extinction time in terms of the Cheeger constant of $E$, which is a fundamental quantity in geometry related to the  spectral gap of the Laplace-Beltrami operator \cite{Cheeger70}.\footnote{For purposes of the  Giga and Yama-uchi estimate, we include the portion of the perimeter on the hypercube boundary  in the definition of the Cheeger constant,(i.e., the ambient set is $\R^m$).  Therefore, the  estimates of the  Cheeger constant when  the hypercube boundary is excluded from the perimeter are not applicable in the  Giga and Yama-uchi setting (i.e., when the ambient set is the hypercube   itself; cf. Section 5.5 in \cite{Milman09}).}

\textbf{The level set formulation.} Level set formulations are a way of describing hypersurfaces $S$ of $\mathbb{R}^m$ in terms of a scalar function $u$, ``the level set function'', such that $S = \{u=0\}$. In this way, one can describe geometric quantities of $S$ in terms of an expression involving $u$. For instance, if $x\in S$, then the normal vectors to $S$ are given by $\pm \nabla u / |\nabla u|$. Note also that a family of hypersurfaces $S_t$ varying in time can be represented by a time-dependent level set function $u(x,t)$. For a general introduction to the level set method, see the book by Giga \cite{giga2006surface} and the book by Fedkiw and Osher \cite{osher2004level}. 

A standard computation shows that if $u(x,t)$ is a smooth level set function and $\nabla u \neq 0$,  then for each $\lambda$ the level set $S_t(\lambda) = \{ x : u(x,t) = \lambda\} = \partial \{ x: u(x,t)<\lambda \}$ is evolving in time with an inner normal velocity $V$ at $x \in S_t(\lambda)$ given by 
\begin{align}\label{e:level set velocity}
  V = \frac{\partial_t u}{|\nabla u|}.    
\end{align}
This brings us to level set equations which are also parabolic\footnote{For brevity, throughout this paper we use ``parabolic'' to refer to ``backward parabolic equations'', as we will only be working with backward parabolic equations, in accordance to the structure of the online Koml\'os game.} PDEs. These are equations of the form
 \begin{align*}
   \partial_t u + |\nabla u|F(D^2u,\nabla u) = 0,
 \end{align*}
 where $F(D^2u,\nabla u)$ is a degenerate parabolic operator, meaning that for a fixed $p$ we have $F(M,p) \leq F(M',p)$ whenever $M'-M$ is a positive semi-definite matrix.  
In light of the formula \eqref{e:level set velocity}, the significance of the parabolic PDE for $u$ is that its level sets are moving with normal velocity equal to $F(D^2u(x),\nabla u(x))$. The level set equation is in addition called geometric if $F(M,p) = F(M+s p\otimes p,p)$ for every $M$, $p\neq 0$, and $s\in\mathbb{R}$. This means that the value of $F(D^2u(x),\nabla u(x))$ at some $x=x_0$ is determined by the shape of the level set of $\{ u = u(x_0)\}$. 

As an important example, the operator $F$ giving the level set formulation of mean curvature flow is
\begin{align*}
  F(D^2u,\nabla u) = \tr(\Pi(\nabla u)(D^2u)\Pi(\nabla u))/|\nabla u| = \dive\big( \tfrac{\nabla u}{|\nabla u|}\big),
\end{align*}
where for $p\neq 0$,  $\Pi(p)$  denotes the $m\times m$ matrix representing the orthogonal projection onto the hyperplane normal to $p$. This matrix is given by $(\Pi(p))_{ij} = \delta_{ij}-|p|^{-2}p_ip_j$. Then, the level set PDE for mean curvature flow is 
 \begin{align*}
   \partial_t u + |\nabla u|\dive\big(\tfrac{\nabla u}{|\nabla u|}\big) = 0.
 \end{align*} 
These equations are fully nonlinear parabolic equations, and so they fall within the scope of the theory of viscosity solutions. The level set formulation was first studied by Evans and Spruck \cite{EvansSpruck1991} and Chen, Giga, and Goto \cite{ChenGigaGoto1989}. We also mention work of Giga, Goto, Ishii, and Sato \cite{GigaGotoIshiiSato1991} dealing with general flows and which will also be essential here. As discussed  earlier, Kohn and Serfaty \cite{KohnSerfaty2006} found the level set PDE for mean curvature flow arises as a continuum limit in vector balancing games. Kohn and Serfaty  subsequently extended the deterministic games framework to general fully nonlinear parabolic and elliptic equations \cite{kohn2010}, and Imbert and Serfaty extended it to include integral curvature flows \cite{IS11}. Concurrently, work of Peres, Schramm, Sheffield, and Wilson \cite{peres2009tug} touched on parallel themes as they established a relationship between tug-of-war games and the infinity Laplacian. There is a $1$-parameter family of problems connecting the infinity Laplacian \cite{peres2009tug} and mean curvature flow \cite{KohnSerfaty2006}: they are two extremes in the class of evolution $p$-Laplacian equations ($p\in [1,\infty]$). We also mention works of Charro, Garc\'ia Azorero, and Rossi \cite{charro2009mixed}, as well of Gonzalvez, Miranda, Rossi, and  Ruiz-Cases \cite{GMRR26} involving two-player zero-sum probabilistic games.

  \textbf{More general flows}. One may evolve $S_t = \partial E_t$  according to some other function of the principal curvatures of $S_t$, leading to parabolic level set PDEs as above, as well as to PDE's for maps $F:\mathbb{S}^{m-1}\to\mathbb{R}^m$ in the smooth setting.  We mention in particular work of Andrews \cite{andrews1994contraction}, which considered general geometric flows starting from convex initial data. An important result in \cite{andrews1994contraction} is the preservation of convexity under natural structural assumptions on the flow (i.e. structural assumptions on the function of principal curvatures used). This is important in defining extinction times for other flows besides mean curvature. 
  
  One possibility is the flow where the normal velocity equals the sum of the largest $n\wedge (m-1)$ principal curvatures of $S_t$. The resulting evolution also produces a family of convex bodies that shrink down to a point at some finite time $\tau_{n}(K)$.  When $n\geq m-1$ this is the usual mean curvature flow, and when $n<m-1$ we refer to this as flow by the largest $n$ principal curvatures. Flows with $n=1$ have been shown to arise in stochastic control problems \cite{larsson2024minimum} (see next paragraph).

\textbf{Stochastic optimization interpretation.} Prior to Kohn and Serfaty's deterministic game setting, various works showed mean curvature flow captures the reachability set for certain stochastic control problems. We mention here the works of Soner and Touzi \cite{soner2002dynamic,soner2003stochastic}, of Buckdahn, Cardaliaguet, and Quincampoix \cite{buckdahn2001representation}. More recently, Larsson and Ruf \cite{larsson2024minimum} found that flow by the minimum curvature describes the largest (deterministic) amount of time that a certain martingale can be kept inside a given convex body.

\textbf{Viscosity solutions and fully nonlinear equations}. Viscosity solutions are a powerful tool when dealing with value functions for games or control problems involving in continuum settings. As we have known since works of Bellman\cite{Bellman1957} and  Pontryagin \cite{Pont1985}, a sufficiently smooth value function will solve a PDE associated to the problem, the Hamilton-Jacobi-Bellman equation. However, value functions are typically not smooth everywhere, and one must deal with generalized solutions. Among several notions of ``weak solutions'' to such nonlinear equations, viscosity solutions stand out for producing a complete theory of existence and uniqueness, and for correctly characterizing the weak formulation of the PDE for non-smooth value functions. Just as important is the role of viscosity solutions in the context of continuous limits, where we have Barles and Souganidis \cite{BarSou1991} celebrated theorem on the convergence of finite difference schemes to the unique viscosity solution of a PDE.

Viscosity solutions are built around a natural monotonicity property of the nonlinear operator, and which is directly related to the maximum principle. This monotonicity, sometimes known as the Global Comparison Property (GCP), imposes severe restrictions on the class of operators. In particular, Courr\`ege showed that a linear operator in $\mathbb{R}^n$ satisfying the GCP is essentially the infinitesimal generator of a Markov process \cite{courrege1965}. The corresponding result for nonlinear operators was obtained by Guillen and Schwab \cite{GuiSch2019,GuiSch2020}, where they show that a nonlinear operator having the GCP can be expressed as a min-max of linear operators of the type obtained by Courr\`ege.

For monographs on viscosity solutions we recommend the book of Fleming and Soner \cite{fleming2006controlled} and the book of Bardi and Capuzzo-Dolcetta \cite{bardi1997optimal}, both of which also cover the perspective from optimal control. We also recommend the lecture notes by Silvestre \cite{silvestre15} as well as lecture notes by Calder \cite{calder24}.  Last but not least we mention the forthcoming monograph by Giga and Liu \cite{GL26}, which covers geometric evolution equations arising from deterministic discrete two-person games.

\section{Statement of the main results}\label{s:main results}

In this section, we state the main results of this paper in full generality, containing each of the special theorems discussed in the Introduction (Section \ref{sec:intro}).  For the rest of the paper $E$ will denote a centrally symmetric convex body in $\R^m$,  $m\geq 2$. This body defines a norm, which we denote by $\|\cdot\|_E$. Now, for any $T,n \in \mathbb{N}$ we define a functional over such convex bodies by
\begin{align*}
  \mathcal{K}_{n,T}(E) := \max_{A_1}\min_{\overline \varepsilon_1}\ldots \max_{A_T}\min_{\overline \varepsilon_T}\|A_1\overline \varepsilon_1+\ldots+A_T\overline \varepsilon_T\|_E,
\end{align*}
where for each $t=1\ldots,T$, the minimum is taken over all $\overline \varepsilon_s \in \{-1,1\}^n$ and the maximum is taken over all $m\times n$ matrices $A$ whose columns have at most unit Euclidean length each. 

As discussed in the introduction, given  $E$ and $n\geq 1$, there is well defined flow $E_t$  at least for small positive times,  $E_0 = E$ and $\partial E_t$ moves with normal velocity equal to the sum of the largest $n\wedge (m-1)$ principal curvatures of $\partial E_t$. The sets $E_t$ are all centrally symmetric convex bodies and the flow is defined up until some maximum interval of existence. Then, we define the \emph{extinction time} for $E$ with respect to the  flow by
\begin{align}\label{e:extinction time functional}
  \tau_n(E) := \sup\, \Big\{  t   \Bigm| 
\begin{tabular}{l}
    There exists a flow starting from $E$ with normal velocity equal to the sum  of  \\
     the largest  $n\wedge (m-1)$ principal curvatures,  and it is non-vanishing in $(0,t)$.
\end{tabular}   \Big \}.
\end{align}
For such a time, it is known that $E_t \to \{0\}$ in the Hausdorff distance as $t\to \tau_n(E)$.

\begin{theorem}\label{t:main theorem}
For any centrally symmetric convex body $E$, we have
\begin{align*}
  \lim \limits_{T\to \infty}\frac{\mathcal{K}_{n,T}(E)}{\sqrt{T}} = \frac{1}{\sqrt{2\tau_n(E)}}.  
\end{align*}
where $\tau_n$ is the extinction time defined in \eqref{e:extinction time functional}.

\end{theorem}

\begin{theorem}\label{t:extinction time asymptotics}
For any centrally symmetric convex body $E$, if $n \geq m-1$, we have
\begin{align*}
  \frac{1}{\,l^2(E)} \leq 2\tau_n(E) \leq \frac{2}{  C^2(E) M^2_m }.
\end{align*}
 where
\[
C(E)  := 1/\max_i \| e_i\|^*_E,
\]
 $M_m$ is the expectation of the largest entry, and $ l(E)$ is the expected $\| \cdot\|_E$ norm,  of the standard  Gaussian random vector  $ N(0,I_m)$.
Equivalently, 
\begin{align*}
  \frac{ 1}{\sqrt{2}} C(E) M_m \leq \lim \limits_{T\to \infty}\frac{\mathcal{K}_{n,T}(E)}{\sqrt{T}} \leq l(E).    
\end{align*}
    
\end{theorem}

\begin {remark} \label{rem:gaussian_constants} $l(E)$ is equal to the Gaussian width of the unit ball of the dual norm $\|\cdot\|_E^*$, and for the $\ell_\infty$ norm, the leading order term of this quantity is  $\sqrt {2 \log {m}}$  as $m$ gets large. When $E$ is the unit ball of any  $\ell_p$ norm, $C(E)$ is 1, and the leading order term of $M_m$ is  $\sqrt {2 \log {m}} $ \cite[\S 2.3.2]{Galambos87} (see also \cite{dasgupta14, kamath98, orabona} for nonasymptotic bounds).   
\end {remark}

The foregoing theorems are proved in Sections \ref{s:continuous limit} and \ref{s:analysis of the extinction time}.  The next section provides heuristics and overview of the proofs, while Sections \ref{s:the value function} and \ref{s:Banaszczyk's theorem and the limiting Bellman operator} as well as the Appendices set forth definition and results used in the proof of Theorem \ref{t:main theorem} in Section \ref{s:continuous limit}.

\section{Heuristics and overview of the proofs} \label{s:heuristics and overview of the proof}

The connection between Spencer's pusher-chooser game and mean curvature flow was first made clear in Kohn and Serfaty's work \cite{KohnSerfaty2006}. In their setting, Paul chooses an orthonormal frame of $m-1$ vectors. This already produces mean curvature flow in the limit. Here we show for the same continuum limit holds when the constraint is that the vectors lie in the Euclidean unit ball (and in particular one can take any $n\geq m$ vectors). Ultimately, this is made possible by an estimate of Banaszczyk for discrepancy with respect to the $\ell^2$ norm \cite{Banaszczyk1990}.

\textbf{Elementary strategies.}  \label{sec:elementary_strategies} To explain in rough terms the appearance of mean curvature, it helps to discuss some elementary strategies for Paul and Carol, which guarantee the following bounds at all $T$:
\begin{align}\label{e:intro value function bound}
 \sqrt{ c(m,n) T}  \leq K_{T}(m,n) \leq  \sqrt {\min(n,m)\, T}
\end{align}
where $c(m,n) := \min(n,m-1)/(\min(n,m-1)+1)$.

A possible strategy for Paul is to choose the columns of $a_{t,1},\dots, a_{t,n^*}$ to be $n^*:=\min(n,m-1)$ orthonormal vectors that are  orthogonal  to $y_{t-1}$ in a $n^*+1$ coordinate subspace of $\R^m$ that stays fixed during the game, e.g., $\text{span} \{e_1,\dots, e_{n^*+1}\}$, and set any remaining vectors to zero. This choice means that regardless of Carol's choice of signs,
\begin{align*}
  |\varepsilon_{t,1}a_{t,1}+\ldots+\varepsilon_{t,n}a_{t,n}|^2 = |\varepsilon_{t,1}a_{t,1}|^2+\ldots+|\varepsilon_{t,n}a_{t,n}|^2 = n^*.    
\end{align*}
(As a reminder, in our notation, $|\cdot |$ refers to the $\ell_2$ norm.) Therefore, Paul can make sure that $|y_t|^2 = |y_{t-1}|^2 + n^*$ at every step $t$, and thus
\begin{align*}
  |y_t|^2 = t \, n^*,\;\forall\;t.    
\end{align*}
Since $|y_t| \leq \sqrt{n^*+1}\,\|y_t\|_\infty$, one concludes that 
\begin{align*}
 \|y_t\|_\infty \geq \sqrt{c(m,n)t}, \;\forall\;t,\;c(m,n) = n^*/(n^*+1).
\end{align*}
Thanks to Banaszczyk's theorem (see Theorem \ref{t:Banaszczyk s theorem}), whatever Paul chooses for the vectors $a_{t,1},\ldots,a_{t,n}$, there exists a choice of signs $\varepsilon_t$ for Carol such that 
\begin{align*}
|\varepsilon_{t,1}a_{t,1}+\ldots\varepsilon_{t,n}a_{t,n}|^2 \leq m.   
\end{align*}
On the other hand, when the signs $\eps_{i,t}$ are independently sampled uniformly at random, the expected value  $\E | y_t |^2 \leq n t $, and  by the monotonicity of the $\ell_p$  norm, $\E \| y_t \|_\infty \leq \sqrt {n t}$. Therefore, by  a probabilistic argument, there exists a deterministic policy  of choosing  $\eps_t$ so that
\begin{align*}
  |y_{t}|^2 \leq |y_{t-1}|^2 + \min(m,n), \Rightarrow \|y_t\|_\infty \leq \sqrt{t\min(m,n)}.
\end{align*}
These strategies guarantee the bounds in \eqref{e:intro value function bound}.

\textbf{Appearance of the mean curvature.} There is a sort of generalization of the above argument for a non Euclidean norm, and it serves to illustrate the emergence of the mean curvature in the  vector balancing context. Instead of the norm $\|\cdot\|_\infty$ let us consider a general  norm $\|\cdot\|$ whose unit ball $E$ is assumed to be smooth. 

Observe that $y \mapsto \|y\|^2$ is a $2$-homogeneous function, which is smooth  away from $0$, and in particular, the second derivatives of $\|\cdot\|$ are $0$-homogeneous and their third derivatives are negatively homogeneous. This means that if $y_{t-1}$ is very large we expect $\|\cdot\|^2$ to be close to its second Taylor expansion in say, a ball of radius $1$ centered at $y_{t-1}$.

Then let us assume $\|y_{t-1}\|$ is large (as will be the case when $t$ is large), and consider the unit  vector $\hat y_{t-1} :=y_{t-1} /\|y_{t-1}\|$ as well as the outer normal vector $n (\hat y_{t-1})$ at $\hat y_{t-1}$, which lies on $\partial E$ with unit dual norm $\|n (\hat y_{t-1})\|^* =1$. Using Euler's identity for 1-homogeneous functions, one can check that 
\begin{align*}
  \nabla \|\cdot\|^2 (y_{t-1}) = 2\|y_{t-1}\|n(\hat y_{t-1}),\;D^2 \|\cdot\|^2 (y_{t-1}) = 2\text{II}_{\partial E}(\hat y_{t-1})+2n(\hat y_{t-1})\otimes n(\hat y_{t-1}),
\end{align*}
where $\text{II}_{\partial E}(y)$ denotes the second fundamental form of $\partial E$ at a point $y\in \partial E$. Then, we can use the second order Taylor expansion for $\|\cdot\|^2$ near $y_{t-1}$ to estimate $\|y_{t-1}+A\overline \varepsilon\|^2$, 
\begin{align*}
  \|y_t\|^2 & = \|y_{t-1}+A\overline\varepsilon\|^2 \\
    & = \|y_{t-1}\|^2 + 2\|y_{t-1}\|(n(\hat y_{t-1}),A\overline \varepsilon) + ( \text{II}_{\partial E}(\hat y_{t-1}) A\overline \varepsilon,A\overline \varepsilon) +(A\overline \varepsilon,n(\hat y_{t-1}))^2 + \textnormal{O}(\|y_{t-1}\|^{-1}).  
\end{align*}
Paul can always make sure the gradient term vanishes regardless of Carol's choice. Paul simply needs to choose all the columns of $A$ to be orthogonal to  $n(\hat y_{t-1})$. In this case we do not see the components of $D^2\|\cdot\|^2(y_{t-1})$ in the directions parallel to $n(\hat y_{t-1})$, and we get
\begin{align*}
  \|y_t\|^2 = \|y_{t-1}\|^2 + ( \text{II}_{\partial E}(\hat y_{t-1}) A\overline \varepsilon,A\overline \varepsilon) + \textnormal{O}(\|y_{t-1}\|^{-1}).  
\end{align*}
Now, thanks to Banaszczyk's theorem the second term in the right hand side can be understood and even written in simple terms (this is explained later in Remark \ref{r:Banaszczyk's theorem restated}). In particular, if $n \geq m-1$ it can be shown that
\begin{align*}
  \max_{A_t} \min_{\varepsilon_t} \|y_{t-1}+A_t\bar \varepsilon_t\|^2 = \|y_{t-1}\|^2 + H(\hat y_{t-1}) + \textnormal{O}(\|y_{t-1}\|^{-1}),
\end{align*}
where $H(y)$ denotes the mean curvature of $\partial E$ at $y\in\partial E$. This computation suggests that for large values of $t$ the evolution of $\|y_t\|^2$ is dominated by the values of the mean curvature of $\partial E$, and that, accordingly, the large $t$ asymptotics for $\|y_t\|$ might be governed by the mean curvature of $\partial E$. 

\subsection{Overview of the proofs} We will now give an overview of the proofs of Theorem \ref{t:main theorem} (which contains Theorem \ref{thm:main} as a special case) and of Theorem \ref{t:extinction time asymptotics} (which contains Theorem \ref{thm:asymptotics as m goes to infinity} as a special case). \\

\noindent \textbf{Overview of Theorem \ref{t:main theorem} on large $T$ asymptotics for $\mathcal{K}_{n,T}(E)$}. This is the more general counterpart to Theorem \ref{thm:main}. The proof centers on a value function $v(y,s)$ that encodes  $\mathcal{K}_{n,T}(E)$ for a general centrally symmetric convex body $E$. To facilitate the PDE analysis, we denote the steps $s$ by negative integers with the starting time $-T$ and the final time $0$. At the start, it will be convenient to consider the case where $E$ has a smooth ($C^2$) boundary, with the general case handled via a stability and approximation argument at the very end of the proof. As in Kohn and Serfaty \cite{KohnSerfaty2006}, we will show that this value function, properly rescaled, solves what amounts to a finite difference scheme for a nonlinear parabolic PDE. 
\begin{defn} \label{def:v}
For a centrally symmetric convex body $E$ in $\R^m$ with $m \geq 2$ and $n \geq 1$,  we define the value function $v:\mathbb{R}^m\times \mathbb{Z}_- \to \R_+$ by the relation
\begin{align*}
  v(y,s) := \max_{A_s}\min_{\overline \varepsilon_s}\ldots\max_{A_{-1}}\min_{\overline \varepsilon_{-1}}\|y+A_s\overline\varepsilon_s+\ldots+A_{-1}\overline\varepsilon_{-1}\|_E,\;  y\in \mathbb{R}^m,\;s=0,-1,\ldots,
\end{align*}
where for each $\tau=s,\ldots,-1$, the minimum is taken over all $\overline \varepsilon_\tau \in \{-1,1\}^n$ and the maximum is taken over all $m\times n$ matrices $A$ whose columns  have at most unit Euclidean length each. 
\end{defn}
Observe that  $\mathcal{K}_{n,T}(E) = v(0,-T)$. Accordingly, what we are after is the limit $v(0,-T)/\sqrt{T}$ as $T\to \infty$. We will compute this limit by showing there is a continuum limit for a family of rescalings of $v$. For $\delta>0$, we define $w_\delta:\mathbb{R}^m\times (\delta^2\mathbb{Z}_-)\to\mathbb{R}$ by 
\begin{align*}
w_\delta(x,t) = \delta v(x/\delta,t/\delta^2).
\end{align*}
Choosing $\delta = 1/\sqrt{T}$, we have
\begin{align*}
  \lim_{T\to \infty} \frac{v(0,-T)}{\sqrt{T}} = \lim\limits_{T\to\infty} w_{\frac{1}{\sqrt{T}}}(0,-1). 
\end{align*}
We will show  that $w_\delta$ is a solution to the (nonlinear) finite difference scheme
\begin{align}\label{e:intro finite difference scheme}
  \tfrac{1}{\delta^2}(w_\delta(x,t)-w_\delta(x,t-\delta^2))+\max_A\min_{\overline \varepsilon}\left \{ \delta^{-2}(w_\delta(x+\delta A\overline\varepsilon,t)-w_\delta(x,t)) \right \} = 0.
\end{align}
Thus, determining the large $T$ asymptotics for $\mathcal K_{n,T}(E)$ is reduced to determining whether the above finite difference scheme approximates a PDE in the $\delta \to 0$ limit. In other words, we must determine whether $w_\delta$ converges to the solution of a PDE as $\delta \to 0$. This entails, on one hand, obtaining estimates for $w_\delta(x,t)$ that are uniform in $\delta$, and, on the other hand, identifying the limiting PDE itself. The estimates are obtained first for $v$ (Section \ref{s:the value function}) by taking advantage of the smoothness of $E$ and the dynamic programming principle, and then transferred to $w_\delta$ by rescaling. 

We identify the limiting PDE in Section \ref{s:Banaszczyk's theorem and the limiting Bellman operator}.  First, we show that as $\delta \to 0$ the finite difference scheme captures only derivatives of order at most $2$, that is, for any fixed $C^2$ function $\phi$ we have
\begin{align*}
  \max_A\min_{\overline \varepsilon}\left \{\delta^{-2}(\phi(x+\delta A\overline\varepsilon)-\phi(x)) \right \} = \tfrac{1}{2}H_{\delta,n}(D^2\phi(x),D\phi(x))+o(1),
\end{align*}
as $\delta \to 0$. Here, $H_{\delta,n}$ is given by
\begin{align*}
  H_{\delta,n}(M,p) = \max_{A}\min_{\overline\varepsilon}\left \{ 2\delta^{-1}(A\overline\varepsilon,p)+(MA\overline\varepsilon,A\overline\varepsilon) \right \},
\end{align*}
the maximum being over $m\times n$ matrices $A$ whose columns all have $\ell^2$ norm at most $1$, and the minimum is over $\overline\varepsilon \in \{-1,1\}^n$. In Lemma \ref{l:liminf and limsup for H delta n}, we describe the behavior of $H_{\delta,n}(M,p)$ as $\delta \to 0$ using the language of half-relaxed limits. We show that when $p\neq 0$
\begin{align*}
  \lim \limits_{\delta \to 0}H_{\delta,n}(M,p) = \max_{A}\min_{\overline\varepsilon}\left \{(\Pi(p)M\Pi(p)A\overline\varepsilon,A\overline\varepsilon) \right \},   
\end{align*}
where $\Pi(p)$ is the $m\times m$ matrix representing the orthogonal projection onto the hyperplane orthogonal to $p$. (We leave the discussion of the case when $p=0$ to Section \ref{s:Banaszczyk's theorem and the limiting Bellman operator}.) This is where the convergence proof departs most from \cite{KohnSerfaty2006}, and this is due to the different set of admissible matrices $A$. We take advantage of Banaszczyk's theorem (Theorem \ref{t:Banaszczyk s theorem}) to characterize this last maxmin, obtaining (see  Lemma \ref{l:H n is sum of positive eigenvalues}) that (still for $p\neq 0$) 
\begin{align*}
  \max_{A}\min_{\overline\varepsilon}\left \{(\Pi(p)M\Pi(p)A\overline\varepsilon,A\overline\varepsilon) \right \} = \sum \limits_{i=0}^{k-1}\sigma_{m-i}((\Pi(p)M\Pi(p))_+) =: H_n(M,p),   
\end{align*}
where $k = \min(n,m-1)$ and  for a given symmetric matrix $M$ we use $M_+$ to denote the positive semi-definite part of $M$, and $\sigma_1(M),\ldots,\sigma_m(M)$ to denote the eigenvalues of $M$ ordered from smallest to largest (counting multiplicities). 

This argument enables one to identify what the limiting PDE should be, it is akin to checking the consistency of a numerical scheme, except here we are finding the PDE itself as we pass to the limit. In concrete terms, the above shows that if $\phi=\phi(x,t)$ is a smooth function (say twice differentiable in $x$, differentiable in $t$) then the limit
\begin{align*}
  \lim \limits_{\delta\to 0}\left \{ \tfrac{1}{\delta^2}(\phi(x,t)-\phi(x,t-\delta^2))+\max_A\min_{\overline \varepsilon}\left \{ \delta^{-2}(\phi(x+\delta A\overline\varepsilon,t)-\phi(x,t)) \right \} \right \},
\end{align*}
exists and is equal to (provided $\nabla \phi(x,t)\neq 0$)
\begin{align*}
  \partial_t \phi(x,t) + \tfrac{1}{2}H_n(D^2\phi,\nabla \phi)(x,t).
\end{align*}
The finite difference operator in the equation solved by $w_\delta$ converges to the above operator as $\delta \to 0$. As explained at the end of Section \ref{s:Banaszczyk's theorem and the limiting Bellman operator} this corresponds to the level set PDE for a curvature flow (and, in particular, to the mean curvature flow if $n\geq m-1$). 

Here is where the theory of viscosity solutions comes in (see Appendix \ref{a:basics of viscosity solutions} for definitions). This theory is used in Section \ref{s:continuous limit} where the proof of convergence for the value function is finalized for $E$ is smooth (Theorem \ref{t:continuum limit smooth E case}). Viscosity solutions are based on the comparison principle and the monotonicity of the differential operators involved; it is for this reason this notion of solution enjoys strong stability properties. In fact, the stability is strong enough to preserve being a solution to the finite difference equation \eqref{e:intro finite difference scheme} in the $\delta \to 0$ limit without requiring any a priori compactness of the $w_\delta$ (beyond boundedness). Loosely speaking, this is achieved by showing that every potential limit point of the sequence $w_\delta(x,t)$ lies above a viscosity supersolution and below a viscosity subsolution to the limiting PDE described earlier (see Lemma \ref{l:half relaxed limits are sub or super solutions}). The comparison principle says that the subsolution must lie below the supersolution, and therefore they must be one and the same function, guaranteeing the whole sequence $w_\delta$ converges to a limit (see Remark \ref{r:uniform convergence}).

The proof of Theorem \ref{t:main theorem} is completed in the last part of Section \ref{s:continuous limit} by first using Theorem \ref{t:continuum limit smooth E case} to verify the convergence of $\mathcal{K}_T(E)$ for smooth $E$. This requires relating the value of $w(0,-1)$ (for $w$ solving the limiting PDE) to the extinction time for the corresponding curvature flow starting from $E$. The general case is done via an approximation, making use of the fact that given a non-smooth $E$ we can find for any given $\epsilon>0$ a centrally symmetric smooth convex body $E_\epsilon$ such that 
\begin{align*}
  (1-\epsilon) E_\epsilon \subset E \subset (1+\epsilon)E_\epsilon.   
\end{align*}
From here, we are able to estimate $\mathcal{K}_T(E)$ in terms of $\mathcal{K}_T(E_\epsilon)$ in a uniform way in $T$, which allow us to conclude the proof for a general (centrally symmetric, convex) $E$, including most importantly, $E=[-1,1]^m$.\\

\noindent \textbf{Overview of Theorem \ref{t:extinction time asymptotics} on estimates for the extinction time}. This is the more general counterpart to Theorem \ref{thm:asymptotics as m goes to infinity}. In the discussion on Theorem \ref{t:main theorem}, we saw the relationship between $\tau_n(E)$ and $w(0,-1)$, so the two-sided bound for $\tau_n(E)$ are obtained by obtaining a respective two-sided bound for $w(0,-1)$. As $w$ solves a parabolic PDE, one may use the comparison principle (see Theorem \ref{t:comparison principle}) together with adequately built \emph{barriers} to estimate $w$ from above or below. For simplicity of exposition, we will only discuss here the plan of the proof for $E=[-1,1]^m$.

Let us review how the comparison principle is used. As we know, $w$ solves (in the viscosity sense)
\begin{align*}
  \partial_t w + \tfrac{1}{2}H_n(D^2w,\nabla w) = 0,\; \text{ in } \mathbb{R}^m\times \mathbb{R}_-.
\end{align*}
If a continuous function $\psi$ is such that $\psi(x,0)\leq w(x,0)$ and satisfies the differential inequality
\begin{align*}
  \partial_t \psi + \tfrac{1}{2}H_n(D^2\psi,\nabla \psi) \geq 0,\; \text{ in } \mathbb{R}^m\times \mathbb{R}_-,
\end{align*}
then for all $x$ and $t$ we have $\psi(x,t)\leq w(x,t)$. The same statement holds if we reverse the inequalities, i.e. if we have a continuous function $\overline \psi$ such that $\overline \psi(x,0) \geq w(x,0)$ and 
\begin{align*}
  \partial_t \overline \psi + \tfrac{1}{2}H_n(D^2\overline \psi,\nabla \overline \psi) \leq 0,\; \text{ in } \mathbb{R}^m\times \mathbb{R}_-.
\end{align*}
Then it follows that $\overline\psi(x,t)\geq w(x,t)$ for all $x,t$. In the PDE nomenclature  one says $\psi$ is a subsolution and $\overline \psi$ is a supersolution (see Appendix \ref{a:basics of viscosity solutions} for exact definitions). An auxiliary (often explicit) function that is used as a subsolution or a supersolution via the comparison principle is colloquially known as a \emph{barrier}.

In the proof of Theorem \ref{t:extinction time asymptotics}, we proceed by constructing functions $\psi,\overline \psi$ which are respectively a subsolution and a supersolution to the PDE solved by $w$. Then, we apply the comparison principle twice to conclude that $\psi\leq w\leq \overline \psi$ everywhere. This yields a two sided estimate for $w(0,-1)$ 
\begin{align*}
  \psi(0,-1) \leq w(0,-1) \leq \overline \psi(0,-1).
\end{align*}
The functions $\psi, \overline \psi$ we use turn out to be solutions to the heat equation (with different diffusivity constants, and different terminal values). In particular, the values $\psi(0,-1)$ and $\overline \psi(0,-1)$ correspond to, respectively, the expected largest component and the $\ell_\infty$ norm of Gaussians (of different variance), yielding the bounds in the statement of the Theorem.

It remains to explain how the heat equation is used to produce a subsolution $\psi$ and a supersolution $\overline \psi$ for the nonlinear PDE for $w$. First, as we intend $\psi$ to bound $w$ from below, we choose as its terminal value the function  $x\mapsto \max_i x_i$ which lies below $\|x\|_\infty$.  So, we define $\psi$ as the unique solution to
\begin{align*}
  \partial_t \psi + \tfrac{1}{4} \Delta \psi = 0 \text{ for } t<0, \;\psi(x,0) = \max_i x_i.
\end{align*}
The second derivatives of the function $\psi$ satisfy the following for $t<0$,
\begin{align*}
 \partial_{ij}\psi \leq 0 \text{ whenever } i\neq j,\;  ~\text{and}~\sum_{j=1}^m \partial_{ij}\psi = 0 \text{ for every } i.    
\end{align*}
This means that $-D^2\psi$ is a Laplacian  matrix.\footnote{We are using here the convention that a Laplacian matrix is negative semi-definite for consistency with the standard Laplacian operator in $\mathbb{R}^m$. However, in parts of the literature the sign convention in the definition of the Laplacian matrix is the opposite to ours, making it a positive semi-definite matrix. \label{fn:laplacian}} With this, it is possible to show the function $\psi$ satisfies the bound (see Lemma \ref{l:submodularity infinity Laplacian bound})
\begin{align*}
  \partial_t \psi + \tfrac{1}{2}H_n(D^2\psi,\nabla \psi) \geq 0.
\end{align*}
We see that $\psi$ is a subsolution, as we wanted.

Now, we define $\overline \psi$ as the solution of (note the different diffusivity constant)
\begin{align*}
  \partial_t \overline \psi + \tfrac{1}{2}\Delta \overline \psi = 0\text{ for } t<0,\; \psi(x,0) = \|x\|_\infty.    
\end{align*}
Observe $\psi(x,0)$ is convex in $x$, thus a standard computation with the heat kernel shows that $D^2\overline \psi \geq 0$ for all $t<0$. In particular, whenever $\nabla \overline \psi \neq 0$ we have 
\begin{align*}
  H_n(D^2\overline \psi,\nabla \overline \psi) = \tr(\Pi(\nabla \overline \psi)D^2\overline \psi) \leq \tr(D^2\overline \psi) = \Delta \overline \psi,
\end{align*}
(recall we are in the case where $n\geq m-1$), and so
\begin{align*}
  \partial_t \overline \psi+\tfrac{1}{2}H_n(D^2\overline \psi,\nabla \overline \psi) \leq \partial_t \overline \psi + \frac{1}{2}\Delta \overline \psi = 0,
\end{align*}
and we see $\overline \psi$ is a supersolution for the PDE solved by $w$, as we wanted.

\section{The value function}\label{s:the value function}

In all that follows, we fix a dimension $m \geq 2$ and a batch size $n \geq 1$. Then, for any centrally symmetric convex body $E$ the iterative maxmin formulation the value function $v$ given by Definition \ref{def:v}  allows us to characterize it by a dynamic program:
\begin{align}\label{e:value function two player game}
  \left \{ \begin{array}{rl} v(y,s-1) = &  \max \limits_{A} \min \limits_{\bar \varepsilon} v(y+A\bar \varepsilon ,s) \;\;\;\text{ for } s\leq 0,\\
  v(y,0) = & \|y\|_E. \end{array} \right.
\end{align}
 The maximum is over the set $\mathcal A$ of all $m\times n$ matrices whose $n$ columns $a_1,\ldots,a_n$ are all such that $|a_i|\leq 1$, and the minimum is over $\bar \varepsilon \in \{-1,1\}^n$.

\begin{remark}
When $E= [-1,1]^m$, it is clear that 
\begin{align*}
  v(0,-1)  = \max \limits_{A} \min_{\bar \varepsilon} \|A\bar \varepsilon\|_\infty = K_1(m,n).
\end{align*} 
\end{remark}

\begin{remark}
The function $v$ characterizes the value of $\mathcal{K}_{n,T}(E)$, to be precise,
\begin{align*}
  \mathcal{K}_{n,T}(E) = v(0,-T).    
\end{align*}
Accordingly, the analysis of the long-time asymptotics of $\mathcal{K}_{n,T}(E)$ involves understanding $v$ as $s\to -\infty$.
\end{remark}

To start our analysis, we gather some basic properties of $v$.
\begin{proposition}
  Given any centrally symmetric convex body $E$ and $v$ as in \eqref{e:value function two player game}, we have
  \begin{enumerate}
  
  \item For every $s$ the function $v(\cdot,s)$ is $1$-Lipschitz in $y$ with respect to the $\|\cdot\|_E$ norm.

  \item The value function is monotone in $s$, namely
  \begin{align*}
    v(y,s-1) \geq v(y,s) \text{ for all } (y,s)\in\mathbb{R}^m\times \mathbb{Z}_-.
  \end{align*}

  \end{enumerate}
\end{proposition}

\begin{proof}
  We argue by induction on the non-negative integer $-s$. For the first claim, we note that the fact that the function $\|y\|_E$ is $1$-Lipschitz with respect to $\|\cdot\|_E$ follows from the triangle inequality for the norm $\|\cdot\|_E$. This proves the claim for $s=0$. 
  
  Now suppose that for some $s<0$ we know that $v(y,s)$ is $1$-Lipschitz continuous in $y$, that is
  \begin{align*}
    v(y,s) \leq v(y+h,s)+ \|h\|_E  \text{ for all } y,h\in\mathbb{R}^m.
  \end{align*}
  Then, in particular, for any matrix $A$, any $\bar \varepsilon \in \{-1,1\}^n$, and $y,h \in \mathbb{R}^m$ we have
  \begin{align*}
    v(y+A\bar \varepsilon,s) \leq v(y+A\bar \varepsilon+h,s)+ \|h\|_E  \text{ for all } y,h\in\mathbb{R}^m.
  \end{align*}
  If we fix one such $A$ and consider the minimum in $\bar \varepsilon$ we have
  \begin{align*}
    \min_{\bar \varepsilon} \{ v(y+A\bar \varepsilon,s)\} \leq \min_{\bar \varepsilon} \{ v(y+A\bar \varepsilon+h,s) \} + \|h\|_E  \text{ for all } y,h\in\mathbb{R}^m.
  \end{align*}
  Then, taking the maximum of the result over all $A$
  \begin{align*}
    \max_A \min_{\bar \varepsilon} \{ v(y+A\bar \varepsilon,s)\} \leq \max_A \min_{\bar \varepsilon} \{ v(y+A\bar \varepsilon+h,s) \} + \|h\|_E  \text{ for all } y,h\in\mathbb{R}^m.
  \end{align*}
  In light of the construction of the function $v$, this last inequality can be rewritten as 
  \begin{align*}
    v(y,s-1) \leq v(y+h,s-1) + \|h\|_E \text{ for all } y,h\in\mathbb{R}^m.
  \end{align*}
  We conclude that $v(\cdot,s-1)$ is also $1$-Lipschitz with respect to $\|\cdot\|_E$, and the first claim is proved.

  As for the second claim, simply note that for each $s$, the matrix $A = 0$ is admissible, so $v(y,s-1) \geq v(y,s)$.

\end{proof}

The following pointwise bound for $v(y,s)$ will be important, it follows via a comparison argument.

\begin{lemma}\label{lemma:bound for the value function}
  If $E$ has a $C^2$ boundary, there exists a constant $C \equiv C(E,n)$ such that
  \begin{align*}
    v(y,s) \leq \sqrt{ \|y\|_E^2 + C|s|}\; \textnormal{ for all } (y,s) \in \mathbb{R}^m\times \mathbb{Z}_-.
  \end{align*}

\end{lemma}

\begin{proof}
For a number $C>0$ to be determined below, we consider the function
\begin{align*}
  b(y,s) = \sqrt{\|y\|_E^2-Cs}\;\text{ defined for all } (y,s) \in \mathbb{R}^m\times \mathbb{Z}_-.
\end{align*}
The lemma will follow from the fact that $b$ is a supersolution for the discrete-in-time terminal value problem that characterizes $v$. Specifically, we will show that
\begin{align*}
  b(y,s) & \geq \max_{A}\min_{\bar \varepsilon} b(y+A\bar \varepsilon,s+1)\;\forall\;y\in\mathbb{R}^m, s<0,\\
  b(y,0) & = \|y\|_E\;\forall\;y\in\mathbb{R}^m
\end{align*}
The last equation is straightforward: regardless of the value of $C$, we have that $b(y,0) = \|y\|_E$. The first  inequality requires comparing $\|y+A\varepsilon\|^2_E$ to $\|y\|^2_E$ for every admissible $A$ and a corresponding $\bar \varepsilon$. To prove this inequality,  note that since $b(y,s)$ is always non-negative, it is sufficient to show that for all $y\in\mathbb{R}^m$ and all $s<0$,
\begin{align*}
  b(y,s)^2 & \geq \max_{A}\min_{\bar \varepsilon} b(y+A\bar \varepsilon,s+1)^2,
\end{align*}
or more explicitly, we only need to show that 
\begin{align*}
  \|y\|_E^2-Cs & \geq \max_{A}\min_{\bar \varepsilon} \left \{ \|y+A\overline \varepsilon\|_E^2-C(s+1) \right \}.
\end{align*}
This in turn is equivalent to
\begin{align*}
  \|y\|_E^2+C& \geq \max_{A}\min_{\bar \varepsilon} \left \{ \|y+A\overline \varepsilon\|_E^2 \right \}.
\end{align*}
To establish this last inequality, we make the following standard observation:
\begin{align*}
  \min_{\overline \varepsilon} \|y+A\overline \varepsilon\|^2_E \leq \min_{\overline \varepsilon} \left \{ \tfrac{1}{2}\|y+A\overline \varepsilon\|^2_E + \tfrac{1}{2}\|y-A\overline \varepsilon\|^2_E \right \}.
\end{align*}
This means that
\begin{align*}
  \max_A \min_{\overline \varepsilon} \|y+A\overline \varepsilon\|_E^2 \leq \max_A \min_{\overline \varepsilon} \left \{ \tfrac{1}{2}\|y+A\overline \varepsilon\|^2_E + \tfrac{1}{2}\|y-A\overline \varepsilon\|^2_E \right \}.
\end{align*}
Therefore, it suffices to show that there is $C>0$ large enough so that 
\begin{align*}
  \max_A \min_{\overline \varepsilon} \left \{ \tfrac{1}{2}\|y+A\overline \varepsilon\|^2_E + \tfrac{1}{2}\|y-A\overline \varepsilon\|^2_E \right \} \leq \|y\|_E^2+C,\;\forall\;y \in \mathbb{R}^m.
\end{align*}
Rearranging this expression, we get
\begin{align*}
  \max_A \min_{\overline \varepsilon} \left \{ \tfrac{1}{2}\|y+A\overline \varepsilon\|^2_E + \tfrac{1}{2}\|y-A\overline \varepsilon\|^2_E - \|y\|_E^2\right \} \leq C,\;\forall\;y\in\mathbb{R}^m.
\end{align*}
Since the boundary of $E$ is $C^2$, the function $y\mapsto \|y\|_E^2$ is at least locally $C^{1,1}$ and homogeneous of degree $2$, and its Hessian exists a.e.  and is a $0$-homogeneous function in $\mathbb{R}^m$.  Therefore,  the operator norm of the Hessian is  bounded uniformly on $\R^m$ by some constant $L(E)>0$ a.e.
For the function  $g(z) :=\tfrac{1}{2}\|y+zA \bar \eps\|_E^2$, the  optimization objective above can be bounded as follows: 
 \[
 g(1)  +  g(-1) - 2g(0) = \int_0^1 g' (z) - g'(z-1) dz \leq \| g''(z)\|_{L_\infty([-1,1])}  \leq \frac{1}{2} L(E) \|A\bar \eps\|^2_{2}  \leq C(E,n)
 \]
 where $C(E,n) := \frac{1}{2} L(E) n^2$.
\end{proof}
    
\begin{remark}
  We will not use the exact value $C(E, n)$ in our  estimate of $v$. Instead,   in our proofs, we will only  use the fact that $C(E,n)$ is a finite number to get uniform pointwise bounds on $v$ as we proceed to the continuum limit.
\end{remark}

\section{Banaszczyk's theorem and the limiting Bellman operator}\label{s:Banaszczyk's theorem and the limiting Bellman operator}

We introduce a scaling parameter and proceed towards a continuum limit, following Kohn and Serfaty \cite{KohnSerfaty2006}. Working with a parabolic scaling with spatial scale $\delta$ we define for $x \in \mathbb{R}^m$ and $t \in \delta^2\mathbb{Z}_-$ 
\begin{align*}
  w_\delta(x,t) := \delta v(\delta^{-1}x,\delta^{-2}t).
\end{align*}
The value function $v(y,s)$ solves \eqref{e:value function two player game}, so $w_\delta$ will solve a rescaled version. Observe that 
\begin{align*}
  w_\delta(x,t-\delta^2) & = \delta v(\delta^{-1}x,\delta^{-2}(t-\delta^2))\\
  & = \max \limits_A \min \limits_{\bar \varepsilon} \delta v(\delta^{-1}x+A\bar\varepsilon,\delta^{-2}(t-\delta^2)+1) \\
  & = \max \limits_A \min \limits_{\bar \varepsilon} \delta v(\delta^{-1}(x+\delta A\overline \varepsilon),\delta^{-2}t).  
\end{align*}
Moreover, for every $x$ we have $w_\delta(x,0) = \delta \|x/\delta\|_E = \|x\|_E$. The above shows that $w_\delta$ solves
\begin{align}
  w_\delta(x,t-\delta^2) = & \max \limits_A \min \limits_{\bar \varepsilon} w_\delta(x+\delta A\bar \varepsilon,t)\;\;\;\text{ for } t \in \delta^2\mathbb{Z}_-, \label{e:rescaled Bellman equation}\\
  w_\delta(x,0) = & \|x\|_E. \label{e:rescaled terminal value}
\end{align}
The following theorem of Banaszczyk \cite[Theorem 1]{Banaszczyk1990} will be key for this section; it allows us to characterize the Bellman equation that arises in the limit $\delta \to 0$.

\begin{theorem}[Banaszczyk's theorem] \label{t:Banaszczyk s theorem} Let $D$ be an ellipsoid in $\R^k$ with center at 0 and principal semi-axes $\lambda_1, \dots, \lambda_k$. Then, for any finite set of vectors $z_1, \dots, z_n \in D$ we can find signs $\bar \eps \in \{\pm 1\}^n$ such that
\begin{align*}
 | \varepsilon_1z_1+\ldots+\varepsilon_n z_n |^2 \leq  \sum_{j=1}^k \lambda_j^2.
\end{align*}
\end{theorem}

In the following remark we rephrase  this theorem in equivalent terms that will make its connection to the Bellman equation more apparent.

\begin{remark}\label{r:Banaszczyk's theorem restated}
 If $M \in \R^{m\times m}$ is   positive semi-definite and $k = \min (n,\text{rank}(M))$,  then
  \begin{align}\label{e:Banaszczyk in compact form}
     \max_{A} \min_{\overline \varepsilon} (MA\overline\varepsilon,A\overline\varepsilon) = \sum\limits_{j=0}^{k-1}\sigma_{m-j}(M),
  \end{align}  
  where the feasible sets of $A$ and $\overline \varepsilon$ are as usual.
  \end{remark}
  
  To see why this remark holds, consider an admissible matrix $A$ whose columns are denoted $a_1,\ldots,a_n$ and write $z_i = M^{1/2}a_i$ for each $i$. Then we observe that 
  \begin{align*}
    (MA\overline \varepsilon,A\overline \varepsilon) = (M(\varepsilon a_1+\ldots +\varepsilon_n a_n),\varepsilon a_1+\ldots +\varepsilon_n a_n) = |\varepsilon_1 z_1+\ldots+\varepsilon_n z_n|^2. 
  \end{align*} 
   If the $a_i$'s lie in some $k$-dimensional subspace $V$, and $Ma_i\neq0 $ for every $i$, then the $z_i$ will all lie in an ellipsoid with squares of principal semi axes adding up to no more than the partial trace $\tr_V(M)$. Then, Theorem \ref{t:Banaszczyk s theorem} guarantees that
  \begin{align*}
    \min_{\overline \varepsilon} (MA\overline \varepsilon, A\overline \varepsilon) \leq \tr_V(M).
  \end{align*} 
  We may now optimize over all such subspaces $k$, and conclude that 
   \begin{align*}
     \max_{A} \min_{\overline \varepsilon} (MA\overline \varepsilon, A\overline \varepsilon) \leq \max_{\dim(V)=k}\tr_V(M).   
   \end{align*}
   For a symmetric matrix $M$ the maximum value of $\tr_V(M)$ over all $k$-dimensional subspaces is the same as the sum of the largest $k$ eigenvalues of $M$ (each again value counted according to its multiplicity), so
   \begin{align*}
     \max_A \min_{\overline \varepsilon} (MA\overline \varepsilon,A\overline \varepsilon) \leq \sum\limits_{j=0}^{k-1} \sigma_{m-j}(M).
   \end{align*} 
   On the other hand, if $A_0$ is chosen with columns given by an eigen-basis for $M$ for the largest $k$ eigenvalues of $M$ (and zero for the remaining columns if $n>k$), we see that 
   \begin{align*}
     (MA_0\overline \varepsilon,A_0\overline \varepsilon) = \sum\limits_{j=0}^{k-1} \sigma_{m-j}(M),\;\forall\;\overline \varepsilon \in \{-1,1\}^n,
   \end{align*}
  and the identity in \eqref{e:Banaszczyk in compact form} now follows.

\subsection{The difference operators} Subtracting $w_\delta(x,t)$ from both sides of \eqref{e:rescaled Bellman equation} gives us 
\begin{align*}
    w_\delta(x,t)-w_\delta(x,t-\delta^2) +  \max \limits_A \min \limits_{\bar \varepsilon} \{ w_\delta(x+\delta A\bar \varepsilon,t)-w_\delta(x,t) \} = 0,
\end{align*}
for every $x\in\mathbb{R}^m,\;t\in \delta^2\mathbb{Z}_-$. We shall rewrite this equation in terms of the following finite difference operators, which act on functions $\phi(x,t)$ defined in $\mathbb{R}^m\times \delta^2\mathbb{Z}_-$. They are
\begin{align*}
  \partial^{-\delta^2}_t \phi(x,t) & := \frac{1}{\delta^2}(\phi(x,t)-\phi(x,t-\delta^2))  \\ 
  \mathcal{H}_{\delta,n}(\phi)(x,t) & := \max \limits_A \min \limits_{\bar \varepsilon}  \tfrac{2}{\delta^2}(\phi(x+\delta A\bar \varepsilon,t)-\phi(x,t) ) 
\end{align*}
 In terms of these operators, the equation for $w_\delta$ can be written more compactly,
\begin{align}\label{e:rescaled value function terminal value problem}
  \left \{ \begin{array}{rl} \partial^{-\delta^2}_t w_\delta(x,t)+ \tfrac{1}{2}\mathcal{H}_{\delta,n}(w_\delta)(x,t) & = 0 \textnormal{ in } \mathbb{R}^m \times \delta^2 \mathbb{Z}_-,\\
  w_\delta(x,0) & = \|x\|_E \text{ in } \mathbb{R}^m.\end{array}\right.
\end{align}
Evidently, as $\delta \to 0$ one expects the linear difference operator $\partial^{-\delta^2}_t$ to converge to the $\partial_t$ operator. Similarly, we expect the nonlinear operator $\mathcal{H}_{\delta,n}$ to converge to a nonlinear elliptic operator.  For a matrix $M\in\mathbb{R}^{m\times m}$ and $p\in\mathbb{R}^m$ we define $H_{\delta,n}(M,p)$ as follows:
\begin{align}\label{e:definition of the auxiliary operator H delta n}
  H_{\delta,n}(M,p) := \max_A\min_{\overline\varepsilon}\left \{ 2\delta^{-1}(p,A\overline\varepsilon) + (M A\overline\varepsilon,A\overline\varepsilon) \right \}.
\end{align}

\begin{proposition}\label{p:consistency of discretization}
  Consider a function $\phi:\mathbb{R}^m\times (-\infty,0)\to\mathbb{R}$ such that $\partial_t\phi$ and $D^2\phi$ are continuous with modulus of continuity $\rho(\cdot)$, then the following estimates hold for all $(x,t)$
  \begin{align*}
    |\partial^{-\delta^2}_t\phi(x,t)-\partial_t \phi(x,t)| \leq \rho(\delta^2), \;\;|\mathcal{H}_{\delta,n}(\phi)(x)-H_{\delta,n}(D^2\phi(x),\nabla \phi(x))| \leq 2\rho(\delta n)n^2.
  \end{align*}

\end{proposition}

The proof of this proposition is standard and it can be found in Appendix \ref{a:proof of some standard proposition}. The operator $H_{\delta,n}(M,p)$ captures how $\mathcal{H}_{\delta,n}$ is converging from a discrete (and thus, nonlocal) operator to a differential (local) operator as $\delta \to 0$. While Proposition \ref{p:consistency of discretization} deals with a fixed function $\phi$ (namely, independent of $\delta$) and not with the sequence $w_\delta$, we will be able to relate (in a weak sense) the limiting behavior of the two sequences 
\begin{align*}
  \left \{\mathcal{H}_{\delta}(w_\delta)(x,t)\right \}_{\delta>0} \text{ and } \{ H_{\delta,n}(D^2w_\delta,\nabla w_\delta)(x,t) \}_{\delta>0}.    
\end{align*}
In this way, the sequence $H_{\delta,n}$ captures all of the relevant information of the operator $\mathcal{H}_{\delta,n}$ as $\delta \to 0$. The fact that $H_{\delta,n}$ is a sequence of real valued functions in a finite dimensional space (that is the Cartesian product of the space of symmetric $m\times m$ matrices and $\mathbb{R}^m$) reflects the emergence of a nonlinear local operator of the form $\phi \mapsto H(D^2\phi,\nabla u)$ in the $\delta \to 0$ limit. 

To describe the limit of $H_{\delta,n}$ it will be useful to define a subset of our family of admissible matrices for each given $p\in\mathbb{R}^m$:
\begin{align}\label{e:definition of the set A of p}
  \mathcal{A}(p) := \{ A = (a_1 \;\ldots\;a_n) : |a_i|\leq 1 \textnormal{ and } (a_i,p)=0 \textnormal{ for all } i\} = \{ A \in \mathcal{A}\;:\; A^\top p = 0\}.
\end{align}
We then define $H_n:\text{Sym}(m)\times \mathbb{R}^m\to \mathbb{R}$ by
\begin{align}\label{e:definition of the operator Hn}
  H_n(M,p) & = \max_{A \in \mathcal{A}(p)} \min_{\overline \varepsilon \in \{\pm 1\}^n} (MA\overline\varepsilon,A\overline\varepsilon)
\end{align}
The function $H_n(M,p)$ is finite and well-defined for every $M$ and $p$  since the expression inside the $A\mapsto M(A\overline \varepsilon,A\overline\varepsilon)$ is continuous with respect to $A$ and $\mathcal{A}(p)$ is always a compact set. Moreover, $H_n(M,p)$ is continuous in the set $\{ (M,p): p \neq 0\}$.

On the other hand, $H_n$ cannot be extended to a continuous function for all $(M,p)$, i.e. there are discontinuities at some $(M,p)$'s with $p=0$. This is to be expected, as for $p\neq 0$ the function $p\mapsto H_n(M,p)$ depends on $p$ only on the direction of $p$. If the eigenvalues of $M$ are all different, a sequence $p_k$ with $p_k\to 0$ along several directions will produce values $H_n(M,p_k)$ that oscillate indefinitely between different values.

However, we can and must work with the upper and lower semi-continuous envelopes of $H_n$ given by
\begin{align*}
  (H_n)^*(M,p) := \limsup \limits_{M' \to M,p' \to p} H_n(M',p'),\\ (H_n)_*(M,p) := \liminf \limits_{M' \to M,p'\to p} H_n(M',p').
\end{align*}
These envelopes will be needed to define the notion of weak solution for the limiting PDE. This is a well known practice in the viscosity solutions literature, in particular works on mean curvature flows \cite{EvansSpruck1991,ChenGigaGoto1989}.

We already noted that $H_n$ is continuous when $p\neq 0$, in this case $H_n^*(M,p) = H_{n,*}(M,p)$ for $p\neq 0$. One can show that
\begin{align*}
  (H_n)^*(M,0) & = \max_A \min_{\overline \varepsilon}(MA\overline\varepsilon,A\overline\varepsilon),\\
  (H_n)_*(M,0)&  = \min_{\sigma \in \mathbb{S}^{m-1}} H_n(M,\sigma).
\end{align*}
We recall the maximum in the first line is over all $A$'s whose columns have $\ell^2$ norm at most $1$, or in the notation just introduced, $A \in \mathcal{A}(0)$. In particular, this shows that $(H_n)^*(M,0) = H_n(M,0)$ so while $H_n$ is not continuous for all $M,p$ it is always upper semi-continuous. We will establish this in the following proposition, where we do something stronger: we characterize the half-relaxed limits of $H_\delta(M,p)$ as $\delta \to 0$ (see Remark \ref{r:half relaxed limits H n delta} below for further details). 

\begin{lemma}\label{l:liminf and limsup for H delta n}
For any sequences
\begin{align*}
  M_k \to M,\; p_k \to p,\; \delta_k \to 0,    \end{align*}
  of symmetric matrices $M_k \in \R^{m\times m}$, vectors $p_k \in \R^m$ and scalars $\delta_k>0$,
we have the limit
\begin{align*}
  \lim \limits_{k\to \infty} H_{\delta_k,n}(M_k,p_k) = H_n(M,p),\;\text{ when } p \neq 0,
\end{align*}
as well as the bounds
\begin{align*}
  \liminf \limits_{k\to \infty}H_{\delta_k,n}(M_k,p_k) & \geq (H_n)_*(M,0)~~~ \text { and} ~~~
  \limsup \limits_{k\to \infty}H_{\delta_k,n}(M_k,p_k)  \leq (H_n)^*(M,0) \text{ when } p=0.
\end{align*}
Moreover, both bounds are sharp, i.e., for both bounds one can find sequences where the bounds hold with equality.

\end{lemma}

\begin{remark}\label{r:half relaxed limits H n delta}
  Another way of stating Lemma \ref{l:liminf and limsup for H delta n} is in terms of the half-relaxed limits of $H_{\delta,n}$ as $\delta \to 0$. In particular, the lemma says that when $p =0$,
  \begin{align*}
    (H_n)^*(M,p) & = \limsup_{M'\to M, p'\to p,\delta \to 0} H_{\delta,n}(M',p') = \max_A \min_{\overline \varepsilon}(MA\overline\varepsilon,A\overline\varepsilon),\\
    (H_n)_*(M,p) & = \liminf_{M'\to M, p'\to p,\delta \to 0} H_{\delta,n}(M',p') = \min_{\sigma \in \mathbb{S}^{m-1}} H_n(M,\sigma), 
  \end{align*}
  and furthermore, $(H_n)^*(M,p) = (H_n)_*(M,p) = H_n(M,p)$ whenever $p \neq 0$. 

\end{remark}

\begin{proof} We start by fixing an arbitrary $\delta>0$ and recording two observations about $H_{\delta,n}$. First, it is clear that restricting the maximization in $A$ to just those $A$'s with $A\in \mathcal{A}(p)$ cannot increase the maximum, and so we have for any symmetric $M$ and vector $p$,
\begin{align}\label{e:H delta n basic lower pointwise bound}
  H_{\delta,n}(M,p) \geq \max_{A\in\mathcal{A}(p)}\min_{\overline\varepsilon}(MA\overline\varepsilon,A\overline\varepsilon) = H_n(M,p).    
\end{align}
Here, we have made use of the fact that $A^\top p = 0$ whenever $A \in \mathcal{A}(p)$. 

The second observation concerns $H_{\delta,n}$ when $p\neq 0$, we are going to show that given any symmetric $M$ and a vector $p\neq 0$, if $A$ achieves the maximum in \eqref{e:definition of the auxiliary operator H delta n} we have 
\begin{align}\label{e:H delta n optimal A bound for non zero p}
  \|(I-\Pi(p))A\| & = 
\|\hat p \otimes \hat p A\| \leq \tfrac{1}{2}\delta |p|^{-1} C(m,n)\|M\|,  
\end{align}
where $\Pi(p) := I-\hat p\otimes \hat p$, $\hat p := p/|p|$ for $p\neq 0$. To prove \eqref{e:H delta n optimal A bound for non zero p} we first note that the minimization in \eqref{e:definition of the auxiliary operator H delta n} involves an objective function that is the sum of an odd and an even function of $\overline \varepsilon$. Keeping this in mind, we compare the values of the objective function at $\overline\varepsilon$ and $-\overline \varepsilon$, and conclude that
\begin{align}
  \min_{\overline\varepsilon}\{2\delta^{-1}(p,A\overline \varepsilon)+ (MA\overline \varepsilon,A\overline \varepsilon)\} & =   \min_{\overline\varepsilon}\{ -2\delta^{-1}|(p,A\overline \varepsilon)|+ (MA\overline \varepsilon,A\overline \varepsilon)\}. \label{e:odd_even_relationship}
\end{align}
Then, writing $A = \Pi(p)A+(I-\Pi(p))A$, where , 
\begin{align*}
  (p,A\overline \varepsilon) = (p,(I-\Pi(p))A\overline \varepsilon) =  |p|(A^\top \hat p,\overline\varepsilon). 
\end{align*}
Then, given that $A$ achieves the maximum  in \eqref{e:definition of the auxiliary operator H delta n} for a given $M$ and $p$,  
\begin{align*}
 0 \leq \min_{\overline \varepsilon}\{ -2\delta^{-1}|(p,A\overline \varepsilon)|+ (MA\overline \varepsilon,A\overline \varepsilon)\},
\end{align*}
and thus for every $\overline \varepsilon$ we have
\begin{align*}
  2\delta^{-1}|(A^\top p,\overline \varepsilon)| \leq (MA\overline \varepsilon, A\overline \varepsilon) \leq C(m,n)\|M\|.
\end{align*}
Then, choosing $\overline \varepsilon$ for which
\begin{align*}
  (A^\top p,\overline \varepsilon)= |(Ae_1)^\top\cdot p| + \ldots + |(Ae_n)^\top\cdot p| = \|A^\top p\|_1,   
\end{align*}
we conclude that 
\begin{align*}
  \|A^\top p\|_1 \leq \tfrac{1}{2}\delta C(m,n)\|M\|.      
\end{align*}
Then, noting that $\|\hat p \otimes \hat p A\| =| A^\top \hat p| \leq |p|^{-1}\| A^\top  p\|_1$ we obtain \eqref{e:H delta n optimal A bound for non zero p}.

With these basic estimates in hand, consider a generic sequence $M_k,p_k,\delta_k$ as in the statement of the lemma. First, take the case where $p_k \to p \neq 0$. If we let $A_k$ denote a matrix achieving the maximum for $(M_k,p_k)$, estimate \eqref{e:H delta n optimal A bound for non zero p} says (without loss of generality, $p_k \neq 0$ for every $k$)
\begin{align*}
  \|(I-\Pi(p_k))A_k\|\leq \tfrac{1}{2}\delta_k |p_k|^{-1} C(m,n)\|M_k\|\;\;\forall\;k.
\end{align*}
On account of $p \neq 0$ we know that for all sufficiently large $k$ we have $\tfrac{1}{2}|p_k|^{-1}\leq |p|^{-1}$. Likewise, for large enough $k$ we have $\|M_k\|\leq 1+ 2\|M\|$. Therefore for all large enough $k$ we have the inequality
\begin{align*}
  \|(I-\Pi(p_k))A_k\| \leq \delta_k |p|^{-1}C(m,n)(1+ 2\|M\|).    
\end{align*}
We use this to estimate $(M_kA_k\overline \varepsilon,\overline \varepsilon)$ for large $k$, via the decomposition $A_k = \Pi(p_k)A_k+(I-\Pi(p_k))A_k$,  
\begin{align*}
  (M_kA_k\overline \varepsilon, A_k\overline \varepsilon) \leq (M_k\Pi(p_k)A_k \overline\varepsilon,\Pi(p_k)A_k \overline \varepsilon) + C\|M_k\|\|(I-\Pi(p_k))A_k\|. 
\end{align*}
Then for all large enough $k$ and for every $\overline \varepsilon$ we have
\begin{align*}
  (M_kA_k\overline \varepsilon, A_k\overline \varepsilon) \leq \left \{ (M_k\Pi(p_k)A_k \overline\varepsilon,\Pi(p_k)A_k \overline \varepsilon)\right \} + C(m,n)(1+2\|M\|)^2|p|^{-1}\delta_k. 
\end{align*}
 From here it follows that
\begin{align*}
   H_{\delta_k,n}(M_k,p_k) & \leq \min_{\overline\varepsilon}(M_k(\Pi(p_k)A_k)\overline\varepsilon, (\Pi(p_k)A_k)\overline \varepsilon) + C(m,n)(1+2\|M\|)^2|p|^{-1}\delta_k\\
   & \leq H_n(M_k,p_k) + C(m,n)(1+2\|M\|)^2|p|^{-1}\delta_k,
\end{align*}
where we used that $\Pi(p_k)A_k \in \mathcal{A}(p_k)$ to obtain the second inequality. Passing to the limit, we have
\begin{align*}
   \limsup_{k \to \infty} H_{\delta_k,n}(M_k,p_k) & \leq H_n(M,p).
\end{align*}
The case $p\neq 0$ now follows simply by applying \eqref{e:H delta n basic lower pointwise bound} to $H_{\delta_k}(M_k,p_k)$, which leads to 
\begin{align*}
   \liminf \limits_{k \to \infty} H_{\delta_k,n}(M_k,p_k) & \geq \lim \limits_{k\to \infty} H_n(M_k,p_k) = H_n(M,p).
\end{align*}
This shows $(H_n)^*(M,p) = (H_n)_*(M,p) = H_n(M,p)$ for $p\neq 0$.

Now we consider what happens when $p_k \to 0$. For the liminf bound, we  use equation \eqref{e:H delta n basic lower pointwise bound} once again:
\begin{align*}
  H_{\delta,n}(M,p) \geq H_{n}(M,p) \geq \min_{q}H_{n}(M,q).
\end{align*}
The minimum being over all $q\in\mathbb{R}^m$, noting that the minimum cannot be achieved only at $q=0$, and using that $H_n(M,q)$ is $0$-homogeneous for $q\neq 0$, we obtain in fact that  
\begin{align*}
  H_{\delta,n}(M,p) \geq \min_{\sigma \in \mathbb{S}^{m-1}}H_{n}(M,\sigma).
\end{align*}
Since this is true for every $M,p,\delta$, it follows that
\begin{align*}
  \liminf_{k \to \infty} H_{\delta_k,n}(M_k,p_k) \geq \min_{\sigma \in \mathbb{S}^{m-1}}H_n(M,\sigma).    
\end{align*}
Now for the limsup bound, we use equation \eqref {e:odd_even_relationship}, which implies that
\begin{align*}
  \max_A \min_{\overline\varepsilon}\{2\delta^{-1}(p,A\overline \varepsilon)+ (MA\overline \varepsilon,A\overline \varepsilon)\} \leq \max_A \min_{\overline\varepsilon} (MA\overline \varepsilon,A\overline \varepsilon) = H_n(M,0).
\end{align*}
Then, considering the sequence $M_k,p_k,\delta_k$ we have
\begin{align*}
  \limsup_{k \to \infty} H_{\delta_k,n}(M_k,p_k) \leq H_n(M,0),
\end{align*}
which proves the bound for the limsup. To see the limsup upper bound is optimal, simply consider a sequence $M_k,p_k,\delta_k$ where $p_k = 0$ for every $k$. To see why the lower bound for the liminf is optimal, fix an arbitrary $\sigma_0 \in \mathbb{S}^{m-1}$ and consider a sequence with $p_k = (\delta_k)^{1/2}\sigma_0$, in this case, we have
\begin{align*}
  \lim_{k \to \infty} H_{\delta_k,n}(M_k,p_k) = H_n(M,\sigma_0).  
\end{align*}
Indeed, this follows by applying \eqref{e:H delta n optimal A bound for non zero p}, and noting that $\delta_k|p_k|^{-1} = \delta_k^{1/2}$, in which case the argument used for the case $p_k \to p \neq 0$ says that
\begin{align*}
  \|(I-\Pi(p_k))A_k\| \leq \delta_k^{1/2}(1+2\|M\|)    
\end{align*}
and so the distance of $A_k$ to $\mathcal{A}(\sigma_0)$ goes to zero as $k\to \infty$, proving the last limit holds. Since $\sigma_0$ is arbitrary, we can choose it to be the minimizer of $\sigma \mapsto H_n(M,\sigma)$, and this completes the proof.

\end{proof}

We close this section with a more direct description of $H_n(M,p)$ when $p\neq 0$. We will use the decomposition of a given symmetric matrix $M$ as $M = M_+-M_-$ where $M_+$ and $M_-$ are each positive semi-definite matrices with $M_+M_-=M_-M_+ = 0$. As before, for $p \neq 0$ we use $\Pi(p)$ to denote the orthogonal projection $I- |p|^{-2}p\otimes p$.

\begin{lemma}\label{l:H n is sum of positive eigenvalues}
For $k = \min\{n,m-1\}$ and any symmetric $m\times m$ matrix $M$, we have
\begin{align}\label{e:identity for Hn}
  H_n(M,p) & = \sum_{j=0}^{k-1} \sigma_{m-j}( (\Pi(p)M\Pi(p))_+),\;\text{ when } p\neq 0,
\end{align}
as well as $(H_n)^*(M,0) = (H_n)^*(M_+,0)$ and  $(H_n)_*(M,0) = \min_{\sigma \in \mathbb{S}^{m-1}} H_n((\Pi(\sigma)M \Pi(\sigma))_+,0)$.

\end{lemma}

\begin{proof}
  First, we note that as long as $p\neq 0$, then $A \in \mathcal{A}(p)$ if and only if $A = \Pi(p)A$ and $A \in \mathcal{A}(0)$. It follows that for any symmetric matrix $M$ we have
  \begin{align*}
    H_n(M,p) & = \max_{A \in \mathcal{A}(p)}\min_{\overline \varepsilon}(MA\overline\varepsilon,A\overline\varepsilon)\\
      & = \max_{A \in \mathcal{A}(p)}\min_{\overline \varepsilon}( M\Pi(p)A\overline\varepsilon,\Pi(p)A\overline\varepsilon)\\
      & = \max_{A\in \mathcal{A}(0)}\min_{\overline \varepsilon}( \Pi(p)M\Pi(p)A\overline\varepsilon,A\overline\varepsilon).
  \end{align*}
  Put in succinct terms, this shows that when $p \neq 0$ we have
  \begin{align*}
    H_n(M,p) = H_n(\Pi(p)M\Pi(p),0).   
  \end{align*}
  In light of this, we shall show that for any matrix $\tilde M$ we have
  \begin{align*} 
    H_n(\tilde M,0) = H_n(\tilde M_+,0),
  \end{align*}
  where $\tilde M_{\pm}$ are the aforementioned decomposition, that is $\tilde M = \tilde M_+-\tilde M_-$.

  Since $\tilde M_-$ is positive semi-definite, for any $\xi$ we have
  \begin{align*}
    (\tilde M\xi,\xi) \leq (\tilde M_+\xi,\xi)    
  \end{align*}
  with equality holding if and only if $\tilde M_-\xi = 0$. Using that this holds for $\xi = A\overline \varepsilon$ for any $A,\varepsilon$ it follows that 
  \begin{align*}
    \max_{A\in\mathcal{A}(0)}\min_{\overline \varepsilon} (\tilde M A\overline \varepsilon,A\overline \varepsilon) \leq \max_{A\in\mathcal{A}(0)}\min_{\overline \varepsilon}\left \{ (\tilde M_+A\overline \varepsilon,A \overline \varepsilon) \right \}.
  \end{align*}
  In other words, we have proved that
  \begin{align*}
    H_n(\tilde M,0) \leq H_n(\tilde M_+,0).    
  \end{align*}
  On the other hand, given any $A \in \mathcal{A}(0)$,  there are matrices $A_1,A_2\in\mathcal{A}(0)$ such that $A = A_1+A_2$, $A_1,A_2 \in \mathcal{A}(0)$ and $\tilde M_-A_1=0$ and $\tilde M_+A_2 = 0$. In this case, we have
  \begin{align*}
    (\tilde M_+A\overline\varepsilon,A\overline\varepsilon) = (\tilde M A_1\overline \varepsilon,A_1\overline \varepsilon) 
  \end{align*}
  and therefore, 
  \begin{align*}
    \max_{A\in\mathcal{A}(0)}\min_{\overline\varepsilon}(\tilde M_+A\overline\varepsilon,A\overline \varepsilon) \leq \max_{A\in\mathcal{A}(0)}\min_{\overline\varepsilon}(\tilde MA\overline\varepsilon,A\overline \varepsilon).   
  \end{align*}
  This proves the reverse inequality, 
  \begin{align*}
    H_n(\tilde M_+,0)\leq H_n(\tilde M,0).    
  \end{align*}
  We apply Theorem \ref{t:Banaszczyk s theorem} with the matrix $\tilde M_+$ (using the observations made in Remark \ref{r:Banaszczyk's theorem restated}) and conclude that 
  \begin{align*}
    H_n(\tilde M_+,0) = \sum \limits_{j=0}^{k-1}\sigma_{m-j}(\tilde M_+).
  \end{align*}
  Applying the above to $\tilde M = \Pi(p)M\Pi(p)$ we complete the proof of the lemma in the case when $p\neq 0$.
  
  As for $p=0$, since $(H_n)^* = H_n$ the argument above already yields $(H_n)^*(M,0) = (H_n)^*(M_+,0)$. For $(H_n)_*$ we apply the arguments from $p\neq 0$ together with Lemma \ref{l:liminf and limsup for H delta n} as follows
  \begin{align*}
    (H_n)_*(M,0) & = \min_{\sigma \in \mathbb{S}^{m-1}}H_n(M,\sigma)\\
    & = \min_{\sigma \in \mathbb{S}^{m-1}}H_n(\Pi(\sigma)M\Pi(\sigma),0)\\
    & = \min_{\sigma \in \mathbb{S}^{m-1}}H_n((\Pi(\sigma)M\Pi(\sigma))_+,0).
  \end{align*}
\end{proof}

  In the case when $n\geq m-1$ we have that $k= m-1$, and in this case \eqref{e:identity for Hn} says that when $p\neq 0$
  \begin{align*}
    H_n(M,p) = \tr((\Pi(p)M\Pi(p))_+).
  \end{align*}
  For a given smooth function $\phi$, the quantity $|\nabla \phi|^{-1}H_n(D^2\phi(x),\nabla \phi(x))$ is the sum of the non-negative principal curvatures of the level set of $\phi$ passing through $x$ (this at least when $|\nabla \phi(x)|\neq 0$). If the sublevel sets of $\phi$ are convex, we have their mean curvature, 
  \begin{align*}
    H_n(D^2\phi,\nabla \phi) = |\nabla \phi| \dv\left (\frac{\nabla \phi}{|\nabla \phi|}\right ).
  \end{align*}
  \subsection{Flow by principal curvatures and flow by positive principal curvatures}\label{s:positive curvatures versus curvatures} The last above point is important for us and is worth dwelling upon. Things are different if $u(\cdot,t)$ has sublevel sets which fail to be convex. To make this situation clear we consider another operator, defined for every $n$ by 
  \begin{align*}
\overline H_n(M,p) = \sum\limits_{j=0}^{k-1}\sigma_{m-j}(\Pi(p)M\Pi(p)),
\end{align*}
  with $k = \min\{n,m-1\}$ just as before (note the difference is we are no longer taking the positive part of $\Pi(p)M\Pi(p)$). We also consider the associated level set PDE,
  \begin{align}\label{e:standard curvature flow}
    \partial_t \overline w + \tfrac{1}{2}\overline H_n(D^2\overline w,\nabla \overline w) = 0.
  \end{align}
  This PDE is the level set formulation for flow by the sum of the largest $k$ curvatures --  in contrast to the flow by the sum of the largest k \emph{non-negative} curvatures.  We note that if $\overline w$ is a solution of \eqref{e:standard curvature flow}  such that $\overline w(\cdot,t)$ has convex sublevel sets for every $t$, then $\overline w$ will also solve  
    \begin{align*}
    \partial_t \overline w + \tfrac{1}{2} H_n(D^2\overline w,\nabla \overline w) = 0,
  \end{align*}
  (the difference is we have the operator $H_n$ instead of $\overline H_n$).  Later at the end of Section \ref{s:continuous limit} we review how when the terminal data is given by $\|\cdot\|_E$ the solution to \eqref{e:standard curvature flow} in fact has this property -- in fact, for $n\geq m-1$ and smooth solutions this follows from a celebrated theorem by Huisken \cite{Huisken1984}. The corresponding extension of Huisken's theorem for viscosity solutions was obtained by Evans and Spruck \cite{EvansSpruck1991}, see also discussion at the end of Appendix \ref{a:basics of viscosity solutions}.

\section{Proof of Theorem \ref{t:main theorem}}\label{s:continuous limit}

Now we shall show the convergence of the rescaled value functions, $w_\delta:\mathbb{R}^m\times (\delta^2 \mathbb{Z}_-)\to\mathbb{R}$, as $\delta\to 0$. As we will see, they converge to a continuous function $w(x,t)$ defined in $\mathbb{R}^m\times (-\infty,0]$, and $w(x,t)$ is the unique viscosity solution to the terminal value problem 
\begin{align}
    \left \{ \begin{array}{rl} \partial_t w +\tfrac{1}{2}H_n(D^2w,\nabla w)  & = 0 \textnormal{ in } \mathbb{R}^m\times (-\infty,0),\\
    w(x,0) & = \|x\|_E \textnormal{ in } \mathbb{R}^m. \end{array} \right.\label{e:curvature flow initial data level set formulation}
\end{align}
In Appendix \ref{a:basics of viscosity solutions} we review the basic definitions and some essential results from the theory of viscosity solutions needed in this section. 

The convergence happens in a locally uniform way (restricted to the domain of definition of $w_\delta$), as detailed below. We record this as a theorem, the proof of which will be the main focus of this section.

\begin{theorem}\label{t:continuum limit smooth E case}
  Assume $E$ has a $C^2$ boundary. As $\delta \to 0^+$, the functions $w_\delta(x,t)$ converge to a limit $w(x,t)$ in the sense that for any $R>0$ we have
  \begin{align*}
    \lim \limits_{\delta \to 0^+} \max \{ |w_\delta(x,t)-w(x,t)| \text{ for } |x| \leq R,|t|\leq R \text{ and } t \in \delta^2 \mathbb{Z}_-\} = 0,
  \end{align*}
  where $w: \mathbb{R}^m\times (-\infty,0] \to \mathbb{R}$ is the unique viscosity solution of the terminal value problem \eqref{e:curvature flow initial data level set formulation}.

\end{theorem}

\subsection{Setting up the proof of convergence} The very first step towards showing the convergence is to bound the values $\{w_\delta(x,t)\}_\delta$ uniformly in $\delta$ at least in compact subsets of $(x,t)$'s.

\begin{proposition}\label{p:pointwise bound for w_delta}
  Assume the convex body $E$ has a $C^2$ boundary, there is a $C\equiv C(E,n)$ such that 
  \begin{align*}
    \|x\|_E \leq w_\delta(x,t) \leq \sqrt{\|x\|_E^2-Ct},\;\forall\;x\in\mathbb{R}^m,t<0.
  \end{align*}  
\end{proposition}

\begin{proof}
This is simply a restating of Lemma \ref{lemma:bound for the value function} in terms of $w_\delta$. Indeed, from the definition of $w_\delta$ and the pointwise bound from Lemma \ref{lemma:bound for the value function} we have
\begin{align*}
  w_\delta(x,t) \leq \delta v(x/\delta,t/\delta^2) \leq \delta \sqrt{ \|x/\delta\|_E^2-C(t/\delta^2)} = \sqrt{\|x\|_E^2-Ct}.
\end{align*}
In light of this, note that
\begin{align*}
  w_\delta(x,t) -\|x\|_E \leq \sqrt{\|x\|_E^2-Ct}-\|x\|_E \leq \sqrt{C|t|}.    
\end{align*}
The last inequality follows from the inequality $\sqrt{a+b}-\sqrt{a}\leq \sqrt{b}$ valid whenever $a,b\geq 0$. The fact that $w_\delta(x,t)\geq \|x\|_E$ follows from the monotonicity of $v(y,s)$ with respect to $s$.

\end{proof}
In light of the bound in Proposition \ref{p:pointwise bound for w_delta} the convergence of the sequence $\{w_\delta\}_\delta$ would fail only if (informally speaking) the ``limsup'' and ``liminf'' of the sequence take different (but in any event finite) values. To investigate this, we must work with the half-relaxed limits of $\{w_\delta\}_\delta$, these are defined by
\begin{align*}
  w^*(x,t) & := \limsup_{\delta \to 0^+,\;x'\to x,\;t'\to t}w_\delta(x',t'),\\
  w_*(x,t) & := \liminf_{\delta \to 0^+,\;x'\to x,\;t'\to t}w_\delta(x',t'),
\end{align*}
where the convergence is understood to be along arbitrary sequences $\{\delta_k\}_k,\{x_k'\}_k,\{t'_k\}_k$ such that 
\begin{align*}
  \delta_k>0,\; t'_k \in \delta_k^2 \mathbb{Z}_- \text{ for all } k.
\end{align*}
The functions $w^*$ and $w_*$ are locally bounded functions in $\mathbb{R}^m\times (-\infty,0]$. From their definition and from Proposition \ref{p:pointwise bound for w_delta} it is clear that 
\begin{align*}
  \|x\|_E \leq w_*(x,t) \leq w^*(x,t) \leq \sqrt{\|x\|_E^2-Ct}.
\end{align*} 
For a generic family of functions $w_\delta$ we cannot expect more than the above, and so $w_* \not\equiv w^*$ in general. Now we work $w$ to rule out this last scenario, and here is where we use discrete equation solved by $w_\delta$. Concretely, we will show $w_*$ and $w^*$ are respectively a viscosity supersolution and viscosity  subsolution  for the problem \eqref{e:curvature flow initial data level set formulation} (see Appendix \ref{a:basics of viscosity solutions} for a review of basic notions and results).  From here, the comparison principle for viscosity sub/supersolutions says that
\begin{align*}
    w^*(x,t) \leq w_*(x,t) \text{ for all } x,t,
\end{align*}
from where $w_* = w^* = w$ will follow. The agreement between the half-relaxed limits guarantees the locally uniform convergence stated in Theorem \ref{t:continuum limit smooth E case}, as we explain in the following remark. 

\begin{remark}\label{r:uniform convergence}
  Consider $\delta_k \to 0$, and for each $k$ take $(x_k,t_k) \in \overline{B}_R(0)\times [-R^2,0] \cap \mathbb{R}^m\times \delta_k^2\mathbb{Z}_-$ such that
    \begin{align*}
    |w_{\delta_k}(x_k,t_k)-w(x_k,t_k)| = \max \{ |w_{\delta_k}(x,t)-w(x,t)| \text{ for } |x| \leq R,|t|\leq R \text{ and } t \in \delta_k^2 \mathbb{Z}_-\}
  \end{align*}
  Passing to a subsequence $(\delta_k',x_k',t_k')$ we may assume without loss of generality $x_k'\to x_*$ and $t_k'\to t_*$ for some $(x_*,t_*)$. Then, since $w$ is continuous, we have that
  \begin{align*}
    \lim\limits_{k}|w_{\delta_k'}(x_k',t_k')-w(x_k',t_k')| = |w(x_*,t_*)-w(x_*,t_*)| = 0.
  \end{align*}    
\end{remark}

\subsection{Establishing the half-relaxed limits are a sub/super solution} We now proceed to show the subsolution and supersolution property for $w^*$ and $w_*$ respectively. First, we show both $w^*$ and $w_*$ have the right terminal values at $t=0$ (in particular, the functions agree at $t=0$).

\begin{proposition}\label{p:half relaxed limits take inital data}
As before, assume $E$ has a $C^2$ boundary. The half-relaxed limits agree at time $0$, namely
\begin{align*}
  w^*(x,0) = w_*(x,0) = \|x\|_E.    
\end{align*}
    
\end{proposition}

\begin{proof}

This is a consequence of the fact that $w_\delta(x,t)$ is monotone in $t$ and the bound in Proposition \ref{p:pointwise bound for w_delta}. Indeed, on one hand $w_\delta(x,t_2) \geq w_\delta(x,t_1)$ for every $t_1,t_2\in \delta^2\mathbb{Z}_{-}$ with $t_2\leq t_1$, so that in particular
\begin{align*}
  w_\delta(x,t) \geq w_\delta(x,0) = \|x\|_E \;\forall\;x\in \mathbb{R}^m,\;\forall\;t\in\delta^2 \mathbb{Z}_{-}.     
\end{align*}
On the other hand, thanks to the second inequality in Proposition \ref{p:pointwise bound for w_delta} we have
\begin{align*}
  w_\delta(x,t)-\|x\|_E \leq \sqrt{C|t|} \;\forall\;x\in \mathbb{R}^m,\;\forall\;t\in\mathbb{Z}_{-}.        
\end{align*}
In other words, for all $(x,t)$ in the domain of definition of $w_\delta$ we have
\begin{align*}
  \|x\|_E \leq w_\delta(x,t) \leq \|x\|_E + \sqrt{C|t|}.   
\end{align*}
Then, passing to the limit along sequences $x'\to x$ and $t'\to 0$ we conclude that
\begin{align*}
  \|x\|_E \leq w_*(x,0)\leq w^*(x,0) \leq \|x\|_E,  
\end{align*}
and therefore $w^*(x,0)= w_*(x,0) = \|x\|_E$.
\end{proof}

\begin{lemma}\label{l:half relaxed limits are sub or super solutions} Suppose $\phi(x,t)$ is a smooth function such that $w^*-\phi$ has a local maximum at $(x_0,t_0)$, $t_0<0$, then
  \begin{align*}
    \partial_t\phi(x_0,t_0) + \tfrac{1}{2}(H_n)^*(D^2\phi(x_0,t_0),\nabla \phi(x_0,t_0)) \geq 0.   
  \end{align*}
  If instead $w_*-\phi$ has a local minimum at $(x_0,t_0)$ then
  \begin{align*}
    \partial_t\phi(x_0,t_0) + \tfrac{1}{2}(H_n)_*(D^2\phi(x_0,t_0),\nabla \phi(x_0,t_0)) \leq 0.   
  \end{align*}

\end{lemma}

\begin{proof}

As this is a standard argument in the theory of viscosity solutions we leave the proof of this lemma for later in Appendix \ref{a:proof of some standard proposition}.    
\end{proof}

We note that in the present setting the argument for Lemma \ref{l:half relaxed limits are sub or super solutions} relies crucially on the relationship between $\mathcal{H}_{\delta,n}$, $H_{\delta,n}$, and $H_n$ as established in Lemma \ref{l:liminf and limsup for H delta n} and Proposition \ref{p:consistency of discretization}, from these the rest of the proof follows common arguments in the literature.

\begin{proof}[Proof of Theorem \ref{t:continuum limit smooth E case}]

As explained earlier in this section, the theorem will be proved if we can show that $w_*(x,t) = w^*(x,t)$ for all $x,t$. Since $E$ has a $C^2$ boundary we can apply Proposition \ref{p:half relaxed limits take inital data} and Lemma \ref{l:half relaxed limits are sub or super solutions} and conclude that (in the viscosity sense) we have 
\begin{align*}
  \partial_t w^* + \tfrac{1}{2}H_n(D^2w^*,\nabla w^*) & \geq 0 \text{ in } \mathbb{R}^m\times (-\infty,0),\\
  w^*(x,0) & = \|x\|_E,
\end{align*}
and 
\begin{align*}
  \partial_t w_* + \tfrac{1}{2}H_n(D^2w_*,\nabla w_*) & \leq 0 \text{ in } \mathbb{R}^m\times (-\infty,0),\\
  w_*(x,0) & = \|x\|_E,
\end{align*}
 Furthermore, 
\begin{align*}
  \|x\|_E \leq w_*(x,t) \leq w^*(x,t) \leq \sqrt{\|x\|_E^2-Ct}.    
\end{align*}
We are then within the hypothesis of the comparison principle (Theorem \ref{t:comparison principle}) and conclude that $w^* \leq w_*$, and so $w^*=w_*$. It follows $w_*$ and $w^*$ are the same function which is both a subsolution and a supersolution, and therefore is the unique viscosity solution to the terminal value problem. This finishes the proof.

\end{proof}

\begin{remark}\label{r:rescaling the 1/2 factor} Strictly speaking, the level sets of $w$ are not evolving by the curvature flow due to the factor of $1/2$ in the PDE. If we set $\tilde w(x,t) := w(x,2t)$, then $\partial_t \tilde w = 2\partial_t w$ so that
\begin{align*}
  \partial_t \tilde w + H_n(D^2\tilde w,\nabla\tilde w) = 0,    
\end{align*}
and the level sets of the rescaled function $\tilde w$ are evolving via the respective curvature flow.

\end{remark}

\subsection{Proof of the main theorem} We recall the functional $\mathcal{K}_{n,T}(E)$, which we defined by
\begin{align*}
  \mathcal{K}_{n,T}(E) = \max_{A_1}\min_{\overline \varepsilon_1}\ldots\max_{A_T}\min_{\overline \varepsilon_T}\|A_1\overline \varepsilon_1+\ldots+A_T\overline \varepsilon_T\|_E.  
\end{align*}

At first sight it is not obvious the limit might exist. This convergence and identification of the limit will be established first when $E$ has a smooth boundary (by means of Theorem \ref{t:continuum limit smooth E case}). For a non smooth $E$ we will use an approximation argument exploiting the monotonicity and homogeneity properties of the functionals $\mathcal{K}_{n,T}$ and $\tau_n$.  First, the monotonicity: if $E,E'$ are centrally symmetric convex bodies and $E \subset E'$, then 
\begin{align*}
  \mathcal{K}_{n,T}(E') \leq \mathcal{K}_{n,T}(E),\; \tau_n(E) \leq \tau_n(E').
\end{align*}
Second, the homogeneity: if $E$ is a centrally symmetric convex body and $\lambda>0$ we have 
\begin{align*}
  \mathcal{K}_{n,T}(\lambda E) = \lambda^{-1} \mathcal{K}_{n,T}(E),\; \tau_n(\lambda E) = \lambda^2 \tau_n(E).    
\end{align*}
These properties imply the following approximation lemma.

\begin{lemma}\label{l:approximation lemma}
    Given a centrally symmetric convex body $E$ and $\epsilon \in (0,1)$ there is a centrally symmetric convex body $E_\epsilon$ with a smooth boundary such that for every $n$ and $t\in \mathbb{N}$ we have
    \begin{align*}
      \frac{1-\epsilon}{1+\epsilon} \mathcal{K}_{n,T}(E_\epsilon)\sqrt{\tau_n(E_\epsilon)} \leq \mathcal{K}_{n,T}(E)\sqrt{\tau_n(E)} \leq  \frac{1+\epsilon}{1-\epsilon} \mathcal{K}_{n,T}(E_\epsilon)\sqrt{\tau_n(E_\epsilon)}.    
    \end{align*}

\end{lemma}

\begin{proof}

For $E$ as in the statement of the lemma and $\varepsilon>0$ small, there is a centrally symmetric convex body $E_\varepsilon$ with a smooth boundary such that 
\begin{align*}
(1-\varepsilon)E_\varepsilon \subset E \subset (1+\varepsilon) E_\varepsilon.    
\end{align*}
From the reverse monotonicity of $\mathcal{K}_{n,T}$ we have
\begin{align*}
\mathcal{K}_{n,T}((1+\epsilon)E_\varepsilon) \leq \mathcal{K}_{n,T}(E) \leq  \mathcal{K}_{n,T}((1-\varepsilon) E_\varepsilon), 
\end{align*}
and from the homogeneity of $\mathcal{K}_{n,T}$, 
\begin{align*}
  \frac{1}{1+\varepsilon}\mathcal{K}_{n,T}(E_\varepsilon) \leq \mathcal{K}_{n,T}(E) \leq \frac{1}{1-\varepsilon}\mathcal{K}_{n,T}(E_\varepsilon).    
\end{align*}
Likewise, for the $\tau_n$ functional, the inclusions imply that
\begin{align*}
  \tau_n((1-\epsilon)E_\epsilon) \leq \tau_n(E) \leq \tau_n((1+\epsilon)E_\epsilon),  
\end{align*}
and then the homogeneity of $\tau_n$ yields
\begin{align*}
  (1-\epsilon)\sqrt{\tau_n(E_\epsilon)} \leq \sqrt{\tau_n(E)} \leq (1+\epsilon)\sqrt{\tau_n(E_\epsilon)}.
\end{align*}
Multiplying the inequalities for $\mathcal{K}_{n,T}$ and $\tau_n$ the lemma follows.

\end{proof}

\begin{proof}[Proof of Theorem \ref{t:main theorem}]

  \textbf{Part I: Case when $E$ has a smooth boundary.} In terms of the value function $v=v_E$ defined by \eqref{e:value function two player game}, we have the identity
\begin{align*}
  \frac{\mathcal{K}_{n,T}(E)}{\sqrt{T}} = \frac{v(0,-T)}{\sqrt{T}}.    
\end{align*}
Now given that $w_\delta(x,t) = \delta v(x/\delta,t/\delta^2)$ by definition we also have
\begin{align*}
  \frac{\mathcal{K}_{n,T}(E)}{\sqrt{T}}  = \frac{\delta^{-1}w_\delta(0,-\delta^2 T)}{\sqrt{T}}.    
\end{align*}
Choosing $\delta = 1/\sqrt{T}$, this results in
\begin{align*}
  \frac{\mathcal{K}_{n,T}(E)}{\sqrt{T}} = w_{1/\sqrt{T}}(0,-1).    
\end{align*}
Given the smoothness assumption on $E$ we can apply Theorem \ref{t:continuum limit smooth E case} and conclude $w_\delta \to w$ locally uniformly, and in particular $w_{1/\sqrt{T}}(0,-1) \to w(0,-1)$ as $T\to \infty$, where $w(x,t)$ is the unique solution to \eqref{e:curvature flow initial data level set formulation}. This establishes the existence of the limit
\begin{align*}
   \lim_{T\to\infty}\frac{\mathcal{K}_{n,T}(E)}{\sqrt{T}},
\end{align*}
which agrees with $w(0,-1)$. Now we consider the function $\overline w$ solving the terminal value problem
\begin{align*}
  \partial_t \overline w + \tfrac{1}{2}\overline H_n(D^2\overline w,\nabla \overline w) & = 0 \text{ in } \mathbb{R}^m\times \mathbb{R}_-\\
  \overline w(x,0) & = \|x\|_E,
\end{align*}
that is, $\overline w$ is the level set function describing flow by the sum of the largest $\min\{n,m-1\}$ curvatures. Then, according to Theorem \ref{t:convex level sets} (see discussion at the end of Appendix \ref{a:basics of viscosity solutions}) the unique viscosity solution $\overline w$ has the property that $\overline w(\cdot,t)$ has convex sublevel sets for every $t$. In light of this (recall the discussion in Section \ref{s:positive curvatures versus curvatures}) it can be shown that $\overline w$ is also a viscosity solution of 
\begin{align*}
  \partial_t \overline w + \tfrac{1}{2}H_n(D^2\overline w,\nabla \overline w) = 0.
\end{align*}
Then, by the comparison principle for viscosity solutions (see Theorem \ref{t:comparison principle}) we must have $w(x,t) = \overline w(x,t)$ for all $x,t$. In particular (recall Remark \ref{r:rescaling the 1/2 factor}) the level sets of $w(\cdot,2t)$ are evolving according to the sum of the largest $\min\{n,m-1\}$ principal curvatures.

To conclude this first part of the proof it remains to explain why is it that 
\begin{align}\label{e:w at 0 -1 gives the extinction time}
w(0,-1) = \frac{1}{\sqrt{2\tau_n(E)}}.
\end{align}
For this, we find it convenient to argue in terms of the rescaled function $\tilde w(x,t) = w(x,2t)$, whose level sets are evolving by the curvature flow (again, as noted in Remark \ref{r:rescaling the 1/2 factor}). In particular, we see for $t>0$ the value $\tilde w(0,-t)$ is equal that of the level set of $\tilde w(\cdot,0) = \|\cdot\|_E$ with extinction time equal to $t$. Since $E = \{ \tilde w(\cdot,0) \leq 1\}$, it follows the level set vanishing at time $\tau_n(E)$ is the one corresponding to $\{\tilde w(\cdot,0) = 1\}$, that is $\partial E$. It follows that $\tilde w(0,-\tau_n(E)) = 1$. Now, the fact that $w(\cdot,0)$ is homogeneous of degree $1$ means $w(x,t)$ will be self similar, and one has that (see Lemma \ref{l:solution is self similar})
\begin{align*}
  w(x,t) = \sqrt{-t}w(x/\sqrt{-t},-1).    
\end{align*}
 Therefore, 
\begin{align*}
  1= \tilde w(0,-\tau_n(E)) = w(0,-2\tau_n(E)) = \sqrt{2\tau_n(E)}w(0,-1),
\end{align*}
and the identity \eqref{e:w at 0 -1 gives the extinction time} follows. This proves that
\begin{align*}
  \lim\limits_{T\to \infty}\frac{\mathcal{K}_{n,T}(E)}{\sqrt{T}} = \frac{1}{\sqrt{2\tau_n(E)}},    
\end{align*}
in the case when $\partial E$ is smooth. 

 \textbf{Part II: Approximation of general $E$ by smooth convex bodies.} For the general case fix $E$ and $\epsilon\in(0,1)$. Then by Lemma \ref{l:approximation lemma} there is a smooth, centrally symmetric convex body $E_\epsilon$ such that 
    \begin{align*}
      \frac{1-\epsilon}{1+\epsilon} \mathcal{K}_{n,T}(E_\epsilon)\sqrt{\tau_n(E_\epsilon)} \leq \mathcal{K}_{n,T}(E)\sqrt{\tau_n(E)} \leq  \frac{1+\epsilon}{1-\epsilon} \mathcal{K}_{n,T}(E_\epsilon)\sqrt{\tau_n(E_\epsilon)}.    
    \end{align*}
  The convex body $E_\varepsilon$ has a smooth boundary and thus according to Part I, 
  \begin{align*}
    \lim_{T\to \infty}\left \{ \frac{\mathcal{K}_{n,T}(E_\epsilon)}{\sqrt{T}}\sqrt{2\tau_n(E_\epsilon)} \right \} = 1.   
  \end{align*}
  Combining this limit with the inequalities above we conclude that 
  \begin{align*}
   \frac{1-\epsilon}{1+\epsilon}\frac{1}{\sqrt{2\tau_n(E)}}\leq \liminf_{T\to \infty}\frac{\mathcal{K}_{n,T}(E)}{\sqrt{T}}\leq \limsup_{T\to \infty}\frac{\mathcal{K}_{n,T}(E)}{\sqrt{T}} \leq \frac{1+\epsilon}{1-\epsilon}\frac{1}{\sqrt{2\tau_n(E)}}.
  \end{align*}
  Since this holds for all $\epsilon\in(0,1)$ we may take $\epsilon \to 0$ and conclude that
  \begin{align*}
    \lim_{T\to \infty}\frac{\mathcal{K}_{n,T}(E)}{\sqrt{T}} = \frac{1}{\sqrt{2\tau_n(E)}},
  \end{align*}
  and the theorem is proved.
\end{proof}

\section{Proof of Theorem \ref{t:extinction time asymptotics}}\label{s:analysis of the extinction time}

We now establish lower and upper bounds on the extinction time for $E$ under mean curvature flow, in particular, $n\geq m-1$ throughout this section. To obtain the bounds, we will construct adequate subsolutions and supersolutions for the PDE solved by $w$, and then apply the comparison principle in the viscosity framework \cite[Theorem 3.2]{EvansSpruck1991}. Then, the bound for the extinction times follows from \eqref{e:w at 0 -1 gives the extinction time}. As noted previously, the definitions of viscosity solutions and the comparison principle are reviewed in Appendix \ref{a:basics of viscosity solutions}.

Throughout this section we will study the specific terminal value problem
  \begin{align}\label{e:terminal value problem cube sec 7}
    \left \{ \begin{array}{rl}\partial_t w + \tfrac{1}{2}|\nabla w|\dive\left ( \frac{\nabla w}{|\nabla w|}\right ) & = 0 \text{ in } \mathbb{R}^m\times (-\infty,0)\\ 
    w(x,0) & = \|x\|_E \text{ in } \mathbb{R}^m \end{array}\right. 
  \end{align}
  The upper bound will be the simpler of the two estimates, so we discuss the lower bound first.

\subsection{The lower bound} For the lower bound we will show that a solution to the heat equation with a properly chosen diffusivity and terminal data will produce a subsolution to \eqref{e:terminal value problem cube sec 7}. Therefore, we define $\psi(x,t)$ as the solution to the terminal value problem,
\begin{align}\label{e:heat equation factor 1/4}
  \left \{ \begin{array}{rl}
 \partial_t \psi + \tfrac{1}{4} \Delta \psi  & =0  \text{ in } \mathbb{R}^m\times (-\infty,0),\\
  \psi(x,0)  & = C(E) \max_i x_i,  \end{array}\right.
\end{align}
where $C(E)$ is the constant defined by
\begin{align*}
C(E) := 1/\max_i \| e_i\|^*_E.
\end{align*}

\begin{remark} The idea behind the definition of $C(E)$ comes from the case where $E=[-1,1]^m$. Then $\|x\|_E = \|x\|_\infty$ and we have that $\|x\|_\infty \geq \max x_i$, so $C_E = 1$ in this case.     
\end{remark}

The initial data $g(x) := C(E) \max_i x_i$ has a few properties that will be useful going forward. First, $g(x)$ is convex and monotone increasing in each variable $x_i$, this means that if $p$ is any element of its sub-differential, then the components of $p$ are all non-negative.  

Next, $g(x)$ is a \emph{submodular} function,  which means that  if $e_1,\ldots,e_m$ denotes the canonical basis, then for any $h>0$ and indices $i\neq j$ we have that
\begin{align*}
  g(x+he_i+he_j)+g(x)-g(x+he_i)-g(x+he_j) \leq 0 \;\forall \;x \in \mathbb{R}^m.   
\end{align*}
Observe that a twice differentiable function is submodular if and only if the mixed second derivatives are all non-positive. We also observe $g(x)$ is linear along the direction of the diagonal. In other words, if $\mathbf{1} := (1,\ldots,1) \in\mathbb{R}^m$, then  for every $x \in \mathbb{R}^m$ and $s\in \mathbb{R}$ we have
\begin{align*}
  g(x+s\mathbf{1}) = g(x) + C(E) s.    
\end{align*}
These are all properties preserved by translations and convex combinations, and thus they are preserved by convolution against a non-negative smooth kernel. Since $\psi(x,t)$ is obtained by convolution of $g(x)$ against the heat kernel, the above applies to $\psi(x,t)$ for $t<0$. These properties of $\psi(x,t)$ for $t<0$ say that $-D^2\psi(x,t)$ is what is known as a Laplacian matrix.

\begin{defn}\label{def:Laplacian matrix}
We will say a $L\in \R^{m\times m}$ is a Laplacian matrix if  it is a symmetric matrix such that\footnote{As a reminder in our sign convention Laplacian is negative definite matrix. See  footnote \ref{fn:laplacian} in Section \ref{s:heuristics and overview of the proof} regarding this convention.}
\begin{align*}
  L_{ij} \geq 0 \text{ if } i\neq j\; \text{and } \sum_{j=1}^mL_{ij} = 0 \;\forall \;i \in [m].
\end{align*}

\end{defn}

This property  yields a useful pointwise bound for $(D^2\psi \nabla \psi,\nabla \psi)$.

\begin{lemma}\label{l:submodularity infinity Laplacian bound}
If $L\in \R^{m\times m}$ is a Laplacian matrix and $p$ is a vector in $\R^m$ with $p_i\geq 0$ for each $i \in [m]$, then 
\begin{align*}
 -(Lp,p) \leq \max_i \{-L_{ii}\}|p|^2.
\end{align*}
\end{lemma}
\begin{proof}
Since $(Lp,p) = \sum_{ij}L_{ij}p_ip_j$ we use the fact that $p_i\geq 0$ and that $L$ is a Laplacian matrix to obtain
\begin{align*}
  -(Lp,p) = -\sum_{i}L_{ii}p_i^2-\sum_i\sum_{j\neq i}L_{ij}p_ip_j & \leq \max_i\{-L_{ii}\}|p|^2.  
\end{align*}

\end{proof}

\begin{lemma}\label{l:subsolution via heat equation}
  If $\psi$ is the solution to \eqref{e:heat equation factor 1/4} and $w$ the solution to \eqref{e:terminal value problem cube sec 7}, then
  \begin{align*}
    \psi(x,t) \leq w(x,t) \;\forall\;x\in\mathbb{R}^m,\; t<0.    
  \end{align*}
\end{lemma}

\begin{proof} 

As observed above $-D^2\psi(x,t)$ is a Laplacian matrix for every $x\in\mathbb{R}^m$ and $t<0$. Likewise, $\psi(x,t)$ is monotone in each variable $x_i$, and so $\partial_k \psi \geq 0$ for every $k$ and $\psi(x,t)$ is linear in the direction of the vector of all ones, and therefore  $\sum_k\partial_k \psi = C(E)$. Therefore,  $\nabla \psi \neq 0$ everywhere for $t<0$. In this case, we may apply Lemma \ref{l:submodularity infinity Laplacian bound} with $L= -D^2\psi$ and $p = \nabla \psi$ to obtain
\begin{align*}
  \frac{1}{|\nabla \psi|^2}(D^2\psi \nabla \psi,\nabla \psi) \leq \max_{k} \{ \partial_{kk}\psi \},\;\text{ for } t<0.
\end{align*}
 In particular, for all $x$ and for all $t<0$ we have
\begin{align}
  \Delta \psi - \frac{1}{|\nabla \psi|^2}(D^2\psi \nabla \psi,\nabla \psi) \geq \Delta \psi - \max_k\{ \partial_{kk}\psi\}.    \label{eq:mcf_bound}
\end{align}
On the other hand, using again that $-D^2 \psi$ is a Laplacian matrix, and therefore $\partial_{\ell j}\psi \leq 0$ for $\ell \neq j$, we have  
\begin{align*}
  \partial_{kk}\psi = \Delta \psi -\partial_{kk}\psi + \sum \limits_{\ell \neq k}\sum_{j\neq k,\ell} \partial_{\ell j}\psi \leq \Delta \psi-\partial_{kk}\psi. 
\end{align*}
In particular, since $\partial_{kk} \psi \geq 0$, the preceding inequality guarantees
\begin{align*}
  \Delta  \psi = \Delta \psi-\partial_{kk}\psi + \partial_{kk}\psi \leq 2 (\Delta \psi - \partial_{kk}\psi).    
\end{align*}
As this holds for every $k$, we have
\begin{align}
  \frac{1}{2}\Delta \psi \leq \Delta \psi - \max_k \partial_{kk}\psi  
  \label{eq:trace_bound}
\end{align}
Now, we put the above inequalities together: the standard identity for the MCF operator and the inequality  \eqref{eq:mcf_bound} provide that
\begin{align*}
  \partial_t\psi + \tfrac{1}{2}|\nabla \psi|\dive\left ( \frac{\nabla \psi}{|\nabla \psi|} \right ) =\partial_t\psi + \tfrac{1}{2}\Big( \Delta \psi - \frac{1}{|\nabla \psi|^2}(D^2\psi \nabla \psi,\nabla \psi) \Big)\geq \partial_t\psi  + \tfrac{1}{2}\left ( \Delta \psi -\max_k \partial_{kk}\psi\right ),
\end{align*}
and then the inequality \eqref{eq:trace_bound} yields
\begin{align*}
  \partial_t\psi + \tfrac{1}{2}|\nabla \psi|\dive\left ( \frac{\nabla \psi}{|\nabla \psi|} \right ) \geq \partial_t\psi  + \tfrac{1}{4}\Delta \psi = 0.
\end{align*}    
Finally, we obtain
\[
C(E) \max_i x_i = \psi(x,0)  \leq \|x\|_E
\] from the following inequality based on definition of the dual norm (see Section \ref{sec:notation}): 
\begin{align*}
\max_i x_i = \|x\|_E \max_i ( e_i, x/\|x\|_E ) \leq \|x\|_E \max_i \|e_i\|^*_E.
\end{align*}
We conclude that $\psi$ is a subsolution of \eqref{e:terminal value problem cube sec 7}, and thus $\psi \leq w$ by the comparison principle (Theorem \ref{t:comparison principle}).

\end{proof}

\begin{theorem}
    
The function $w(x,t)$ solving the problem \eqref{e:terminal value problem cube sec 7} satisfies the following bound for every $T>0$
\begin{align*}
C(E) \, \mathbb E[ \max _i G_i] \sqrt{T/2}\,   \leq  w(0,-T).
\end{align*} where $G$ denotes an $m$-dimensional standard Gaussian random vector with mean $0$ and identity covariance. 
\end{theorem}

\begin{proof}

From Lemma \ref{l:subsolution via heat equation}, we have that 
\begin{align*}
  \psi(0,-T) \leq w(0,-T).
\end{align*}
and since $\psi$ is a solution of the heat equation with diffusion constant $1/4$, we have 
\[
\psi(0,-T) =C(E) \mathbb E [\max _i G_i ]\, \sqrt{T/2}.
\]

\end{proof}

\subsection{The upper bound} As we mentioned earlier in the section, the corresponding upper bound for $w$ is more straightforward than the lower bound. In this case, we simply make use of the fact that $\overline{\psi}(x,t)$ solving the backwards heat equation with terminal data $\|x\|_E$ serves as a supersolution to \eqref{e:terminal value problem cube sec 7}. The value of $\overline{\psi}(x,t)$ is then the expected norm of a Gaussian vector.

We start with the supersolution property for $\overline{\psi}$. For $t>0$ and $x\in\mathbb{R}^m$ we define $\overline{\psi}$ by
\begin{align*}
  \overline{\psi}(x,-t) := \mathbb{E}[\|x+X_t\|_E],    
\end{align*}
where $X_t$ is the standard $m$-dimensional Brownian motion starting from $X_0 = 0$. In particular, $\overline{\psi}$ solves
\begin{align*}
  \partial_t \overline{\psi} + \tfrac{1}{2}\Delta \overline{\psi} & = 0 \text{ in } \mathbb{R}^m\times (-\infty,0),\\
  \overline{\psi}(x,0) & = \|x\|_E.
\end{align*}

\begin{proposition}\label{p:heat equation is a supersolution to MCF level set eqn}
  If $w(x,t)$ is the solution of \eqref{e:terminal value problem cube sec 7} and $\overline{\psi}$ is as defined above, then for every $t>0$ we have
  \begin{align*}
    w(x,-t) \leq \overline{\psi}(x,-t).
  \end{align*}
\end{proposition}

\begin{proof}

As a solution of the heat equation $\overline{\psi}(\cdot,-t)$ is given by the convolution of $\|x\|_E$ with the heat kernel at time $t$. In particular, $\overline{\psi}(\cdot,-t)$ is twice differentiable, and as $\|x\|_E$ is convex, so will $\overline{\psi}(\cdot,-t)$. 

Now, take a point $(x_0,t_0)$ with $t_0<0$ such that $\nabla \overline{\psi}(x_0,t_0) \neq 0$. Then at $(x_0,t_0)$ we have 
\begin{align*}
  \Delta \overline{\psi} & = |\nabla \overline{\psi}|\dive \left (\frac{\nabla \overline{\psi}}{|\nabla \overline{\psi}|} \right ) + \frac{1}{|\nabla \overline{\psi}|^2}((D^2\overline{\psi})\nabla \overline{\psi},\nabla \overline{\psi}) \geq |\nabla \overline{\psi}|\dive \left (\frac{\nabla \overline{\psi}}{|\nabla \overline{\psi}|} \right ),
\end{align*}
thanks to $((D^2\overline{\psi})\nabla \overline{\psi},\nabla \overline{\psi})\geq 0$ everywhere. This shows that
\begin{align*}
  \partial_t \overline{\psi} +\tfrac{1}{2}|\nabla \overline{\psi}|\dive\left ( \frac{\nabla \overline{\psi}}{|\nabla \overline{\psi}|} \right) \leq \partial_t \overline{\psi} + \tfrac{1}{2} \Delta \overline{\psi} = 0, \text{ in } \{ \nabla \overline{\psi}\neq 0\}.
\end{align*}
If $\nabla \overline{\psi}(x_0,t_0) = 0$, then at $(x_0,t_0)$ we have (recall $(H_n)_*$ from Lemma \ref{l:liminf and limsup for H delta n} and Remark \ref{r:half relaxed limits H n delta})
\begin{align*}
  (H_n)_*(D^2\overline{\psi},0) = \min_{\sigma\in\mathbb{S}^{m-1}} H_n(D^2\overline \psi,\sigma),
\end{align*}
Since $n\geq m-1$ and for any $\sigma \in \mathbb{S}^{m-1}$ and $M$ the matrix $\Pi(\sigma)M\Pi(\sigma)$ has at least zero as an eigenvalue, we have $H_n(M,\sigma) = \tr((\Pi(\sigma)M\Pi(\sigma))_+)$, and since $D^2\psi \geq 0$ we have $\Pi(\sigma)D^2\overline \psi\Pi(\sigma)\geq 0$. It follows that for
\begin{align*}
  H_n(D^2\overline\psi,\sigma) = \tr(\Pi(\sigma)D^2\overline \psi\Pi(\sigma)) = \Delta \overline \psi - ((D^2\overline \psi) \sigma,\sigma) \leq \Delta \overline \psi,
\end{align*}
for any $\sigma \in \mathbb{S}^{m-1}$. In particular, $(H_n)_*(D^2\overline \psi,0) \leq \Delta \overline \psi$. Therefore, at $(x_0,t_0)$ we have also
\begin{align*}
  \partial_t \overline{\psi} + \tfrac{1}{2}(H_n)_*(D^2\overline{\psi},0) \leq \partial_t \overline{\psi} +\tfrac{1}{2} \Delta \overline{\psi} = 0.  
\end{align*}
We conclude $\overline{\psi}$ is a supersolution for the problem solved by $w$ (we already have  $\overline{\psi} = w$ at $t=0$ by definition). Therefore, we may apply the comparison principle (Theorem \ref{t:comparison principle}) and conclude that $w(x,t)\leq \overline{\psi}(x,t)$ for all $x\in\mathbb{R}^m$ and $t<0$.

\end{proof}
Thus, Proposition \ref{p:heat equation is a supersolution to MCF level set eqn} provides the desired upper bound, 
the expected norm of an $m$-dimensional normal random variable is well understood: 
\begin{align*}
w(0,-T) \leq  \overline{\psi}(0, -T) =  \mathbb{E}\|X_t\|_E  = \sqrt {T}\, \mathbb{E}\|G\|_E
\end{align*} where $G$ denotes an $m$-dimensional standard Gaussian random vector with mean $0$ and identity covariance.

\subsection*{Acknowledgements} We honor the memory of Robert Kohn whose constant encouragement and generous sharing of ideas made this project possible.  V.A.K.  expresses his profound appreciation to Afonso Bandeira for the introduction to the Koml\'os conjecture, which was the initial spark for this project. V.A.K. is immensely grateful to his postdoctoral mentor, Maria Han Veiga, for her guidance and collaboration throughout his time at The Ohio State University. N.G. is indebted to Deane Yang and Shay Sadovsky for many helpful and interesting discussions. N.G. would like to acknowledge the generous support of the National Science Foundation through grant DMS-214423.

\bibliographystyle{plain}
\bibliography{Komlos_MCF}

@article{BKMZ24,
title = {On the concentration of {G}aussian {C}ayley matrices},
journal = {Applied and Computational Harmonic Analysis},
volume = {73},
pages = {101694},
year = {2024},
issn = {1063-5203},
doi = {https://doi.org/10.1016/j.acha.2024.101694},
url = {https://www.sciencedirect.com/science/article/pii/S106352032400071X},
author = {Afonso S. Bandeira and Dmitriy Kunisky and Dustin G. Mixon and Xinmeng Zeng}
}

@misc{calder24,
  author    = {Jeff Calder},
  title     = {Lecture notes on viscosity solutions}, 
  note = {Available at \url{https://www-users.cse.umn.edu/~jwcalder/viscosity_solutions.pdf}},
  year  = {2024}
}

@misc{silvestre15,
  author    = {Luis Silvestre},
  title     = {Viscosity solutions of elliptic equations}, 
  note = {Available at \url{https://www.math.uchicago.edu/~luis/preprints/viscosity-solutions.pdf}},
  year  = {2015}
}

@Inbook{GL26,
author="Yoshikazu Giga and Qing Liu",
series="Monographs in Mathematics",
title="Surface Evolution Equations:
A Level Set Approach",
year="2026",
publisher="Birkhäuser Cham",
isbn="978-3-032-20302-1",
edition={2}
}

@misc{bandeira2026matrixdiscrepancyrepresentationsfinite,
      title={Matrix Discrepancy for Representations of Finite Groups}, 
      author={Afonso S. Bandeira and Helmut Bölcskei},
      year={2026},
      eprint={2606.12181},
      archivePrefix={arXiv},
      primaryClass={math.PR},
      url={https://arxiv.org/abs/2606.12181}, 
}

@article{kohn2010,
author = {Robert V. Kohn and Sylvia Serfaty},
year = {2010},
pages = {1298--1350},
title = {A Deterministic-Control-Based Approach to Fully Nonlinear Parabolic and Elliptic Equations},
volume = {63},
journal = {Communications on Pure and Applied Mathematics},
}

@article{Spencer86,
title = {Balancing vectors in the max norm},
journal = {Combinatorica},
volume = {6},
issue ={1},
number = {1},
pages = {55-65},
year = {1986},
issn =  {1439-6912},
doi = {10.1007/BF02579409},
url = {https://doi.org/10.1007/BF02579409},
author = {Joel Spencer},
}

@article{MNT20,
    author = {Matoušek, Jiří and Nikolov, Aleksandar and Talwar, Kunal},
    title = {Factorization Norms and Hereditary Discrepancy},
    journal = {International Mathematics Research Notices},
    volume = {2020},
    number = {3},
    pages = {751-780},
    year = {2020},
    month = {02},
    issn = {1073-7928},
    doi = {10.1093/imrn/rny033},
    url = {https://doi.org/10.1093/imrn/rny033},
    eprint = {https://academic.oup.com/imrn/article-pdf/2020/3/751/32219164/rny033.pdf}
}

@InProceedings{BHMSS22,
  author =	{Bansal, Nikhil and Jiang, Haotian and Meka, Raghu and Singla, Sahil and Sinha, Makrand},
  title =	{{Prefix Discrepancy, Smoothed Analysis, and Combinatorial Vector Balancing}},
  booktitle =	{13th Innovations in Theoretical Computer Science Conference (ITCS 2022)},
  pages =	{13:1--13:22},
  series =	{Leibniz International Proceedings in Informatics (LIPIcs)},
  ISBN =	{978-3-95977-217-4},
  ISSN =	{1868-8969},
  year =	{2022},
  volume =	{215},
  editor =	{Braverman, Mark},
  publisher =	{Schloss Dagstuhl -- Leibniz-Zentrum f{\"u}r Informatik},
  address =	{Dagstuhl, Germany},
  URL =		{https://drops.dagstuhl.de/entities/document/10.4230/LIPIcs.ITCS.2022.13},
  URN =		{urn:nbn:de:0030-drops-156092},
  doi =		{10.4230/LIPIcs.ITCS.2022.13}
}

@article{GMRR26,
    title      = {A two-player zero-sum probabilistic game that approximates the mean curvature flow},
    author     = {Irene Gonzalvez and Alfredo Miranda and Julio D. Rossi and Jorge Ruiz-Cases},
    url        = {https://cm.episciences.org/15764},
    doi        = {10.46298/cm.15764},
    journal    = {Communications in Mathematics},
    issn       = {2336-1298},
    volume     = {Volume 34 (2026), Issue 2 (Special issue: Latin American mathematics)},
    eid        = 1,
    year       = {2025},
    month      = {Jul},
}

@article{IS11,
title = {Repeated games for
  non-linear parabolic integro-differential equations and integral
  curvature flows},
journal = {Discrete and Continuous Dynamical Systems},
volume = {29},
number = {4},
pages = {1517-1552},
year = {2011},
issn = {1078-0947},
doi = {10.3934/dcds.2011.29.1517},
url = {https://www.aimsciences.org/article/id/0bb94f94-7d20-4793-a2ca-823f09f1c7db},
author = {Cyril Imbert and Sylvia Serfaty}
}

@article {GuiSch2019,
    AUTHOR = {Guillen, Nestor and Schwab, Russell W.},
     TITLE = {Min-max formulas for nonlocal elliptic operators},
   JOURNAL = {Calc. Var. Partial Differential Equations},
  FJOURNAL = {Calculus of Variations and Partial Differential Equations},
    VOLUME = {58},
      YEAR = {2019},
    NUMBER = {6},
     PAGES = {Paper No. 209, 79},
      ISSN = {0944-2669,1432-0835},
   MRCLASS = {49N70 (35R09 45K05 47G20 49L25)},
  MRNUMBER = {4037461},
MRREVIEWER = {Rodica\ Luca},
       DOI = {10.1007/s00526-019-1631-z},
       URL = {https://doi-org.libproxy.txstate.edu/10.1007/s00526-019-1631-z},
}

@article {GuiSch2020,
    AUTHOR = {Guillen, Nestor and Schwab, Russell W.},
     TITLE = {Min-max formulas for nonlocal elliptic operators on
              {E}uclidean space},
   JOURNAL = {Nonlinear Anal.},
  FJOURNAL = {Nonlinear Analysis. Theory, Methods \& Applications. An
              International Multidisciplinary Journal},
    VOLUME = {193},
      YEAR = {2020},
     PAGES = {111468, 51},
      ISSN = {0362-546X,1873-5215},
   MRCLASS = {35R09 (35J60 45K05 46T99 47G20 49L25 49N70)},
  MRNUMBER = {4062972},
       DOI = {10.1016/j.na.2019.02.021},
       URL = {https://doi-org.libproxy.txstate.edu/10.1016/j.na.2019.02.021},
}

@misc{SV26,
      title={Discrepancy and {F}isher information}, 
      author={Gleb Smirnov and Roman Vershynin},
      year={2026},
      eprint={2605.13107},
      archivePrefix={arXiv},
      primaryClass={math.PR},
      url={https://arxiv.org/abs/2605.13107}, 
}

@inproceedings{NTZ13,
author = {Nikolov, Aleksandar and Talwar, Kunal and Zhang, Li},
title = {The geometry of differential privacy: the sparse and approximate cases},
year = {2013},
isbn = {9781450320290},
publisher = {Association for Computing Machinery},
address = {New York, NY, USA},
url = {https://doi.org/10.1145/2488608.2488652},
doi = {10.1145/2488608.2488652},
booktitle = {Proceedings of the Forty-Fifth Annual ACM Symposium on Theory of Computing},
pages = {351–360},
numpages = {10},
location = {Palo Alto, California, USA},
series = {STOC '13}
}

@misc{FJLN26,
      title={The Mean-Field Limit of Online Stochastic Vector Balancing}, 
      author={Christian Fiedler and Joe Jackson and Daniel Lacker and Jonathan Niles-Weed},
      year={2026},
      eprint={2605.14149},
      archivePrefix={arXiv},
      primaryClass={math.PR},
      url={https://arxiv.org/abs/2605.14149}, 
}

@misc{JKS19,
      title={Online Geometric Discrepancy for Stochastic Arrivals with Applications to Envy Minimization}, 
      author={Haotian Jiang and Janardhan Kulkarni and Sahil Singla},
      year={2019},
      eprint={1910.01073},
      archivePrefix={arXiv},
      primaryClass={cs.DS},
      url={https://arxiv.org/abs/1910.01073}, 
}

@inproceedings{BJMSS21,
author = {Nikhil Bansal and Haotian Jiang and Raghu Meka and Sahil Singla and Makrand Sinha},
title = {Online Discrepancy Minimization for Stochastic Arrivals},
booktitle = {Proceedings of the 2021 ACM-SIAM Symposium on Discrete Algorithms (SODA)},
chapter = {},
pages = {2842-2861},
doi = {10.1137/1.9781611976465.169},
URL = {https://epubs.siam.org/doi/abs/10.1137/1.9781611976465.169},
eprint = {https://epubs.siam.org/doi/pdf/10.1137/1.9781611976465.169},
}

@inproceedings{BJSS20,
author = {Bansal, Nikhil and Jiang, Haotian and Singla, Sahil and Sinha, Makrand},
title = {Online vector balancing and geometric discrepancy},
year = {2020},
isbn = {9781450369794},
publisher = {Association for Computing Machinery},
address = {New York, NY, USA},
url = {https://doi.org/10.1145/3357713.3384280},
doi = {10.1145/3357713.3384280},
booktitle = {Proceedings of the 52nd Annual ACM SIGACT Symposium on Theory of Computing},
pages = {1139–1152},
numpages = {14},
location = {Chicago, IL, USA},
series = {STOC 2020}
}

@book{AS2000,
author = {Noga Alon and Joel H. Spencer},
title = {The Probabilistic Method},
publisher = {John Wiley \& Sons, Inc.},
year = {2000},
doi={10.1002/0471722154}
}

@article{ES18,
author = {Eldan, Ronen and Singh, Mohit},
title = {Efficient algorithms for discrepancy minimization in convex sets},
journal = {Random Structures \& Algorithms},
volume = {53},
number = {2},
pages = {289-307},
doi = {https://doi.org/10.1002/rsa.20763},
url = {https://onlinelibrary.wiley.com/doi/abs/10.1002/rsa.20763},
eprint = {https://onlinelibrary.wiley.com/doi/pdf/10.1002/rsa.20763},
year = {2018}
}

@article{Rothvoss17,
author = {Rothvoss, Thomas},
title = {Constructive Discrepancy Minimization for Convex Sets},
journal = {SIAM Journal on Computing},
volume = {46},
number = {1},
pages = {224-234},
year = {2017},
doi = {10.1137/141000282},
URL = { 
https://doi.org/10.1137/141000282
},
eprint = { 
https://doi.org/10.1137/141000282
}

}

@INPROCEEDINGS{Bansal10,
  author={Bansal, Nikhil},
  booktitle={2010 IEEE 51st Annual Symposium on Foundations of Computer Science}, 
  title={Constructive Algorithms for Discrepancy Minimization}, 
  year={2010},
  volume={},
  number={},
  pages={3-10},
  doi={10.1109/FOCS.2010.7}}

@inproceedings{CNN11,
author = {Charikar, Moses and Newman, Alantha and Nikolov, Aleksandar},
title = {Tight hardness results for minimizing discrepancy},
year = {2011},
publisher = {Society for Industrial and Applied Mathematics},
address = {USA},
booktitle = {Proceedings of the Twenty-Second Annual ACM-SIAM Symposium on Discrete Algorithms},
pages = {1607–1614},
numpages = {8},
location = {San Francisco, California},
series = {SODA '11}
}

@article{Costello09,
  title={Balancing {G}aussian vectors},
  author={Costello, Kevin P.},
  journal={Israel Journal of Mathematics},
  volume={172},
  number={1},
  pages={145--156},
  year={2009},
url={https://doi.org/10.1007/s11856-009-0068-z},
doi= {10.1007/s11856-009-0068-z}
}

@article {QingSchikorraZhou2016,
    AUTHOR = {Liu, Qing and Schikorra, Armin and Zhou, Xiaodan},
     TITLE = {A game-theoretic proof of convexity-preserving properties for
              motion by curvature},
   JOURNAL = {Indiana Univ. Math. J.},
  FJOURNAL = {Indiana University Mathematics Journal},
    VOLUME = {65},
      YEAR = {2016},
    NUMBER = {1},
     PAGES = {171--197},
      ISSN = {0022-2518,1943-5258},
   MRCLASS = {49L25 (35J60 91A20)},
  MRNUMBER = {3466458},
MRREVIEWER = {Alp\'ar\ R.\ M\'esz\'aros},
       DOI = {10.1512/iumj.2016.65.5740},
       URL = {https://doi.org/10.1512/iumj.2016.65.5740},
}

@article {GigaGotoIshiiSato1991,
    AUTHOR = {Giga, Y. and Goto, S. and Ishii, H. and Sato, M.-H.},
     TITLE = {Comparison principle and convexity preserving properties for
              singular degenerate parabolic equations on unbounded domains},
   JOURNAL = {Indiana Univ. Math. J.},
  FJOURNAL = {Indiana University Mathematics Journal},
    VOLUME = {40},
      YEAR = {1991},
    NUMBER = {2},
     PAGES = {443--470},
      ISSN = {0022-2518,1943-5258},
   MRCLASS = {35B05 (35D05 35K65)},
  MRNUMBER = {1119185},
MRREVIEWER = {Maura\ Ughi},
       DOI = {10.1512/iumj.1991.40.40023},
       URL = {https://doi-org.libproxy.txstate.edu/10.1512/iumj.1991.40.40023},
}

@article {ChenGigaGoto1989,
    AUTHOR = {Chen, Yun Gang and Giga, Yoshikazu and Goto, Shun'ichi},
     TITLE = {Uniqueness and existence of viscosity solutions of generalized
              mean curvature flow equations},
   JOURNAL = {Proc. Japan Acad. Ser. A Math. Sci.},
  FJOURNAL = {Japan Academy. Proceedings. Series A. Mathematical Sciences},
    VOLUME = {65},
      YEAR = {1989},
    NUMBER = {7},
     PAGES = {207--210},
      ISSN = {0386-2194},
   MRCLASS = {35K15 (35K22)},
  MRNUMBER = {1030181},
MRREVIEWER = {W.\ P.\ Ziemer},
       URL = {http://projecteuclid.org.libproxy.txstate.edu/euclid.pja/1195512765},
}

@misc{altschuler2025thresholdonlinebalancingsparse,
      title={A threshold for online balancing of sparse i.i.d. vectors}, 
      author={Dylan J. Altschuler and Konstantin Tikhomirov},
      year={2025},
      eprint={2509.02432},
      archivePrefix={arXiv},
      primaryClass={math.PR},
      url={https://arxiv.org/abs/2509.02432}, 
}

@misc{albors2025structurebadsciencematrices,
      title={On the Structure of Bad Science Matrices}, 
      author={Alex Albors and Hisham Bhatti and Lukshya Ganjoo and Raymond Guo and Dmitriy Kunisky and Rohan Mukherjee and Alicia Stepin and Tony Zeng},
      year={2025},
      eprint={2408.00933},
      archivePrefix={arXiv},
      primaryClass={math.FA},
      url={https://arxiv.org/abs/2408.00933}, 
}

@article{crandall1992user,
  title={User’s guide to viscosity solutions of second order partial differential equations},
  author={Crandall, Michael G and Ishii, Hitoshi and Lions, Pierre-Louis},
  journal={Bulletin of the American mathematical society},
  volume={27},
  number={1},
  pages={1--67},
  year={1992}
}

@misc{steinerberger2024badsciencematrices,
      title={Bad Science Matrices}, 
      author={Stefan Steinerberger},
      year={2024},
      eprint={2402.03205},
      archivePrefix={arXiv},
      primaryClass={math.FA},
      url={https://arxiv.org/abs/2402.03205}, 
}

@misc{sinha2025structureextremalbadscience,
      title={The Structure of Extremal Bad Science Matrices}, 
      author={Shridhar Sinha},
      year={2025},
      eprint={2509.10580},
      archivePrefix={arXiv},
      primaryClass={math.FA},
      url={https://arxiv.org/abs/2509.10580}, 
}

@article{EL19,
author = {Ezra, Esther and Lovett, Shachar},
title = {On the {B}eck-{F}iala conjecture for random set systems},
journal = {Random Structures \& Algorithms},
volume = {54},
number = {4},
pages = {665-675},
doi = {https://doi.org/10.1002/rsa.20810},
url = {https://onlinelibrary.wiley.com/doi/abs/10.1002/rsa.20810},
eprint = {https://onlinelibrary.wiley.com/doi/pdf/10.1002/rsa.20810},
year = {2019}
}

@inproceedings{CV14,
author = {Chandrasekaran, Karthekeyan and Vempala, Santosh S.},
title = {Integer feasibility of random polytopes: random integer programs},
year = {2014},
isbn = {9781450326988},
publisher = {Association for Computing Machinery},
address = {New York, NY, USA},
url = {https://doi.org/10.1145/2554797.2554838},
doi = {10.1145/2554797.2554838},
booktitle = {Proceedings of the 5th Conference on Innovations in Theoretical Computer Science},
pages = {449–458},
numpages = {10},
location = {Princeton, New Jersey, USA},
series = {ITCS '14}
}

@article{APZ19,
doi = {10.1088/1751-8121/ab227a},
url = {https://doi.org/10.1088/1751-8121/ab227a},
year = {2019},
month = {June},
publisher = {IOP Publishing},
volume = {52},
number = {29},
pages = {294003},
author = {Aubin, Benjamin and Perkins, Will and Zdeborová, Lenka},
title = {Storage capacity in symmetric binary perceptrons},
journal = {Journal of Physics A: Mathematical and Theoretical}
}

@InProceedings{TMR20,
  title = 	 {Balancing {G}aussian vectors in high dimension},
  author =       {Turner, Paxton and Meka, Raghu and Rigollet, Philippe},
  booktitle = 	 {Proceedings of Thirty Third Conference on Learning Theory},
  pages = 	 {3455--3486},
  year = 	 {2020},
  editor = 	 {Abernethy, Jacob and Agarwal, Shivani},
  volume = 	 {125},
  series = 	 {Proceedings of Machine Learning Research},
  month = 	 {09--12 Jul},
  publisher =    {PMLR},
  url = 	 {https://proceedings.mlr.press/v125/turner20a.html}
}

@article{FS20,
author = {Franks, Cole and Saks, Michael},
title = {On the discrepancy of random matrices with many columns},
journal = {Random Structures \& Algorithms},
volume = {57},
number = {1},
pages = {64-96},
doi = {https://doi.org/10.1002/rsa.20909},
url = {https://onlinelibrary.wiley.com/doi/abs/10.1002/rsa.20909},
eprint = {https://onlinelibrary.wiley.com/doi/pdf/10.1002/rsa.20909},
year = {2020}
}

@inproceedings{HR19,
author = {Hoberg, Rebecca and Rothvoss, Thomas},
title = {A {F}ourier-analytic approach for the discrepancy of random set systems},
year = {2019},
publisher = {Society for Industrial and Applied Mathematics},
address = {USA},
booktitle = {Proceedings of the Thirtieth Annual ACM-SIAM Symposium on Discrete Algorithms},
pages = {2547–2556},
numpages = {10},
location = {San Diego, California},
series = {SODA '19}
}

@InProceedings{BJMSS22,
  author =	{Bansal, Nikhil and Jiang, Haotian and Meka, Raghu and Singla, Sahil and Sinha, Makrand},
  title =	{{Smoothed Analysis of the Koml\'{o}s Conjecture}},
  booktitle =	{49th International Colloquium on Automata, Languages, and Programming (ICALP 2022)},
  pages =	{14:1--14:12},
  series =	{Leibniz International Proceedings in Informatics (LIPIcs)},
  ISBN =	{978-3-95977-235-8},
  ISSN =	{1868-8969},
  year =	{2022},
  volume =	{229},
  editor =	{Boja\'{n}czyk, Miko{\l}aj and Merelli, Emanuela and Woodruff, David P.},
  publisher =	{Schloss Dagstuhl -- Leibniz-Zentrum f{\"u}r Informatik},
  address =	{Dagstuhl, Germany},
  URL =		{https://drops.dagstuhl.de/entities/document/10.4230/LIPIcs.ICALP.2022.14},
  URN =		{urn:nbn:de:0030-drops-163556},
  doi =		{10.4230/LIPIcs.ICALP.2022.14}
}

@misc{kunisky2025asymptoticboundsonlinealgorithms,
      title={Asymptotic Bounds and Online Algorithms for Average-Case Matrix Discrepancy}, 
      author={Dmitriy Kunisky and Timm Oertel and Nicola Wengiel and Peiyuan Zhang},
      year={2025},
      eprint={2410.23915},
      archivePrefix={arXiv},
      primaryClass={math.PR},
      url={https://arxiv.org/abs/2410.23915}, 
}

@inproceedings{Bansal22,
author = {Nikhil Bansal},
title = {Discrepancy theory and related algorithms},
year = {2022},
publisher = {International Mathematical Union},
url = {https://ems.press/content/book-chapter-files/33335},
doi = {10.4171/ICM2022/169},
booktitle = {Proc. Int. Cong. Math},
pages = {5178--5210},
volume={7}
}

@article{AN22,
author = {Altschuler, Dylan J. and Niles-Weed, Jonathan},
title = {The discrepancy of random rectangular matrices},
journal = {Random Structures \& Algorithms},
volume = {60},
number = {4},
pages = {551-593},
doi = {https://doi.org/10.1002/rsa.21054},
url = {https://onlinelibrary.wiley.com/doi/abs/10.1002/rsa.21054},
eprint = {https://onlinelibrary.wiley.com/doi/pdf/10.1002/rsa.21054},
year = {2022}
}

@article{CGRT22,
title = {Gaussian discrepancy: A probabilistic relaxation of vector balancing},
journal = {Discrete Applied Mathematics},
volume = {322},
pages = {123-141},
year = {2022},
issn = {0166-218X},
doi = {https://doi.org/10.1016/j.dam.2022.08.007},
url = {https://www.sciencedirect.com/science/article/pii/S0166218X22002992},
author = {Sinho Chewi and Patrik Gerber and Philippe Rigollet and Paxton Turner}
}

@article{ST04,
author = {Spielman, Daniel A. and Teng, Shang-Hua},
title = {Smoothed analysis of algorithms: Why the simplex algorithm usually takes polynomial time},
year = {2004},
issue_date = {May 2004},
publisher = {Association for Computing Machinery},
address = {New York, NY, USA},
volume = {51},
number = {3},
issn = {0004-5411},
url = {https://doi.org/10.1145/990308.990310},
doi = {10.1145/990308.990310},
journal = {J. ACM},
month = may,
pages = {385–463},
numpages = {79}
}

@article{HRS24,
author = {Haghtalab, Nika and Roughgarden, Tim and Shetty, Abhishek},
title = {Smoothed Analysis with Adaptive Adversaries},
year = {2024},
issue_date = {June 2024},
publisher = {Association for Computing Machinery},
address = {New York, NY, USA},
volume = {71},
number = {3},
issn = {0004-5411},
url = {https://doi.org/10.1145/3656638},
doi = {10.1145/3656638},
journal = {J. ACM},
month = jun,
articleno = {19},
numpages = {34}
}

@article{AHT25, 
title={Smoothed Analysis of the {K}omlós Conjecture: {R}ademacher Noise}, 
volume={32}, 
url={https://www.combinatorics.org/ojs/index.php/eljc/article/view/v32i1p52}, 
DOI={10.37236/13213}, 
number={1}, 
journal={The Electronic Journal of Combinatorics}, 
author={Aigner-Horev, Elad and Hefetz, Dan and Trushkin, Michael}, 
year={2025},
month={Mar.}, 
pages={P1.52} }

@inproceedings{ALS21,
author = {Alweiss, Ryan and Liu, Yang P. and Sawhney, Mehtaab},
title = {Discrepancy minimization via a self-balancing walk},
year = {2021},
isbn = {9781450380539},
publisher = {Association for Computing Machinery},
address = {New York, NY, USA},
url = {https://doi.org/10.1145/3406325.3450994},
doi = {10.1145/3406325.3450994},
booktitle = {Proceedings of the 53rd Annual ACM SIGACT Symposium on Theory of Computing},
pages = {14–20},
numpages = {7},
location = {Virtual, Italy},
series = {STOC 2021}
}

@misc{kobzar22,
  url = {https://arxiv.org/pdf/2202.05767.pdf},
  author = {Vladimir A. Kobzar and Robert V. Kohn},
  title = {A {PDE}-Based Analysis of the Symmetric Two-Armed {B}ernoulli Bandit},
  publisher = {arXiv},
  year = {2022}
}

@misc{kobzar26,
  author = {Vladimir A. Kobzar},
  title = {The Symmetric Two-Armed Bandit: the General Case},
  year = {2026}
}

@inproceedings{KJR24,
author = {Kulkarni, Janardhan and Reis, Victor and Rothvoss, Thomas},
title = {Optimal Online Discrepancy Minimization},
year = {2024},
isbn = {9798400703836},
publisher = {Association for Computing Machinery},
address = {New York, NY, USA},
url = {https://doi.org/10.1145/3618260.3649720},
doi = {10.1145/3618260.3649720},
booktitle = {Proceedings of the 56th Annual ACM Symposium on Theory of Computing},
pages = {1832–1840},
numpages = {9},
location = {Vancouver, BC, Canada},
series = {STOC 2024}
}

@article{dasgupta14,
    author = {Anirban DasGupta and S. N. Lahiri and Jordan Stoyanov},
    title = {Sharp fixed $n$ bounds and asymptotic expansions for the mean and the median of a {G}aussian sample maximum, and applications to the {D}onoho--{J}in model},
    journal = {Statistical Methodology},
    volume = {20},
    pages = {40-62},
    year = {2014}
}

@misc{orabona,
  author    = {Francesco Orabona and D{\'a}vid P{\'a}l},
  title     = {Optimal Non-Asymptotic Lower Bound on the Minimax Regret of Learning with Expert Advice}, 
  note = {Available at \url{https://arxiv.org/abs/1511.02176}},
  year  = {2015}
}

@misc{kamath98,
 author = {Gautam Kamath},
 title = {Bounds on the Expectation of the Maximum of Samples from a
{G}aussian},
 year = {1998},
 url = {http://www.gautamkamath.com/writings/gaussian_max.pdf},
 }

@article{BS20,
author = {Bansal, Nikhil and Spencer, Joel H.},
title = {On-line balancing of random inputs},
journal = {Random Structures \& Algorithms},
volume = {57},
number = {4},
pages = {879-891},
doi = {https://doi.org/10.1002/rsa.20955},
url = {https://onlinelibrary.wiley.com/doi/abs/10.1002/rsa.20955},
eprint = {https://onlinelibrary.wiley.com/doi/pdf/10.1002/rsa.20955},
year = {2020}
}

@book{Galambos87,
author = {Galambos, J.} ,
title = {Asymptotic Theory of Extreme Order Statistics},
publisher = {Wiley, New York},
year = {1987}
}

@article{BDGL19,
 author = {Bansal, Nikhil and Dadush, Daniel and Garg, Shashwat and Lovett, Shachar},
 title = {The {G}ram--{S}chmidt Walk: A Cure for the {B}anaszczyk Blues},
 year = {2019},
 pages = {1--27},
 doi = {10.4086/toc.2019.v015a021},
 publisher = {Theory of Computing},
 journal = {Theory of Computing},
 volume = {15},
 number = {21},
 URL = {https://theoryofcomputing.org/articles/v015a021},
}

@InProceedings{BG26,
author="Bednorz, Witold
and Godlewski, Piotr",
editor="Adamczak, Rados{\l}aw
and Gozlan, Nathael
and Lounici, Karim
and Madiman, Mokshay
and Merlev{\`e}de, Florence
and Werner, Elisabeth",
title="Some Remarks on the {G}ram-{S}chmidt Walk Algorithm and Consequences for the {K}oml{\'o}s Conjecture",
booktitle="High Dimensional Probability X",
year="2026",
publisher="Springer Nature Switzerland",
address="Cham",
pages="37--66",
isbn="978-3-032-06057-0"
}

@article {BarSou1991,
    AUTHOR = {Barles, G. and Souganidis, P. E.},
     TITLE = {Convergence of approximation schemes for fully nonlinear
              second order equations},
   JOURNAL = {Asymptotic Anal.},
  FJOURNAL = {Asymptotic Analysis},
    VOLUME = {4},
      YEAR = {1991},
    NUMBER = {3},
     PAGES = {271--283},
      ISSN = {0921-7134},
   MRCLASS = {35K55 (35A40 35J60 65M12)},
  MRNUMBER = {1115933},
}

@incollection {Pont1985,
    AUTHOR = {Pontryagin, L. S.},
     TITLE = {The mathematical theory of optimal processes and differential
              games},
      NOTE = {Topology, ordinary differential equations, dynamical systems},
   JOURNAL = {Trudy Mat. Inst. Steklov.},
  FJOURNAL = {Akademiya Nauk SSSR. Trudy Matematicheskogo Instituta imeni V.
              A. Steklova},
    VOLUME = {169},
      YEAR = {1985},
     PAGES = {119--158, 254--255},
      ISSN = {0371-9685},
   MRCLASS = {49-02 (49A45 90D25)},
  MRNUMBER = {836572},
MRREVIEWER = {Claudia\ Simionescu-Badea},
}

@book {Bellman1957,
    AUTHOR = {Bellman, Richard},
     TITLE = {Dynamic programming},
    SERIES = {Princeton Landmarks in Mathematics},
      NOTE = {Reprint of the 1957 edition,
              With a new introduction by Stuart Dreyfus},
 PUBLISHER = {Princeton University Press, Princeton, NJ},
      YEAR = {2010},
     PAGES = {xxx+340},
      ISBN = {978-0-691-14668-3},
   MRCLASS = {90-01 (49L20 90C39)},
  MRNUMBER = {2641641},
}

@article{courrege1965,
  title={Sur la forme int{\'e}gro-diff{\'e}rentielle des op{\'e}rateurs de $ C^\infty_k $ dans $ C $ satisfaisant au principe du maximum},
  author={Courrege, Philippe},
  journal={S{\'e}minaire de th{\'e}orie du potentiel},
  volume={10},
  number={1},
  pages={1--38},
  year={1965}
}

@article{peres2009tug,
  title={Tug-of-war and the infinity {L}aplacian},
  author={Peres, Yuval and Schramm, Oded and Sheffield, Scott and Wilson, David},
  journal={Journal of the American Mathematical Society},
  volume={22},
  number={1},
  pages={167--210},
  year={2009}
}

@article{charro2009mixed,
  title={A mixed problem for the infinity Laplacian via tug-of-war games},
  author={Charro, Fernando and Garc{\'\i}a Azorero, Jesus and Rossi, Julio D},
  journal={Calculus of Variations and Partial Differential Equations},
  volume={34},
  number={3},
  pages={307--320},
  year={2009},
  publisher={Springer}
}

@article{soner2002dynamic,
  title={Dynamic programming for stochastic target problems and geometric flows},
  author={Soner, H Mete and Touzi, Nizar},
  journal={Journal of the European Mathematical Society},
  volume={4},
  number={3},
  pages={201--236},
  year={2002}
}

@article{soner2003stochastic,
  title={A stochastic representation for mean curvature type geometric flows},
  author={Soner, H Mete and Touzi, Nizar},
  journal={The Annals of probability},
  volume={31},
  number={3},
  pages={1145--1165},
  year={2003},
  publisher={Institute of Mathematical Statistics}
}

@article{buckdahn2001representation,
  title={A representation formula for the mean curvature motion},
  author={Buckdahn, Rainer and Cardaliaguet, Pierre and Quincampoix, Marc},
  journal={SIAM Journal on Mathematical Analysis},
  volume={33},
  number={4},
  pages={827--846},
  year={2001},
  publisher={SIAM}
}

@article{larsson2024minimum,
  title={Minimum curvature flow and martingale exit times},
  author={Larsson, Martin and Ruf, Johannes},
  journal={Electronic Journal of Probability},
  volume={29},
  pages={1--32},
  year={2024},
  publisher={The Institute of Mathematical Statistics and the Bernoulli Society}
}

@book{bardi1997optimal,
  title={Optimal control and viscosity solutions of Hamilton-Jacobi-Bellman equations},
  author={Bardi, Martino and Dolcetta, Italo Capuzzo and others},
  volume={12},
  year={1997},
  publisher={Springer}
}

@book{fleming2006controlled,
  title={Controlled Markov processes and viscosity solutions},
  author={Fleming, Wendell H and Soner, H Mete},
  year={2006},
  publisher={Springer}
}

@book{brakke1978motion,
  title={The motion of a surface by its mean curvature},
  author={Brakke, Kenneth A.},
  series={Mathematical Notes},
  publisher ={Princeton Univ. Press, Princeton, NJ},
  year={1978}
}

@book{tonegawa2019brakke,
  title={Brakke's Mean Curvature Flow: An Introduction},
  author={Tonegawa, Yoshihiro},
  year={2019},
  publisher={Springer}
}

@article{andrews1994contraction,
  title={Contraction of convex hypersurfaces in {E}uclidean space},
  author={Andrews, Ben},
  journal={Calculus of Variations and Partial Differential Equations},
  volume={2},
  number={2},
  pages={151--171},
  year={1994},
  publisher={Springer}
}

@book{giga2006surface,
  title={Surface evolution equations: A level set approach},
  author={Giga, Yoshikazu},
  year={2006},
  publisher={Springer}
}

@book{osher2004level,
author = {Stanley Osher and Ronald Fedkiw},
title = {Level Set Methods and Dynamic Implicit Surfaces},
series = {Applied Mathematical Sciences},
publisher = {Springer New York, NY},
year = {2002},
doi = {10.1007/b98879}
}

@article{colding2015mean,
  title={Mean curvature flow},
  author={Colding, Tobias and Minicozzi, William and Pedersen, Erik and others},
  journal={Bulletin of the American Mathematical Society},
  volume={52},
  number={2},
  pages={297--333},
  year={2015}
}

@InProceedings{kobzar_geom, 
title = {New Potential-Based Bounds for the Geometric-Stopping Version of Prediction with Expert Advice}, 
author = {Kobzar, Vladimir A. and Kohn, Robert V. and Wang, Zhilei}, 
pages = {537--554}, 
year = {2020}, 
editor = {Jianfeng Lu and Rachel Ward}, 
volume = {107}, 
series = {Proceedings of Machine Learning Research}, 
address = {Princeton University, Princeton, NJ, USA}, 
month = {20--24 Jul}, 
publisher = {PMLR}, 
 booktitle   = {Proceedings of the 1st Annual Conference on Mathematical and Scientific Machine Learning}, 
url = {http://proceedings.mlr.press/v107/kobzar20a.html}, 
}

@misc{Lonke98,
 author = {Y. Lonke},
 title = {Combinatorial problems in finite dimensional normed spaces  ({PhD} dissertation, {H}ebrew University)},
 year = {1998},

 }

@article{Kashin85,
 author = {Kashin, B. S.},
 title = {On an isometric operator in {{\(L^ 2(0,1)\)}}},
 fjournal = {Doklady Bolgarsko{\u{\i}} Akademii Nauk (C.R. Acad Bulgare Sci.)},
 journal = {C. R. Acad. Bulg. Sci.},
 issn = {1310-1331},
 volume = {38},
 pages = {1613--1615},
 year = {1985},
 language = {Russian},
}

@article{bandeira2022remark,
  title = {A remark on {K}ashin’s discrepancy argument and partial coloring in the {K}oml{\'o}s conjecture},
  author = {Bandeira, Afonso S and Maillard, Antoine and Zhivotovskiy, Nikita},
  journal = {Portugaliae Mathematica},
  volume = {79},
  number = {3},
  pages = {311--316},
  year = {2022},
  doi = {10.4171/PM/2085}
}

@InProceedings{kohn07,
author = {Robert V. Kohn and Sylvia Serfaty},
  title =	{Second-Order {PDE}'s and Deterministic Games},
  booktitle =	{ICIAM 07 : 6th International Conference on Industrial and Applied Mathematics},
  year =	{2007},
 month = {July},
  address ={Zurich, Switzerland}
}

@book{Spe94,
author = {Spencer, Joel},
title = {Ten Lectures on the Probabilistic Method},
publisher = {Society for Industrial and Applied Mathematics},
year = {1994},
doi = {10.1137/1.9781611970074},
edition   = {2nd Edition},
URL = {https://epubs.siam.org/doi/abs/10.1137/1.9781611970074},
eprint = {https://epubs.siam.org/doi/pdf/10.1137/1.9781611970074}
}

@book{CST14,
  title={A Panorama of Discrepancy Theory},
  editor={William Chen and Anand Srivastav and Giancarlo Travaglini},
  url={https://doi.org/10.1007/978-3-319-04696-9},
  year={2014},
  publisher={Springer Cham}
}

@book{Chazelle2000, 
title={The Discrepancy Method: Randomness and Complexity}, 
publisher={Cambridge University Press}, 
author={Chazelle, Bernard}, 
year={2000}}

@book{Travaglini14, 
series={London Mathematical Society Student Texts}, 
title={Number Theory, Fourier Analysis and Geometric Discrepancy}, publisher={Cambridge University Press}, 
author={Travaglini, Giancarlo}, 
year={2014}, 
collection={London Mathematical Society Student Texts}}

@article{HSS24,
author = {Christopher Harshaw and Fredrik Sävje and Daniel A. Spielman and Peng Zhang},
title = {Balancing Covariates in Randomized Experiments with the {G}ram–{S}chmidt Walk Design},
journal = {Journal of the American Statistical Association},
volume = {119},
number = {548},
pages = {2934--2946},
year = {2024},
publisher = {Taylor \& Francis},
doi = {10.1080/01621459.2023.2285474},
URL ={https://doi.org/10.1080/01621459.2023.2285474},
eprint = { https://doi.org/10.1080/01621459.2023.2285474}
}

@book{Matousek99,
  title={Geometric Discrepancy},
  author={Jiří Matoušek},
  url={https://doi.org/10.1007/978-3-642-03942-3},
  year={1999},
  publisher={Springer Berlin, Heidelberg}
}

@article{Gunturk12,
author = {Güntürk, C. Sinan},
title = {Mathematics of Analog-to-Digital Conversion},
journal = {Communications on Pure and Applied Mathematics},
volume = {65},
number = {12},
pages = {1671-1696},
doi = {https://doi.org/10.1002/cpa.21425},
url = {https://onlinelibrary.wiley.com/doi/abs/10.1002/cpa.21425},
eprint = {https://onlinelibrary.wiley.com/doi/pdf/10.1002/cpa.21425},
year = {2012}
}

@article{Banaszczyk1993,
author = {Banaszczyk, Wojciech},
journal = {Studia Mathematica},
language = {eng},
number = {1},
pages = {93-100},
title = {Balancing vectors and convex bodies},
url = {http://eudml.org/doc/216005},
volume = {106},
year = {1993},
}

@article{Giannopoulos1997,
author = {Giannopoulos, Apostolos},
journal = {Studia Mathematica},
language = {eng},
number = {3},
pages = {225-234},
title = {On some vector balancing problems},
url = {http://eudml.org/doc/216373},
volume = {122},
year = {1997},
}

@article{hajela88,
  title={On a conjecture of {K}omlos about signed sums of vectors inside the sphere},
  author={Hajela, D},
  journal={European Journal of Combinatorics},
  volume={9},
  number={1},
  pages={33--37},
  year={1988},
  publisher={Academic Press Ltd. London, UK, UK}
}

@article{CS21,
url = {https://doi.org/10.1515/advgeom-2019-0030},
title = {A note on norms of signed sums of vectors},
author = {Giorgos Chasapis and Nikos Skarmogiannis},
pages = {5--14},
volume = {21},
number = {1},
journal = {Advances in Geometry},
doi = {doi:10.1515/advgeom-2019-0030},
year = {2021},
}

@article{BDG19,
author = {Bansal, Nikhil and Dadush, Daniel and Garg, Shashwat},
title = {An Algorithm for {K}omlós Conjecture Matching {B}anaszczyk's Bound},
journal = {SIAM Journal on Computing},
volume = {48},
number = {2},
pages = {534-553},
year = {2019},
doi = {10.1137/17M1126795},
URL = { https://doi.org/10.1137/17M1126795},
eprint = {         https://doi.org/10.1137/17M1126795}
}

@article{LM15,
author = {Lovett, Shachar and Meka, Raghu},
title = {Constructive Discrepancy Minimization by Walking on the Edges},
journal = {SIAM Journal on Computing},
volume = {44},
number = {5},
pages = {1573-1582},
year = {2015},
doi = {10.1137/130929400},
URL = {https://doi.org/10.1137/130929400},
eprint = {https://doi.org/10.1137/130929400}
}

@article{Glushkin89,
title = {Extremal properties of orthogonal parallelepipeds and their applications to the geometry of {B}anach spaces},
journal = {Mathematics of the USSR-Sbornik},
volume = {64},
number = {1},
pages = {85-96},
year = {1989},
doi = {https://doi.org/10.1070/SM1989v064n01ABEH003295},
url = {https://doi.org/10. 1070/SM1989v064n01ABEH003295},
author = {Efim Davydovich Gluskin},
}

@article{Spencer77,
title = {Balancing games},
journal = {Journal of Combinatorial Theory, Series B},
volume = {23},
number = {1},
pages = {68-74},
year = {1977},
issn = {0095-8956},
doi = {https://doi.org/10.1016/0095-8956(77)90057-0},
url = {https://www.sciencedirect.com/science/article/pii/0095895677900570},
author = {Joel Spencer},
}

@InProceedings{BLV22,
  author =	{Bansal, Nikhil and Laddha, Aditi and Vempala, Santosh},
  title =	{{A Unified Approach to Discrepancy Minimization}},
  booktitle =	{Approximation, Randomization, and Combinatorial Optimization. Algorithms and Techniques (APPROX/RANDOM 2022)},
  pages =	{1:1--1:22},
  series =	{Leibniz International Proceedings in Informatics (LIPIcs)},
  ISBN =	{978-3-95977-249-5},
  ISSN =	{1868-8969},
  year =	{2022},
  volume =	{245},
  editor =	{Chakrabarti, Amit and Swamy, Chaitanya},
  publisher =	{Schloss Dagstuhl -- Leibniz-Zentrum f{\"u}r Informatik},
  address =	{Dagstuhl, Germany},
  URL =		{https://drops.dagstuhl.de/entities/document/10.4230/LIPIcs.APPROX/RANDOM.2022.1},
  URN =		{urn:nbn:de:0030-drops-171238},
  doi =		{10.4230/LIPIcs.APPROX/RANDOM.2022.1}
}

@InProceedings{LRR17,
author="Levy, Avi
and Ramadas, Harishchandra
and Rothvoss, Thomas",
editor="Eisenbrand, Friedrich
and Koenemann, Jochen",
title="Deterministic Discrepancy Minimization via the Multiplicative Weight Update Method",
booktitle="Integer Programming and Combinatorial Optimization",
year="2017",
publisher="Springer International Publishing",
address="Cham",
pages="380--391",
isbn="978-3-319-59250-3"
}

@article{Ban98,
author = {Banaszczyk, Wojciech},
title = {Balancing vectors and {G}aussian measures of n-dimensional convex bodies},
journal = {Random Structures \& Algorithms},
volume = {12},
number = {4},
pages = {351-360},
doi = {https://doi.org/10.1002/(SICI)1098-2418(199807)12:4<351::AID-RSA3>3.0.CO;2-S},
year = {1998}
}

@article{BECK19811,
title = {“{I}nteger-making” theorems},
journal = {Discrete Applied Mathematics},
volume = {3},
number = {1},
pages = {1-8},
year = {1981},
issn = {0166-218X},
doi = {https://doi.org/10.1016/0166-218X(81)90022-6},
url = {https://www.sciencedirect.com/science/article/pii/0166218X81900226},
author = {J{\'o}zsef  Beck and Tibor Fiala},

}

@inproceedings{BJ2025,
author = {Nikhil Bansal and Haotian Jiang},
title = {Decoupling via Affine Spectral-Independence: {B}eck-{F}iala and {K}oml\'os Bounds Beyond {B}anaszczyk},
year = {2026},
publisher = {Association for Computing Machinery},
address = {New York, NY, USA},
url = {https://arxiv.org/abs/2508.03961},
booktitle = {Proceedings of the 58th Annual ACM SIGACT Symposium on Theory of Computing, to appear},
series = {STOC 2026}
}

@article{Spe85,
 ISSN = {00029947},
 URL = {http://www.jstor.org/stable/2000258},
 author = {Joel Spencer},
 journal = {Transactions of the American Mathematical Society},
 number = {2},
 pages = {679--706},
 publisher = {American Mathematical Society},
 title = {Six Standard Deviations Suffice},
 volume = {289},
 year = {1985}
}

@article{Kun23,
author = {Kunisky, Dmitriy},
title = {The Discrepancy of Unsatisfiable Matrices and a Lower Bound for the {K}omlós Conjecture Constant},
journal = {SIAM Journal on Discrete Mathematics},
volume = {37},
number = {2},
pages = {586-603},
year = {2023},
doi = {10.1137/22M149819X},
URL = { https://doi.org/10.1137/22M149819X},
eprint = { https://doi.org/10.1137/22M149819X}
}

@article{Milman09,
author = {Emanuel Milman},
title = {On the role of convexity in isoperimetry, spectral gap and concentration},
volume = {177},
journal = {Inventiones mathematicae},
pages = {1 -- 43},
year = {2009},
url={https://doi.org/10.1007/s00222-009-0175-9}
}

@misc{Cheeger70,
 author = {Cheeger, J.},
 title = {A lower bound for the smallest eigenvalue of the {Laplacian}},
 year = {1970},
 language = {English},
 howpublished = {Probl. {Analysis}, {Sympos}. in {Honor} of {Salomon} {Bochner}, {Princeton} {Univ}. 1969, 195-199 (1970).},
 zbMATH = {3337135},
 Zbl = {0212.44903}
 }

@InProceedings{kobzar, 
title = {New Potential-Based Bounds for Prediction with Expert Advice}, 
author = {Kobzar, Vladimir A. and Kohn, Robert V. and Wang, Zhilei}, 
pages = {2370--2405}, 
year = {2020}, 
editor = {Jacob Abernethy and Shivani Agarwal}, 
volume = {125}, 
series = {Proceedings of Machine Learning Research}, 
address = {}, 
booktitle =  {Proceedings of the 33rd Annual Conference on Learning Theory (COLT)},
month = {09--12 Jul}, 
publisher = {PMLR}, 
url = {http://proceedings.mlr.press/v125/kobzar20a.html} 
}

@article{Banaszczyk1990,
  title={A {B}eck—{F}iala-type theorem for {E}uclidean norms},
  author={Banaszczyk, Wojciech},
  journal={European Journal of Combinatorics},
  volume={11},
  number={6},
  pages={497--500},
  year={1990},
  publisher={Elsevier}
}

@incollection{EvansSpruck1991,
  title={Motion of level sets by mean curvature. {I}},
  author={Evans, Lawrence C and Spruck, Joel},
  booktitle={Fundamental Contributions to the Continuum Theory of Evolving Phase Interfaces in Solids: A Collection of Reprints of 14 Seminal Papers},
  pages={328--374},
  year={1991},
  publisher={Springer}
}

@article{GigaYamauchi1993,
  title={On a lower bound for the extinction time of surfaces moved by mean curvature},
  author={Giga, Yoshikazu and Yama-uchi, Kazuyuki},
  journal={Calculus of Variations and Partial Differential Equations},
  volume={1},
  number={4},
  pages={417--428},
  year={1993},
  publisher={Springer}
}

@article{Huisken1984,
  title={Flow by mean curvature of convex surfaces into spheres},
  author={Huisken, Gerhard},
  journal={Journal of Differential Geometry},
  volume={20},
  number={1},
  pages={237--266},
  year={1984},
  publisher={Lehigh University}
}

@article{KohnSerfaty2006,
  title={A deterministic-control-based approach to motion by curvature},
  author={Kohn, Robert and Serfaty, Sylvia},
  journal={Communications on Pure and Applied Mathematics: A Journal Issued by the Courant Institute of Mathematical Sciences},
  volume={59},
  number={3},
  pages={344--407},
  year={2006},
  publisher={Wiley Online Library}
}

@inproceedings{PesentiVladu2023,
  title={Discrepancy minimization via regularization},
  author={Pesenti, Lucas and Vladu, Adrian},
  booktitle={Proceedings of the 2023 Annual ACM-SIAM Symposium on Discrete Algorithms (SODA)},
  pages={1734--1758},
  year={2023},
  organization={SIAM}
}

\appendix

\section{Proofs of some standard propositions}\label{a:proof of some standard proposition}

In this Appendix, we record the proofs of several propositions used in the paper. These propositions are relatively standard and might distract from the main ideas in the paper, so their proofs have been relegated here. In what follows, we review the proofs of Proposition \ref{p:consistency of discretization}, Lemma \ref{l:half relaxed limits are sub or super solutions}, and also state and prove Proposition \ref{p:stability of touching by test functions}. 

\begin{proof}[Proof of Proposition \ref{p:consistency of discretization}]

  We are  writing down the formula for the remainder term in a Taylor expansion, which we review for completeness and clarity on how the error term depends on $\rho$. Observe that
  \begin{align*}
    \partial_t^{-h}\phi(x,t)-\partial_t\phi(x,t)= \tfrac{1}{h}(\phi(x,t)-\phi(x,t-h)-h\partial_t \phi(x,t)).    
  \end{align*} 
  Therefore, the first half of the proposition is equivalent to the statement
  \begin{align*}
    |\phi(x,t)-\phi(x,t-h)-h\partial_t \phi(x,t)| & \leq \rho(h)h.
  \end{align*}
  Indeed, integrating $\partial_t \phi$ we have
  \begin{align*}
    \phi(x,t) & = \phi(x,t-h) + \int_{t-h}^t \partial_s \phi(x,s) ds\\
      & =  \phi(x,t-h) + h \partial_t\phi(x,t) + \int_{t-h}^t ((\partial_s \phi)(x,s)-(\partial_s \phi)(x,t))\;ds.
  \end{align*}
  Since $\rho(\cdot)$ is a modulus of continuity for $\partial_t \phi$, the integral is on the right has a modulus not larger than $h \rho(h)$, from where the first bound follows.  Next, we look at the spatial derivatives, let $z \in \mathbb{R}^m$, then 
  \begin{align*}
    \nabla \phi(x+z,t) & = \nabla \phi(x,t)+\int_0^1\frac{d}{ds}\nabla \phi(x+sz,t)\;ds\\
      & = \nabla \phi(x,t)+ \int_0^1 D^2\phi(x+sz,t) z\;ds\\
      & = \nabla \phi(x,t)+ D^2\phi(x,t)z + \int_0^1 (D^2\phi(x+sz,t)-D^2\phi(x,t))z\;ds.
  \end{align*}
  Integrating again, we have
  \begin{align*}
    & \phi(x+ z,t)-\phi(x,t)-\nabla \phi(x,t)\cdot z\\
    & = \int_0^1 (\nabla \phi(x+s_1z,t)-\nabla \phi(x,t))\cdot z\;ds_1\\
    & = \int_0^1 (D^2\phi(x,t)z)\cdot z s_2\;ds_2+\int_0^1\int_0^1 (D^2\phi(x+s_1s_2z,t)-D^2\phi(x,t))z\cdot z\;ds_1ds_2.
  \end{align*}
  Rearranging, we have 
  \begin{align*}
    & \phi(x+ z,t)-\phi(x,t)-\nabla \phi(x,t)\cdot z-\tfrac{1}{2}D^2\phi(x,t)z\cdot z\\
    & = \int_0^1\int_0^1 (D^2\phi(x+s_1s_2z,t)-D^2\phi(x,t))z\cdot z\;ds_1ds_2.
  \end{align*}
  The integrand in the last double integral has absolute value no larger than $\omega(|z|)|z|^2$, and so
  \begin{align*}
    |\phi(x+z) - \phi(x)  -  z\cdot \nabla \phi(x) - \tfrac{1}{2}(D^2\phi(x) z,z)) | \leq \rho(|z|)|z|^2.
  \end{align*}
  We apply this inequality by taking $z = \delta A\overline \varepsilon$ and use it to bound the difference
  \begin{align*}
    \left | \max_A \min_{\overline\varepsilon}\delta^{-2}(\phi(x+\delta A\overline \varepsilon)-\phi(x)) - \tfrac{1}{2}\max_A \min_{\overline\varepsilon}\left \{ 2\delta^{-1} A\overline \varepsilon \cdot \nabla \phi(x) + (D^2\phi(x)A\overline \varepsilon,A\overline \varepsilon)\right \} \right |,
  \end{align*}
  by the maximum of $\rho(\delta \|A\overline \varepsilon\|_2)\|A\overline\varepsilon\|_2^2$ over all admissible $A$ and $\overline\varepsilon$. As the difference is simply $\mathcal{H}_{\delta,n}(\phi)(x)$ minus $H_{\delta,n}(D^2\phi(x),\nabla \phi(x))$, we have 
  \begin{align*}
    \left | \mathcal{H}_{\delta,n}(\phi)(x) - H_{\delta,n}(D^2\phi(x),\nabla \phi(x))  \right | \leq 2\max_{A,\overline\varepsilon}\left \{ \rho(\delta \|A\overline\varepsilon\|_2)\|A\overline\varepsilon\|_2^2 \right \} \leq 2\rho(n\delta )n^2
  \end{align*}

\end{proof}

\begin{proof}[Proof of Lemma \ref{l:half relaxed limits are sub or super solutions}]
  Consider first the case where we have a smooth function $\phi$ such that $w^*-\phi$ has a strict maximum at some $(x_0,t_0)$. Then, it can be shown (see Proposition \ref{p:stability of touching by test functions} below there are sequences $\delta_k \to 0^+$ and $(x_k,t_k)$ with $(x_k,t_k) \to (x_0,t_0)$ and $t_k \in \delta_k^2\mathbb{Z}_{-}$ such that $w_{\delta_k}-\phi$ has a local maximum at $(x_k,t_k)$ for every $k$. In particular, we have the inequality 
  \begin{align}\label{e:appendix test function discrete subsolution property}
    \partial_t^{-\delta_k^2}\phi(x_k,t_k+\delta^2) + \tfrac{1}{2}\mathcal{H}_{\delta_k,n}(\phi)(x_k,t_k+\delta^2) \geq 0.
  \end{align}
  Indeed, since $w_{\delta_k}-\phi$ has a local maximum at $(x_k,t_k)$, we have
\begin{align*}
    \phi(x',t')-\phi(x_k,t_k) \geq w_{\delta_k}(x',t')-w_{\delta_k}(x_k,t_k),    
  \end{align*}
  for any $x'\in\mathbb{R}^m,\;t'\in\delta^2_k\mathbb{Z}_-$. In particular, for every $A$ and $\overline \varepsilon$ we have
  \begin{align*}
    \phi(x_k+\delta_k A\overline \varepsilon,t_k+\delta_k^2)-\phi(x_k,t_k) \geq w_{\delta_k}(x_k+\delta_k A\overline \varepsilon,t_k+\delta_k^2)-w_{\delta_k}(x_k,t_k). 
  \end{align*}
  Taking the minimum for all admissible $\overline \varepsilon$ followed by the maximum over all admissible $A$ it follows that 
  \begin{align*}
    & -\phi(x_k,t_k)+ \max_A \min_{\overline \varepsilon}\phi(x_k+\delta_k A\overline \varepsilon,t_k+\delta_k^2) \\
    & \geq -w_{\delta_k}(x_k,t_k) + \max_A \min_{\overline \varepsilon} w_{\delta_k}(x_k+\delta_k A\overline \varepsilon,t_k+\delta_k^2) = 0.
  \end{align*}
  Then, from the definition of $\partial_t^{-\delta_k^2}$ and $\mathcal{H}_{\delta_k,n}$ (see Section \ref{s:Banaszczyk's theorem and the limiting Bellman operator}) we have 
  \begin{align*}
    \partial_t^{-\delta_k^2}\phi(x_k,t_k+\delta_k^2) + \tfrac{1}{2}\mathcal{H}_{\delta_k,n}(\phi)(x_k,t_k+\delta)  \geq 0,
  \end{align*}
  and thus we obtain the inequality in \eqref{e:appendix test function discrete subsolution property}. Now, with $\rho(\cdot)$ denoting a modulus of continuity for $\partial_t\phi$ and $D^2\phi$ we apply Proposition \ref{p:consistency of discretization} and obtain
  \begin{align*}
    |\partial_t^{-\delta_k^2}\phi(x_k,t_k+\delta_k^2) - \partial_t\phi(x_k,t_k+\delta_k^2)| \leq \rho(\delta_k^2),  
  \end{align*}
  as well as
  \begin{align*}
    |\mathcal{H}_{\delta_k,n}(\phi)(x_k,t_k+\delta_k^2) - H_{\delta_k,n}(D^2\phi,\nabla \phi)(x_k,t_k+\delta_k^2)| \leq 2\rho(\delta_k n)n^2.  
  \end{align*}
  In particular, 
  \begin{align*}
  \partial_t\phi(x_k,t_k+\delta_k^2)+\tfrac{1}{2}H_{\delta,n}(D^2\phi(x_k,t_k+\delta_k^2),\nabla \phi(x_k,t_k+\delta_k^2)) \geq - \rho(\delta_k^2)-2\rho(\delta_kn)n^2.
  \end{align*}
  We now use Lemma \ref{l:liminf and limsup for H delta n} which states that 
  \begin{align*}
    \limsup_k H_{\delta_k,n}(M_k,p_k) \leq (H_n)^*(M,p),
  \end{align*}
  and conclude that 
  \begin{align*}
    \partial_t\phi(x_0,t_0)+\tfrac{1}{2}(H_n)^*(D^2\phi(x_0,t_0),\nabla \phi(x_0,t_0)) \geq - \lim_{\delta \to 0^+}\{ \rho(\delta^2)+2\rho(\delta n)n^2\} = 0.
  \end{align*}
  This shows $w^*$ is a viscosity subsolution. Now suppose $\phi$ is a smooth function such that $w_*-\phi$ has a local minimum at $(x_0,t_0)$. As before, we can produce a sequence $\delta_k\to 0^+$ and $(x_k,t_k)$ with $t_k \in \delta_k^2\mathbb{Z}_-$ such that $w_{\delta_k}-\phi$ has a local minimum at $(x_k,t_k)$, in which case
  \begin{align*}
    \partial_t^{-\delta_k^2}\phi(x_k,t_k+\delta_k^2) + \tfrac{1}{2}\mathcal{H}_{\delta_k,n}(\phi)(x_k,t_k+\delta_k^2) \leq 0.    
  \end{align*}
  Again as before, we use Proposition \ref{p:consistency of discretization} and conclude that
  \begin{align*}
    \partial_t\phi(x_k,t_k+\delta_k^2)+\tfrac{1}{2}H_{\delta,n}(D^2\phi(x_k,t_k+\delta_k^2),\nabla \phi(x_k,t_k+\delta_k^2)) \leq \rho(\delta_k^2)+2\rho(\delta_kn)n^2,
  \end{align*}
  then, according to Lemma \ref{l:liminf and limsup for H delta n} we have 
    \begin{align*}
    \liminf_k H_{\delta_k,n}(M_k,p_k) \geq (H_n)_*(M,p),
  \end{align*}
  and we conclude that 
  \begin{align*}
    \partial_t\phi(x_0,t_0) + \tfrac{1}{2}(H_{n})_*(D^2\phi(x_0,t_0),\nabla \phi(x_0,t_0)) \leq 0.
  \end{align*}
  To finish the proof of the lemma it remains to consider the case where $\phi$ is such that $w^*-\phi$ has a maximum (or minimum) at some $(x_0,t_0)$ that is not necessarily strict. We explain the procedure for the case of the maximum. In this case, we consider for $\beta>0$ and small the function
  \begin{align*}
    \phi_\beta(x,t) = \phi(x,t) + \beta(|x-x_0|^2+(t-t_0)^2).    
  \end{align*}
  Then the function $w^*-\phi_\beta = w^*-\phi - \beta ( |x-x_0|^2+(t-t_0)^2)$ has a strict maximum at $(x_0,t_0)$, and so by the previous discussion we have
  \begin{align*}
    \partial_t\phi_\beta(x_0,t_0) + \tfrac{1}{2}(H_{n})^*(D^2\phi_\beta(x_0,t_0),\nabla \phi_\beta(x_0,t_0)) \geq 0.
  \end{align*}
  This holds for every $\beta>0$, and taking $\beta \to 0$ we obtain the corresponding inequality for $\phi$. The case where $w_*-\phi$ has a not-necessarily strict minimum is handled similarly, and this completes the proof of the lemma.
  
\end{proof}

The following proposition is a straightforward argument in real analysis, and it is a modification of a basic argument seen often in the viscosity solutions literature. It is used in Section \ref{s:continuous limit} to show the viscosity subsolution/supersolution property for the half-relaxed limits of $w_\delta$.
\begin{proposition}\label{p:stability of touching by test functions}
  Suppose $u_k:\mathbb{R}^m\times \delta_k^2 \mathbb{Z}_- \to \mathbb{R}$ is a sequence of locally uniformly bounded functions continuous in $x$, and denote by $u^*$ their upper semi-continuous limit. Then, if $\phi$ is a smooth function such that $u^*-\phi$ has a strict local maximum at $(x_{\max},t_{\max})$ we can find a sequence $\{(x_k,t_k)\}$ with $(x_k,t_k) \in \mathbb{R}^m\times \delta_k^2 \mathbb{Z}_-$ for every $k$  such that $x_k\to x_{\text{max}},t_k \to t_{\text{max}}$ and $u_k - \phi$ has a local maximum at $(x_{k},t_k)$. 
    
\end{proposition}

\begin{remark} A respective version of the proposition holds if one instead considers the lower semi-continuous limit and a smooth function $\phi$ such that $u_*-\phi$ has a strict local minimum -- one simply needs to apply the proposition above to the sequence $-u_k$. 
\end{remark}

\begin{proof}
  By assumption there exists $\epsilon_0>0$ so that $(x_{\max},t_{\max})$ is the unique maximum of $u^*-\phi$ in 
  \begin{align*}
    D := \{ (x,t): |x-x_{\max}|^2+|t-t_{\max}|^2\leq \epsilon_0^2\}.    
  \end{align*} 
  Given the continuity of each $u_k$ in $x$, there is a maximum point $(x_k,t_k)$ for $u_k-\phi$ in the compact set\footnote{There are only finitely many possible values of $t$ in this set, so one only need to take the largest maximum value among the finitely many time slices.}
  \begin{align*}
    D_k:= \left \{ (x,t) : |x-x_{\text{max}}|^2+|t-t_{\text{max}}|^2 \leq \epsilon_0^2,\; t \in \delta^2_k \mathbb{Z}_- \right \}. 
  \end{align*} 
  By compactness, the sequence $\{(x_k,t_k)\}$ has a non-empty set of limit points, and again by compactness the sequence will actually converge if there is exactly one limit point. Suppose $(x_*,t_*)$ is such a limit point, we are going to show $(x_*,t_*) =  (x_{\max},t_{\max})$ and thus prove the proposition.   
  
  Consider a sequence $(\hat x_k,\hat t_k)$  such that $(\hat x_k,\hat t_k) \in D_k$, $(\hat x_k,\hat t_k) \to (x_{\max},t_{\max})$ as $k \to \infty$, and
  \begin{align*}
    \lim_ku_k(\hat x_k,\hat t_k) = u^*(x_{\max},t_{\max}).  
  \end{align*}
  Then, since $(x_k,t_k)$ achieves the maximum of $u_k-\phi$ in $D_k$, we have
  \begin{align*}
    u_k(\hat x_k,\hat t_k)-\phi(\hat x_k,\hat t_k) \leq u_k(x_k,t_k)-\phi(x_k,t_k)\;\forall\;k.  
  \end{align*}
  Since $(x_*,t_*)$ is a limit point of $\{(x_k,t_k)\}_k$ there is a sequence $j_k \to \infty$ as $k\to \infty$ such that $(x_{j_k},t_{j_k}) \to (x_{*},t_{*})$. In this case, we may pass to the limit in the last inequality and conclude that 
  \begin{align*}
    u^*(x_{\max},t_{\max})-\phi(x_{\max},t_{\max}) \leq u^*(x_{*},t_{*})-\phi(x_{*},t_{*}).    
  \end{align*}  
  This means $(x_*,t_*)$ achieves the maximum of $u^*-\phi$ in $D$, and so $(x_{*},t_{*}) = (x_{\max},t_{\max})$ as we wanted, finishing the proof. 

\end{proof}
\section{Basics of viscosity solutions}\label{a:basics of viscosity solutions}

In this Appendix, we recall some fundamental definitions and results from the viscosity solution theory for the equation
\begin{align}\label{e:appendix PDE}
  \partial_t w + \tfrac{1}{2}H_n(D^2w,\nabla w) = 0 \text{ in } \mathbb{R}^m\times (-T,0),    
\end{align}
posed for some $T>0$. 

The strength of the viscosity solutions theory resides in the comparison principle. Roughly speaking, if $u,v$ are two viscosity solutions and $u\leq v$ at $t=0$, then $u\leq v$ for all times. This comparison property holds if $u,v$ do not satisfy the PDE but rather satisfy respective partial differential inequalities -- this is what brings us to the notion of subsolution and supersolution. Moreover, viscosity solutions deal with operators satisfying a monotonicity property (recall the ``Global Comparison Property'' discussed towards the end of Section \ref{s:level set PDE}), and it is this monotonicity that permits this notion of solution where one checks the PDE by ``testing'' against smooth functions that ``touch'' the subsolution/supersolution at a point. 

In what follows, we will say $\phi$ is a \emph{test function} if $\phi(x,t)$ is defined in some open set of $\mathbb{R}^m$ and the derivatives $\partial_t\phi$, $\nabla\phi(x,t)$, and $D^2\phi(x,t)$ are well defined and continuous in the domain of definition of $\phi$.

\begin{defn}\label{d:definition of viscosity subsolutions and supersolutions}
  An upper semi-continuous function $u:\mathbb{R}^m\times (-\infty,0]\to\mathbb{R}$ is called a viscosity subsolution of \eqref{e:appendix PDE} if for every test function $\phi(x,t)$ such that $u-\phi$ has a local maximum at some $(x_0,t_0)$, $t_0<0$, we have
  \begin{align*}
    \partial_t \phi(x_0,t_0) + \tfrac{1}{2}(H_n)^*(D^2\phi(x_0,t_0),\nabla\phi(x_0,t_0)) \geq 0.    
  \end{align*}
  A lower semi-continuous function $v:\mathbb{R}^m\times (-\infty,0]\to\mathbb{R}$ is called a viscosity supersolution of \eqref{e:appendix PDE} if for every smooth $\phi(x,t)$ such that $v-\phi$ has a local minimum at some $(x_0,t_0)$, $t_0<0$, we have 
  \begin{align*}
    \partial_t \phi(x_0,t_0) + \tfrac{1}{2}(H_n)_*(D^2\phi(x_0,t_0),\nabla\phi(x_0,t_0)) \leq 0.    
  \end{align*}
  Finally, a function is called a viscosity solution of \eqref{e:appendix PDE} if it is both a viscosity subsolution and viscosity supersolution. 

\end{defn}

We state the comparison principle for viscosity solutions of the equation. The proof follows largely along the lines for what is typically done in the literature, for instance the comparison principle in the first of the papers by Evans and Spruck, see \cite[Theorem 3.2]{EvansSpruck1991} and the more general comparison result in the work of Giga, Goto, Ishii, and Sato \cite{GigaGotoIshiiSato1991}. The theorem in \cite{EvansSpruck1991} provides a comparison principle for the mean curvature flow equation, which corresponds to $n\geq m-1$ here. Strictly speaking, the comparison principle in \cite{EvansSpruck1991} does not apply as stated to \eqref{e:appendix PDE} when $n<m-1$. Here we rely on the more general comparison result in \cite{GigaGotoIshiiSato1991} which deals with a general geometric flow and allows us for terminal data that might grow at infinity (a must have if we are to work with a Banach space norms as terminal data). We state a special case of the comparison principle in \cite{GigaGotoIshiiSato1991} and briefly discuss how the assumptions are satisfied in the case of \eqref{e:appendix PDE}.
 
\begin{theorem}\label{t:comparison principle} Suppose $u,v$ are respectively a viscosity subsolution and a viscosity supersolution of \eqref{e:appendix PDE} for some $T>0$, and suppose the following holds

1. There is a $K>0$ such that for all $x\in\mathbb{R}^m, t\in [-T,0]$ we have
\begin{align*}
  u(x,t) \leq K(|x|+1), v(x,t) \geq -K(|x|+1).    
\end{align*}

2. The functions $u(x,t),v(x,t)$ are continuous as $t\to 0$, that is, we have $u^*(x,0)=u_*(x,0)$ and $v^*(x,0)=v_*(x,0)$ for every $x \in \mathbb{R}^m$.

3. For some modulus of continuity $\rho$ we have
\begin{align*}
  u(x,0)-v(y,0) \leq \rho(|x-y|).    
\end{align*}
Then, we have
\begin{align*}
  u(x,t) \leq v(x,t) \;\forall\;x\in\mathbb{R}^m,\;t\in[-T,0].    
\end{align*}
    
\end{theorem}

\begin{proof}
  This is a special case of \cite[Theorem 3.2]{GigaGotoIshiiSato1991}, a result that deals with general unbounded domain $U$ (here we only need it for $U=\mathbb{R}^m$) and a more general PDE $F(D^2u,\nabla u)$ (in our case, $F = \tfrac{1}{2}H_n$). The only requirements made on $F$ in \cite[Theorem 3.2]{GigaGotoIshiiSato1991} are

  \begin{enumerate}
      \item $F(M,p)$ is continuous in the neighborhood of $(M,p)$ when $p\neq 0$ and is degenerate elliptic in the sense that $F(M,p) \leq F(M+N,p)$ whenever $N$ is positive semi-definite. 
      \item We must have $F_*(0,0)=F^*(0,0)$. 
     \item For every $R>0$, there must be some finite constant $C_R>0$ such that 
     \begin{align*}
       |F(M,p)| \leq C_R \text{ whenever } |p|\leq R, \|M\|_{\text{Fr}}\leq R.    
     \end{align*}
  \end{enumerate}
  Each of these properties are established $F=\tfrac{1}{2}H_n$ by the results at the end of Section \ref{s:Banaszczyk's theorem and the limiting Bellman operator}, specifically Lemma \ref{l:liminf and limsup for H delta n} and Remark \ref{r:half relaxed limits H n delta}.
\end{proof}

As it is standard, the comparison principle immediately implies the uniqueness of viscosity solutions. It is also possible to use it to prove existence of such a solution via the Perron method (see, for instance \cite[Section 4]{crandall1992user}). However, as solutions are built here from the continuum limit of the discrete game we will not discuss this method. 

The comparison principle in Theorem \ref{t:comparison principle} is a key component in the proof of Theorem \ref{t:continuum limit smooth E case}. These two theorems guarantee the existence and uniqueness of a viscosity solution to the terminal value problem \eqref{e:curvature flow initial data level set formulation}. We also prove this unique solution is self-similar and has convex sublevel sets in space for each fixed time. The first of these results is made possible by the uniqueness of the solution to \eqref{e:curvature flow initial data level set formulation} (which follows from Theorem \ref{t:comparison principle}) and the fact that the terminal data $x\mapsto \|x\|_E$ is a homogeneous function.

\begin{lemma}\label{l:solution is self similar}
  If $w(x,t)$ is the unique viscosity solution of \eqref{e:curvature flow initial data level set formulation}, then for every $x \in \mathbb{R}^m$ and $t<0$ we have
  \begin{align*}
    w(x,t) = \sqrt{-t}w(x/\sqrt{-t},-1).    
  \end{align*}
    
\end{lemma}

\begin{proof}
    From Theorem \ref{t:continuum limit smooth E case} we know that the function $w$, unique viscosity solution for \eqref{e:curvature flow initial data level set formulation}, is the limit of $w_\delta$ as $\delta\to 0$. Therefore this function satisfies the pointwise bounds
    \begin{align*}
      \|x\|_E \leq w(x,t) \leq \sqrt{\|x\|_E^2-Ct}.    
    \end{align*}    
    Now, for $\lambda>0$ define 
    \begin{align*}
      w^{(\lambda)}(x,t) = \lambda^{-1}w(\lambda x,\lambda^2 t).    
    \end{align*}
    It follows from a straightforward computation and $w$ being a viscosity solution of the PDE in \eqref{e:curvature flow initial data level set formulation} that $w^{(\lambda)}$ is a (viscosity) solution of the same PDE. On the other hand, from the homogeneity of the terminal data we have
    \begin{align*}
      w^{(\lambda)}(x,0) = \lambda^{-1} \|\lambda x\|_E = \|x\|_E = w(x,0).    
    \end{align*}
    That is, $w$ and $w^{(\lambda)}$ are two viscosity solutions of the same terminal value problem. Moreover, these two functions satisfy the following two sided bounds for every $x,t$,
    \begin{align*}
      \|x\|_E \leq w(x,t),\;w^{(\lambda)}(x,t)  \leq \sqrt{\|x\|^2_E-Ct}.
    \end{align*}
    In particular, 
    \begin{align*}
       |w(x,t)-w^{(\lambda)}(x,t)|\leq \sqrt{C|t|}.   
    \end{align*}
    This means that $w,w^{(\lambda)}$ satisfy the growth assumptions in Theorem \ref{t:comparison principle}. We conclude that 
    \begin{align*}
      w(x,t) = w^{(\lambda)}(x,t),\;\forall\;(x,t) \in \mathbb{R}^m\times \mathbb{R}_-,    
    \end{align*}
    and this holds for any $\lambda>0$. In particular, given $x,t$, choosing $\lambda = (\sqrt{-t})^{-1}$ we have
    \begin{align*}
      w(x,t) = \lambda^{-1}w(\lambda x,\lambda^2t) = \sqrt{-t}w(x/\sqrt{-t},-1),
    \end{align*}
    and the lemma is proved.
\end{proof}

To finish this appendix, we discuss the convexity of the level sets of $u$. For this, it will be convenient (as discussed in Section \ref{s:positive curvatures versus curvatures}) to consider the terminal value problem
\begin{align}\label{e:appendix terminal value standard curvature flow}
  \left \{ \begin{array}{rl} \partial_t \overline w + \tfrac{1}{2}\overline H_n(D^2\overline w,\nabla \overline w) & = 0 \text{ in } \mathbb{R}^m\times \mathbb{R}_-,\\
  \overline w(x,0) & = \|x\|_E \text{ in } \mathbb{R}^m.\end{array} \right.
\end{align}
The PDE of direct interest to us is \eqref{e:appendix PDE}, that of flow by the sum of principal positive curvatures. However, the PDE in \eqref{e:appendix terminal value standard curvature flow} which is flow by the sum of principal curvatures (positive or not) is more standard and vastly studied in the literature. For solutions whose level sets are convex the two equations agree, so one can drop the distinction of ``positive'' in the principal curvatures. For this reason it is an important fact that solutions to \eqref{e:appendix terminal value standard curvature flow} always have convex sublevel sets in space for each time. We record this as a theorem, which we invoke in the proof of the main theorem at the end of Section \ref{s:continuous limit}.

\begin{theorem}\label{t:convex level sets}

Consider $\overline w:\mathbb{R}^m\times (-T,0]\to\mathbb{R}$ a viscosity solution of \eqref{e:appendix terminal value standard curvature flow}. Then for every $t \in (-T,0]$ the sublevel sets of $\overline w(\cdot,t)$ are convex, in particular, $\overline w$ is also a viscosity solution of \eqref{e:appendix PDE}. 
    
\end{theorem}

We state this theorem without proof, and refer the reader to the relevant literature. First, we note that for $n\geq m-1$ (so, the case of the standard mean curvature flow) this result is contained in Evans Spruck \cite[Section 7.4]{EvansSpruck1991}, where they obtain the earlier landmark result of Huisken \cite{Huisken1984} but in the non-smooth viscosity solution setting (they do this by reducing the problem to one for an elliptic equation, from where the result for the level set function follows). For the other values of $n$ (so, $n\leq m-2$) the preservation of convexity result follows first by considering the work of Andrews \cite{andrews1994contraction} who considers more general functions of the principal curvatures in the smooth setting and shows the propagation of convexity, and then applying an approximation argument to pass his result to the flow by the sum of the largest $k$ principal curvatures. Last but not least, it is worth mentioning a result of Qing, Schikorra, and Zhou \cite{QingSchikorraZhou2016} for the case $m=2$ and $n=1,2$ notable in that they obtain a new proof of the preservation of convexity of sublevel sets for the mean curvature flow using not PDE but game theoretic tools, making use of the work of Kohn and Serfaty \cite{KohnSerfaty2006} (they also obtain analogous results for other flows using a game theoretic perspective).

\end{document}